\documentclass[12pt]{amsart}

\usepackage{amsthm}
\usepackage{thmtools}
\usepackage{thm-restate}

\usepackage[english]{babel}
\usepackage{pgf}
\usepackage{tikz}
\usetikzlibrary{shapes,shapes.geometric,arrows,fit,calc,positioning,arrows,automata,}
\usepackage[utf8]{inputenc} 
\usepackage[T1]{fontenc} 

\usepackage[margin = 1in]{geometry}

\usepackage{hyperref}

\usepackage{array} 
\usepackage{tabularx}
\usepackage{graphicx} 
\usepackage{amsmath}
\usepackage{amssymb}

\usepackage{scrhack}
\usepackage{enumerate}

\usepackage[capitalize]{cleveref}

\allowdisplaybreaks

\newcommand{\N}{\mathbb{N}}
\newcommand{\Z}{\mathbb{Z}}
\newcommand{\R}{\mathbb{R}}

\newcommand{\C}{\mathbb{C}}

\newcommand{\U}{\mathbb{U}}

\newcommand{\bfw}{\mathbf{w}}
\newcommand{\bfv}{\mathbf{v}}

\DeclareMathOperator{\e}{\mathrm{e}}

\DeclareMathOperator{\G}{\mathrm{G}}

\DeclareMathOperator{\ind}{\mathbf{1}}
\DeclareMathOperator{\sgn}{\mathrm{sgn}}
\DeclareMathOperator{\lcm}{\mathrm{lcm}}
\DeclareMathOperator{\tr}{\mathrm{tr}}
\DeclareMathOperator{\card}{\mathrm{card}}

\newcommand{\floor}[1]{\left\lfloor #1 \right\rfloor}

\newcommand{\abs}[1]{\left| #1 \right|}
\newcommand{\norm}[1]{\left\| #1 \right\|}

\newcommand{\rb}[1]{\left( #1 \right)}
\declaretheorem[numberwithin=section, name=Theorem]{theorem}
\declaretheorem[sibling=theorem,name=Lemma]{lemma}
\declaretheorem[sibling=theorem,name=Corollary]{corollary}
\declaretheorem[sibling=theorem,name=Proposition]{proposition}
\declaretheorem[sibling=theorem,name=Conjecture]{conjecture}
\declaretheorem[sibling=theorem,name=Definition]{definition}
\declaretheorem[numbered=no,name=Remark]{remark}
\declaretheorem[numbered=no,name=Example]{example}

\begin{document}
 \selectlanguage{english} 
\title{Automatic sequences fulfill the Sarnak conjecture}
\author{Clemens Müllner}
\thanks{This research is supported by the project F55-02 of the Austrian Science Fund FWF which is part of the Special Research Program ``Quasi-Monte Carlo Methods: Theory and Applications'',
by Project F5002-N15 (FWF), which is a part of the Special Research Program “Algorithmic and Enumerative Combinatorics” and by Project I1751 (FWF), called MUDERA
(Multiplicativity, Determinism, and Randomness).
Furthermore, the author wants to thank Michael Drmota for introducing this topic as well as for his help and guidance.}
\begin{abstract}
  We present in this paper a new method to deal with automatic sequences.
  This method allows us to prove a M\"obius-randomness-principle for automatic sequences from which we deduce the Sarnak conjecture for this class of sequences.
  Furthermore, we can show a Prime Number Theorem for automatic sequences 
  that are generated by strongly connected automata where the initial state is fixed by the transition corresponding to $0$.
\end{abstract}
\maketitle
 \section{Introduction}
In 2009, Peter Sarnak stated the following conjecture \cite{sarnak_three_lectures}:
\begin{conjecture}
  Let $\mu$ be the M\"obius function. For any bounded sequence $\xi(n)$ observed by a deterministic flow
  \footnote{A flow $F$ is a pair $(X,S)$ where $X$ is a compact metric space and $S:X\to X$ is a continuous map.
  For a dynamical system to be deterministic means that - roughly speaking - there are only few different orbits.}
  $(X,S)$, it holds that
  \begin{align*}
	  \sum_{n\leq N} \xi(n) \mu(n) = o(N).
  \end{align*}
\end{conjecture}

This conjecture has aroused great interest and numerous papers were recently devoted to show the Sarnak conjecture for different classes of flows  
\cite{walsh_bounds, rank_one_2013, horocycle, davenport, odometer, thue_morse_type, rank_one, interval, bounded_depth, green_tao, ChowlaSarnakErgodicView, quasi_discrete_spectrum, interval_map, rotations, distal, peckner_2015, horocycle_primes, morse_kakutani}.

Furthermore, we would like to mention two results connecting the Sarnak conjecture to the Chowla conjecture \cite{frantzikinakis, tao_log}.

Sarnak stated his conjecture in terms of a deterministic flow $(X,S)$.
We are interested in special flows that origin from a sequence:

Fix a sequence $\mathbf{a} = (a_n)_{n\in\N}$ that takes values in a finite set $\mathcal{A}$.
We denote by $S$ the shift operator on the sequences ($\mathcal{A}^{\N}$) in $\mathcal{A}$  and define $X:=\overline{\{S^n(\mathbf{a}):n\in\N_0\}}$.
By taking the product topology of the discrete topologies on each copy of $\mathcal{A}$, we find that $\mathcal{A}^{\N}$ is compact and complete.
The product topology is induced by the following metric $d(x,y) = \sum_{n=0}^{\infty} 2^{-n-1} d_n(x_n,y_n)$ on $X$, where $d_n$ denotes the discrete metric on $\mathcal{A}$.
The dynamical system (or flow) $(X,S)$ is called the \textit{symbolic dynamical system associated with $\mathbf{a}$}.

Therefore, we say that a sequence $\mathbf{a}$ fulfills the Sarnak conjecture if the symbolic dynamical system associated with $\mathbf{a}$ fulfills the Sarnak conjecture.
For more information about symbolic dynamical systems and their complexity see for example \cite{seq_complex, subst}.

A M\"obius-randomness principle is often closely related to a Prime Number Theorem (PNT), i.e.,~an asymptotic formula for the sum $\sum_{n\leq x} \Lambda(n) u_n$.
Some classes or examples of automatic sequences, which are observed by a deterministic flow, have already been covered:
\begin{itemize}
\item the Thue-Morse sequence:
  \begin{itemize}
    \item a M\"obius-randomness principle follows from the work of Indlekofer and K{\'a}tai~\cite{katai2001, katai1986}
    \item Dartyge and Tenenbaum~\cite{tenenbaum} give additionally an explicit error bound
    \item a M\"obius-randomness principle and a PNT by Mauduit and Rivat~\cite{mauduit_rivat_2010}
   \end{itemize}
\item the Rudin-Shapiro sequence:
  \begin{itemize}
  \item Tao suggests a strategy to prove a M\"obius-randomness principle in~\cite{tao_2012} 
  \item Mauduit and Rivat provide a different approach which allows to prove a M\"obius-randomness principle and a PNT for more general functions with an explicit error bound~\cite{mauduit_rivat_rs}
  \end{itemize}
\item sequences generated by invertible automata
  \begin{itemize}
    \item Sarnak conjecture \cite{fourauthors} 
    \item a M\"obius-randomness principle and a PNT~\cite{drmota2014}
  \end{itemize}
\item sequences generated by synchronizing automata
  \begin{itemize}
    \item the Sarnak conjecture and a PNT~\cite{synchronizing}
  \end{itemize}
\item some special digital functions (e.g.,~ the parity of occurrences of a word in the digital expansion)
  \begin{itemize}
    \item a M\"obius-randomness principle together with a PNT~\cite{hanna}
  \end{itemize}
\end{itemize}

The purpose of this paper is to show the Sarnak conjecture for all automatic sequences, which covers all results concerning automatic sequences mentioned above.
%

\begin{theorem}\label{th:moebius}
  Let $\mu$ be the M\"obius function, $(a_n)_{n\in\N}$ be a complex valued automatic sequence and
  let $(X,S)$ be the symbolic dynamical system associated with $(a_n)_{n\in\N}$.
  Then for all sequences $\xi(n) := f(S^n(x))$, with $x\in X$ and $f \in C(X,\C)$, we have
  \begin{align*}
    \sum_{n\leq N} \xi(n) \mu(n) = o(N).
  \end{align*}
\end{theorem}
\begin{remark}
  This theorem does not provide an explicit error term.
  This is mainly due to the fact that we consider the symbolic dynamical system associated with an automatic sequence.
  When considering only $\xi(n) = a_n$ we have basically two contributions for the error term:
  \begin{itemize}
   \item Estimates for the M\"obius function along arithmetic progressions.
   \item Estimates provided by \cref{thm:mobius} which gives a power saving.
  \end{itemize}
  However, for the general case we use \cref{le:dynamical_system_uniform} to relate $\xi(n)$ to other automatic sequences which provides no explicit error term.
\end{remark}

We present in this paper a new method which relates the output of a strongly connected deterministic finite automaton 
to the output of a newly introduced special class of transducers (so called \textit{naturally induced transducers}).
These transducers combine synchronizing and invertible properties and are a crucial step to prove \cref{th:moebius}.\\
Furthermore, the setup we develop can also be used to derive a Prime Number Theorem for a large class of automatic sequences -- which covers all mentioned results, except~\cite{hanna} which is only covered partially.

\begin{theorem}\label{th:prime}
  Let $A = (Q',\Sigma, \delta',q_0',\Delta,\tau)$ be a strongly connected deterministic finite automaton with output (DFAO) with $\Sigma = \{0,\ldots,k-1\}$ and $\delta'(q_0',0) = q_0'$
  \footnote{All necessary definitions are in \cref{subsec:def} and a detailed treatment of automatic sequences can be found for example in \cite{AllouchShallit}.}.
  Then the frequencies of the letters for the prime subsequence $(a_{p})_{p\in \mathcal{P}}$ exist.
\end{theorem}

\begin{remark}
  This theorem relies on an asymptotic formula for $\sum_{n\leq N} \Lambda(n) a_n$.
  This sum is modified and split into two types of sums.
  \begin{itemize}
   \item The first type corresponds to counting primes in arithmetic progressions and gives the main term.
   \item The second type gives only an error term and can be estimated by \cref{thm:prime}.
  \end{itemize}
\end{remark}

\begin{remark}
  The proof of \cref{th:prime} allows to determine these frequencies.\\
  All block-additive, i.e.,~digital, functions\footnote{For the definition of digital functions see \cite{AllouchShallit}.} are covered by \cref{th:prime}. 
  For the residue of any block-additive function $f \bmod m$ satisfying $(k-1,m) = 1$ and
	$(\gcd(f(n)_{n\in\N}),m) = 1$ one finds that all letters appear with the same frequencies along the primes.
\end{remark}

\begin{remark}
  The requirements in \cref{th:prime} for the DFAO $A$ are sufficient to ensure frequencies for the automatic sequence itself.
  On the one hand, one might relax these conditions without changing the theorem itself.
  On the other hand, we would like to mention that frequencies of the automatic sequence do not ensure frequencies for the prime subsequence.
\end{remark}
\begin{example}
We consider the following automaton and the corresponding automatic sequence $(a_n)_{n\in\N}$, which is properly defined in Definition~\ref{def:automatic_sequence}.
We feed here the automaton ``from left to right'' which means we start with the most significant digit.

  \begin{tikzpicture}[->,>=stealth',shorten >=1pt,auto,node distance=2.8cm, semithick, bend angle = 15, cross line/.style={preaction={draw=white,-,line width=4pt}}]
    
    \node[state, initial](A)                    {$a$};
    \node[state]         (B) [right of=A] 			{$b$};
    \node[state]         (C) [right of=B] 			{$c$};
    
    \path [every node/.style={font=\footnotesize}, pos = 0.66]
    (A) edge 		  node 		 {1,2}      (B)
	edge [loop above] node [pos=0.5] {0} 	(A)
    (B) edge [bend left]  node 		 {1}	(C)
	edge [loop above]  node 	[pos=0.5]	 {0,2}	(B)
    (C) edge [bend left] node 		 {1} 	(B)
	edge [loop above] node 		[pos=0.5] 	{0,2}	(C);
	\end{tikzpicture}
	
  Note that $a_n = b$ holds in exactly two cases:
  \begin{itemize}
   \item the sum of digits of $n$ in base $3$ is even and the first digit of $n$ in base $3$ is $2$
   \item the sum of digits of $n$ in base $3$ is odd and the first digit of $n$ in base $3$ is $1$.
  \end{itemize}
  We are able to reformulate this characterization as the sum of digits of $n$ in base $3$ is even if and only if $n$ is even:
  we have $a_n = b$ if and only if $n$ is even and the first digit of $n$ in base $3$ is $2$ or $n$ is odd and the first digit of $n$ in base $3$ is $1$.
  
  The corresponding automatic sequence is 
  \begin{align*}
      abbbcbbcbbcbcbcbcbbcbcbcbcbbcbcbcbcbcbcbcbcbcbcbcbcbcb
  \end{align*}
  One finds easily that the automatic sequence is equally distributed on $\{b,c\}$, as the generating automaton has only one final component which is primitive and symmetric.
  But, if we consider only odd integers we are detecting the first digit of $n$ in base $3$:
  the subsequence of odd integers of the automatic sequence is constant on all intervals of the form $[3^m, 2\cdot 3^m), [2 \cdot 3^m,3^{m+1})$.
  This shows that the densities of $b$ and $c$ do not exist for the subsequence of odd integers.
  This consideration together with the prime number theorem also shows that the densities of $b$ and $c$ do not exist along the primes.
\end{example}

In \cref{sec:induced_trans}, we properly define naturally induced transducers.
Furthermore, we describe a way to construct them and study some of their properties.
In \cref{sec:reduction}, we show how to use the concept of (naturally) induced transducers to reduce
\cref{th:moebius} and \cref{th:prime} to corresponding results on naturally induced transducers.
To actually show these results, we use some modified results of \cite{mauduit_rivat_rs}.
This is done in \cref{sec:RS1}.
Thereafter, we show that the output of a naturally induced transducer fulfills a carry lemma and has uniformly small Fourier-terms.
All the more technical proofs are postponed to the last \cref{sec:technical}.

\subsection{Definitions}\label{subsec:def}
In this paper we let $\N$ denote the set of positive integers, by $\U$ the set of complex numbers of modulus $1$ and we use
the abbreviation $\e(x) = $ exp$(2\pi i x)$ for any real number $x$.\\
For two functions, $f$ and $g$, that take only strictly positive real values such that $f/g$ is bounded, we write $f = O(g)$ or
$f\ll g$.\\
We let $\floor{x}$ denote the floor function and $\{x\}$ denote the fractional part of $x$.
Furthermore, we let $\chi_{\alpha}(x)$ denote the indicator function for $\{x\}$ in $[0,\alpha)$.\\
Moreover we let $\tau(n)$ denote the number of divisors of $n$, $\omega(n)$ denote the number of distinct prime factors of $n$, 
$\varphi(n)$ denote the number of integers smaller than $n$ that are co-prime to $n$ and $\mathcal{P}(a, m)$ denote the set of prime numbers $p \equiv a \bmod m$.

Next we give all definitions related to automata which can also be found in \cite{AllouchShallit}.
\begin{definition}
A deterministic finite automaton, or DFA, is a quadruple $A = (Q',\Sigma, \delta', q_0')$, where $Q'$ is a finite set of states, $\Sigma$ is the finite input alphabet,
$\delta': Q' \times \Sigma \rightarrow Q'$ is the transition function and $q_0' \in Q'$ is the initial state.
\end{definition}
We extend $\delta$ to a function $\delta: Q\times\Sigma^{*}\to Q$, in a natural way, by
\begin{align*}
  \delta(q,ab) = \delta(\delta(q,a),b)
\end{align*}
for all $a,b\in \Sigma^{*}$.
By definition $\delta(q,\bfw)$ consists of $\abs{\bfw}$ ``steps'' for every $\bfw\in \Sigma^{*}$.

We usually work with $\bfw \in \Sigma^{*} = \{0,\ldots,k-1\}^{*}$, where we let $\Sigma^{*}$ denote the set of all (finite) words over the alphabet $\Sigma$. 
We can view $\bfw$ as the digital representation of a natural number $n$ in base $k$.
We define for $\bfw = (w_0,\ldots,w_{r})$ the corresponding natural number $[\bfw]_k:= \sum_{i=0}^{r} k^{r-i} w_i$ (i.e.,~we assume that the first digit is the most significant).
We denote for a natural number $n$ the corresponding digit representation in base $k$ (without leading zeros) by $(n)_k$.
Furthermore, we let $(n)_k^t$ denote the unique word $\bfw$ of length $t$ such that  $[\bfw]_k \equiv n \bmod k^{t}$.
\begin{example}
  We find $(37)_2 = (1,0,0,1,0,1), (37)_2^4 = (0,1,0,1)$ and $[(1,0,1,1,0)]_2 = 22$.
\end{example}

As already mentioned, we always read words from left to right i.e.,~for digital representations of numbers we start with the most significant digit.
\begin{definition}
A DFAO $A = (Q',\Sigma,\delta',q'_0, \Delta, \tau)$ is a DFA with an additional output function $\tau: Q' \to \Delta$. 
\end{definition}

\begin{definition}\label{def:automatic_sequence}
  We say that a sequence $\mathbf{a} = (a_n)_{n\in \N}$ is a $k$-automatic sequence, if and only if there exists a DFAO $A = (Q, \Sigma = \{0,\ldots,k-1\}, \delta, q_0,\Delta, \tau)$ 
  such that $a_n = \tau(\delta(q_0,(n)_k))$.
\end{definition}

In this paper we restrict ourself to complex valued output functions; i.e.,~$\tau: Q'\to \C$ and we omit $\Delta = \C$ from now on.
Note that we can consider in \cref{th:prime} without loss of generality only complex valued automatic sequences.

We call an automaton (with output) $A = (Q',\Sigma, \delta', q_0')$ strongly connected if and only if for all $q'_1,q'_2 \in Q$ there exists $\bfw \in \Sigma^{*}$ with
$\delta'(q'_1,\bfw) = q'_2$.
We denote by a final component a strongly connected component that is closed under $\delta'$.

Furthermore, we let $p(A)$ denote the period of an automaton $A$, i.e., $p(A) = \gcd\{n: \exists q' \in Q', \bfw \in \Sigma^{n} \text{ such that } \delta'(q',\bfw) = q'\}$
and call an automaton with period $1$ primitive.

We define for $M' \subseteq Q'$ and $\bfw \in \Sigma^{*}$
\begin{align*}
  \delta'(M',\bfw) := \{\delta'(q',\bfw)|q' \in M'\}
\end{align*}
and for a tuple $(q_1',\ldots,q_r') \in (Q')^{r}$
\begin{align*}
  \delta'((q_1',\ldots,q_r'), \bfw) = (\delta'(q_1',\bfw),\ldots,\delta'(q_r',\bfw)).
\end{align*}
We let $\pi_{\ell}$ denote the projection on the $\ell$-th coordinate (i.e.,~$\pi_{\ell}(q_1',\ldots,q_r') = q_{\ell}'$).

\begin{definition}
A finite-state transducer is a $6$-tupel $\mathcal{T} = (Q, \Sigma, \delta, q_0, \Delta, \lambda)$, where $Q$ is a finite set of states, $\Sigma$ is the input alphabet, 
$\delta$ is the transition function,
$q_0$ is the initial state, $\Delta$ is the output alphabet and $\lambda:Q\times \Sigma \rightarrow \Delta^{*}$ is the output function. 
\end{definition}
We will restrict ourself to $\lambda:Q \times \Sigma \rightarrow \Delta$.
Throughout this work we assume that $(\Delta,\circ)$ is a group generated by $im(\lambda)$ and the product in the definition of $T$ is a product according to $\circ$.\\
A transducer can be viewed as a mean to define functions: on input $\bfw = w_1w_2\ldots w_r \in \Sigma^{r}$ the transducer enters states 
$q_0 = \delta(q_0, \varepsilon), \delta(q_0,w_1),\ldots,\delta(q_0,w_1w_2\ldots w_r)$ and produces the outputs
\begin{align*}
  \lambda(q_0,w_1),\lambda(\delta(q_0,w_1),w_2),\ldots,\lambda(\delta(q_0,w_1w_2\ldots w_{r-1}),w_r).
\end{align*}
The function $T(\bfw)$ is then defined as
\begin{align*}
  T(\bfw) := \prod_{j = 0}^{r-1} \lambda(\delta(q_0,w_1w_2\ldots w_{j}),w_{j+1}).
\end{align*}
We also define the slightly more general form,
\begin{align*}
  T(q,\bfw) := \prod_{j = 0}^{r-1} \lambda(\delta(q,w_1w_2\ldots w_{j}),w_{j+1}).
\end{align*}
We define for a set $M \subseteq \Delta$ the inverse set and the multiplication of two sets $M_1,M_2 \subseteq \Delta$ as usual,
\begin{align*}
  M^{-1} &:= \{g^{-1}| g \in M\}\\
  M_1 \cdot M_2 &:= \{g_1 \circ g_2| g_1 \in M_1, g_2 \in M_2\}.
\end{align*}
We say a transducer (or analogously a DFA) is \textbf{synchronized} if and only if 
\begin{align*}
  \exists q \in Q, \bfw_q \in \Sigma^{*}\quad \forall q_1 \in Q: \delta(q_1,\bfw_q) = q.
\end{align*}
We call $\bfw_q$ a \textbf{synchronizing word}.\\
We define for $\sigma \in S_n$ and a $n$-tuple $\mathbf{x} = (x_1,\ldots,x_n)$
\begin{align*}
  \sigma \cdot \mathbf{x} := (x_{\sigma^{-1}(1)},\ldots,x_{\sigma^{-1}(n)}),
\end{align*}
which is the natural way to define the multiplication of a permutation from the left as this ensures associativity: 
$(\sigma_1 \circ \sigma_2) \cdot \mathbf{x} = \sigma_1 \cdot(\sigma_2 \cdot \mathbf{x})$ holds for all $\sigma_1,\sigma_2 \in S_n$.

Let $A$ be a $n\times n$ matrix with complex entries.
We let $A^{H}$ denote the conjugate transpose of $A$. 
Furthermore, the Frobenius norm of a matrix is given by
\begin{align*}
  \norm{A}_{F}:= \sqrt{\sum_{i=1}^{n}\sum_{j=1}^{n} \abs{a_{ij}}^2}.
\end{align*}

\section{Induced Transducer}\label{sec:induced_trans}
The goal of this section is to define for every strongly connected Automaton $A = (Q',\Sigma,\delta',q'_0)$ a "naturally induced transducer" 
$\mathcal{T}_A = (Q,\Sigma,\delta,q_0,\Delta,\lambda)$ such that we can reconstruct $\delta'(q'_0,(n)_k)$ 
by knowing $T(q_0,(n)_k)$ and $\delta(q_0,(n)_k)$ for all $n$.

The main advantage of this step is that it reveals the actual structure of $A$. On a more technical side, we can treat
$\delta$ and $T$ almost independently which simplifies the analysis of the automatic sequence.

\subsection{Automaton to Transducer}\label{sec:aut_to_trans}
In this subsection, we develop a special way to construct for an automaton $A = (Q',\Sigma, \delta', q_0')$ a "naturally induced" transducer $\mathcal{T}_{A}=(Q,\Sigma,\delta,q_0,\Delta,\lambda)$ which contains all the information of $A$.
The main idea is to merge some states of $A$ into a single state of $\mathcal{T}_{A}$ consistent with $\delta$ and $\delta'$.
To illustrate this concept, we start with a simple motivating example:
\begin{example}
	A classical block-additive automatic sequence is the Rudin-Shapiro sequence.
	The automaton generating this function looks as follows.
	
	
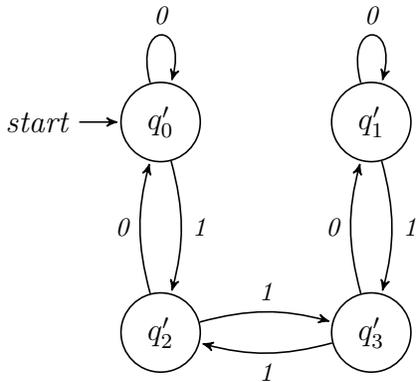
\begin{tikzpicture}[->,>=stealth',shorten >=1pt,auto,node distance=2.8cm, semithick, bend angle = 15, cross line/.style={preaction={draw=white,-,line width=4pt}}]
    
    \node[state, initial](A)                    {$q_0'$};
    \node[state]         (B) [right of=A] 			{$q_1'$};
    \node[state]         (C) [below of=A] 			{$q_2'$};
    \node[state]         (D) [right of=C] 			{$q_3'$};
    
    \path [every node/.style={font=\footnotesize}, pos = 0.5]
    (A) edge [bend left]  node 		 {1}      (C)
	edge [loop above] node [pos=0.5] {0} 	(A)
    (B) edge [bend left]  node 		 {1}	(D)
	edge [loop above]  node 	[pos=0.5]	 {0}	(B)
    (C) edge [bend left] node 		 {0} 	(A)
	edge [bend left] node 					 {1}	(D)
    (D) edge [bend left]  node 		 {0}	(B)
	edge [bend left] node 					 {1}	(C);
	\end{tikzpicture}
	
	One finds easily
	\begin{align*}
		\delta'(\{q'_0,q'_1\},0) = \{q'_0,q'_1\} \qquad \delta'(\{q'_0,q'_1\},1) = \{q'_2,q'_3\}\\
		\delta'(\{q'_2,q'_3\},0) = \{q'_0,q'_1\} \qquad \delta'(\{q'_2,q'_3\},1) = \{q'_3,q'_2\}.
	\end{align*}
	Consequently, we want to merge the pair $(q'_0,q'_1)$ into a single state $q_0$ (i.e.,~$q_0 = (q'_0,q'_1)$) and
	$q_1 = (q'_2,q'_3)$. We find
	\begin{align*}
		\delta'(q_0,0) = q_0 \qquad &\delta'(q_0,1) = q_1\\
		\delta'(q_1,0) = q_0 \qquad &\delta'(q_1,1) = (12) \cdot q_1.
	\end{align*}
	We define ($\Delta := S_2$)
	\begin{align*}
		\delta(q_0,0) = q_0 \qquad &\delta(q_0,1) = q_1\\
		\delta(q_1,0) = q_0 \qquad &\delta(q_1,1) = q_1\\
		\lambda(q_0,0) = id \qquad &\lambda(q_0,1) = id\\
		\lambda(q_1,0) = id \qquad &\lambda(q_1,1) = (12)
	\end{align*}
	and find that $\delta'(q_i,\varepsilon) = \lambda(q_i,\varepsilon) \cdot \delta(q_i,\varepsilon)$ holds.
	This is already the first example of a naturally induced transducer and an illustration of it is given below.
	
	 \tikzset{elliptic state/.style={draw,ellipse}}
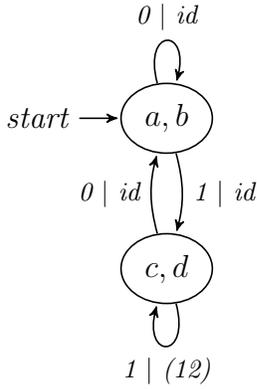
	\begin{tikzpicture}[->,>=stealth',shorten >=1pt,auto,node distance=2.0cm, semithick, bend angle = 15, cross line/.style={preaction={draw=white,-,line width=4pt}}]
    
    \node[elliptic state, initial]	(A)                    {$a,b$};
    \node[elliptic state]	         (B) [below of=A] 	{$c,d$};
    
    \path [every node/.style={font=\footnotesize}, pos = 0.5]
    (A) edge [loop above] node 		 	{0 $\mid$ id}   (A)  	
		(A)	edge [bend left]  node 		 	{1 $\mid$ id} 	(B)
    (B) edge [bend left]  node 			{0 $\mid$ id}		(A)
		(B) edge [loop below] node			{1 $\mid$ (12)} 	(B);
  \end{tikzpicture}
\end{example}

To formalize this idea we first define induced transducers.
\begin{definition}
  We call a transducer $\mathcal{T}_{A}$ \textbf{induced} by a DFA $A = (Q',\Sigma, \delta', q_0')$ if and only if $\mathcal{T}_{A} = (Q, \Sigma, \delta, q_0, \Delta, \lambda)$ where
  \begin{enumerate}[1)]
   \item $\exists n_0 \in \N: Q\subseteq (Q')^{n_0}$ 
   \item $\pi_1(q_0)=q_0'$
   \item $\Delta < S_{n_0}$ is the subgroup generated by $Im(\lambda)$
   \item $\forall q \in Q, a \in \Sigma: \delta'(q,a) = \lambda(q,a) \cdot \delta(q,a)$.
  \end{enumerate}
  We call $\mathcal{T}_{A}$ a \textbf{naturally induced transducer} if furthermore
  \begin{enumerate}[1)]
   \setcounter{enumi}{4}
   \item $\forall i \neq j \leq n_0, q \in Q: \pi_i(q) \neq \pi_j(q)$,
   \item $\forall q_1 \neq q_2 \in Q \not \exists \hspace{1mm} \sigma \in S_{n_0}: q_1 = \sigma \cdot q_2$,
   \item $\mathcal{T}_{A}$ is strongly connected
   \item $\mathcal{T}_{A}$ is synchronizing
  \end{enumerate}
  holds.
\end{definition}

\begin{remark}
  Properties 5)-7) only assure that $\mathcal{T}_{A}$ is chosen minimal to certain aspects.
\end{remark}

\begin{example}
  We call $\mathcal{T}_{A}$ the \textbf{trivial} transducer induced by $A$ when:\\
  $n_0 = 1, Q = Q', \Delta = \{id\}, \forall q \in Q, a \in \Sigma: \delta(q,a) = \delta'(q,a), \lambda(q,a) = id $.\\
  The trivial induced transducer is a naturally induced transducer if and only if $A$ is synchronizing and strongly connected.
\end{example}

\begin{proposition}\label{prop:synch_trans}
  For every strongly connected automaton $A$, there exists a naturally induced transducer $\mathcal{T}_{A} = (Q,\Sigma, \delta, q_0, \Delta,\lambda)$.
  All other naturally induced transducers can be obtained by changing the order on the elements of $Q$ and possibly changing the initial state $q_0$.
\end{proposition}
\begin{proof}
The proof is not constructive and we start by defining $n_0(A):= \min\{|\delta'(Q',\bfw)|: \bfw \in \Sigma^{*}\}$ and
\begin{align*}
  S(A) := \{M\subseteq Q' : |M| = n_0(A), \exists \bfw \in \Sigma^{*} \text{ with } \delta'(Q',\bfw) = M\},
\end{align*}
which are exactly the sets of size $n_0(A)$ that are reachable from $Q'$ using $\delta'$.
One easily observes that an automaton is synchronizing if and only if $n_0(A) = 1$ 
and invertible\footnote{For the definition of invertible automata see~\cite{invertible}.} if and only if $n_0(A) = |Q'|$.

Now we are prepared to construct a naturally induced transducer for $A$, similar to the previous example.\\
We choose $n_0 = n_0(A)$.
Furthermore, we can choose any ordering $\leq_{Q'}$ of $Q'$ such that $q_0'$ is the minimal element.
For every $M \in S(A)$ we define a corresponding $n_0$-tuple $q_M$ which consists of the elements of $M$ ordered by $\leq_{Q'}$.
We define $Q:= \{q_M: M\in S(A)\}$.

Take an arbitrary $q_M \in Q$.
By minimality of $\abs{M}$ for $a \in \Sigma$ we have $\delta'(M,a) \in S(A)$. Thus we define $\delta(q_M,a):= q_{\delta'(M,a)}$.

We note that $Q'$ is covered by $S(A)$: 
Take $q'_1 \in Q'$ and $M \in S(A)$ with some $q'_2 \in M$. As $A$ is strongly connected, there exists $\bfw$ such that $\delta'(q'_2,\bfw) = q'_1$.
Thus, we find $q'_1 \in \delta'(M,\bfw) \in S(A)$.

We define $q_0$ as an arbitrary element $q_M \in Q$ 
such that $q_0'\in M$. (Note that $\pi_1(q_M) = q_0'$ since $q'_0$ is the minimal element).

We define $\lambda': (Q')^{n_0} \rightarrow S_{n_0}$ such that for all $\tilde{q} = (q'_1,\ldots,q'_{n_0}) \in (Q')^{n_0}$, $\lambda'(\tilde{q})$ is the permutation 
such that $\lambda'(\tilde{q})\cdot \tilde{q}$ is ordered with respect to $\leq_{Q'}$
and $\lambda: Q \times \Sigma \rightarrow S_{n_0}, \lambda(q_M,a) = \lambda'(\delta'(q_M,a))^{-1}$ for all $M \in S(A)$ and $a \in \Sigma$.

Let $\Delta$ be the subgroup of $S_{n_0}$ generated by $Im(\lambda)$.
This defines an induced transducer $\mathcal{T}_{A} = (Q,\Sigma, \delta, q_0, \Delta, \lambda)$.

It remains to show that this induced transducer is a naturally induced transducer.
5) and 6) follow directly from the construction of the transducer.
Every $q_M \in Q$ corresponds to $M \in S(A)$ for which there exists $\bfw_M$ such that $\delta'(Q',\bfw_M) = M$.
Thus it follows directly that for all $q_{M'} \in Q$ it holds $\delta(q_{M'},\bfw_M) = q_M$ which shows that the constructed transducer is synchronizing and also strongly connected. \\
To prove the second part of this proposition we assume that $\overline{\mathcal{T}_{A}} = (\overline{Q},\Sigma,\overline{\delta},\overline{q_0},\overline{\Delta},\overline{\lambda})$ is an arbitrary naturally induced transducer.
We start by showing $n_0(A) = \overline{n_0}$. 
\begin{itemize}
 \item Assume $\overline{n_0} > n_0(A)$:
    By the definition of $n_0(A)$ there exists  some $\bfw\in \Sigma^{*}$ such that $|\delta'(Q',\bfw)| = n_0(A)$.
    Consequently $\overline{\delta}(\overline{q},\bfw)\in (Q')^{\overline{n_0}}$ violates 5) for every $\overline{q}\in \overline{Q}$.
 \item Assume $\overline{n_0}<n_0(A)$: As $\overline{\mathcal{T}_{A}}$ is synchronizing there exists $\bfw \in \Sigma^{*}$ and $\overline{q_1}$ 
      such that for all $\overline{q} \in \overline{Q}$ we have
    $\overline{\delta}(\overline{q},\bfw) = \overline{q_1}$. 
    Since $\overline{Q}$ covers $Q'$ this implies that $|\delta'(Q',\bfw)|\leq \overline{n_0} < n_0(A)$ which contradicts the definition of $n_0(A)$.
\end{itemize}
  Assume now that there exists some $q \in Q$ such that no element of $\overline{Q}$ is a permutation of $q$.
  The elements of $q$ form a set $M_q\in S(A)$ and thus there exists $\bfw_{q}$ such that $\delta'(Q',\bfw_q) = M_q$.
  $\overline{\delta}(\overline{q},\bfw_q)$ contains exactly the elements of $M_q$ which means that it is a permutation of $q$ (by property 5)).\\
  The elements of $\overline{Q}$ that are permutations of elements of $Q$ form a strongly connected component. 
  Thus, the proposition follows by property 7).
\end{proof}



\begin{remark}
	Let $\mathcal{T} = (Q,\Sigma,\delta,q_0,\Delta,\lambda)$ be a transducer fulfilling 1)-5) and 7)-8).
	We define an equivalence relation $\sim$ on $Q$ by: $q_1 \sim q_2$ if and only if $\exists \sigma$: $q_1 = \sigma \cdot q_2$.
	Then $\mathcal{T}/\sim$ is a naturally induced transducer. 
\end{remark}

\begin{remark}
  One can easily show the minimal length of a synchronizing word for a naturally induced transducer $\mathcal{T}_{A}$ is bounded by $O(\abs{Q'}^{3})$.
  This already allows to determine $n_0$, although in a rather inefficient way, by considering all paths of this length.
  A more efficient approach would be to determine all sets that can be obtained by $\delta'(Q',\bfw)$. This can be done for example by a depth-first search.
\end{remark}

\begin{example}
  We consider the following more complex automaton with $\Sigma = \{0,1\}$ and find a naturally induced transducer.
  
  
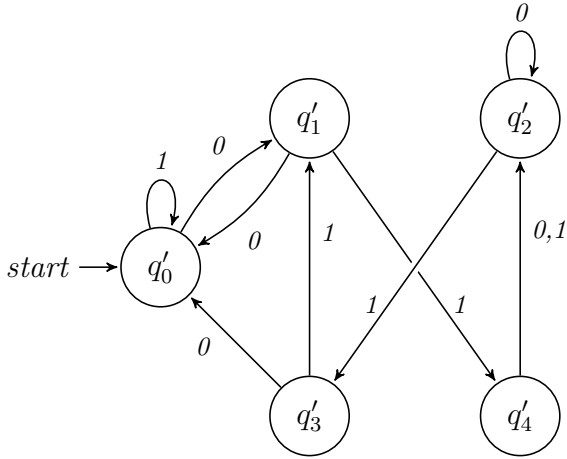
\begin{tikzpicture}[->,>=stealth',shorten >=1pt,auto,node distance=2.8cm, semithick, bend angle = 15, cross line/.style={preaction={draw=white,-,line width=4pt}}]
    
    \node[state, initial](A)                    {$q_0'$};
    \node[state]         (B) [above right of=A] {$q_1'$};
    \node[state]         (C) [right of=B] 	{$q_2'$};
    \node[state]         (D) [below right of=A] {$q_3'$};
    \node[state]         (E) [right of=D]       {$q_4'$};
    
    \path [every node/.style={font=\footnotesize}, pos = 0.66]
    (A) edge [bend left]  node 		 {0}      (B)
	edge [loop above] node [pos=0.5] {1} 	(A)
    (B) edge [bend left]  node 		 {0}	(A)
	edge              node 	[right]	 {1}	(E)
    (C) edge [loop above] node [pos=0.5] {0} 	(C)
	edge [cross line] node 	[left]	 {1}	(D)
    (D) edge              node 		 {0}	(A)
	edge              node 	[right]	 {1}	(B)
    (E) edge              node [right]	 {0,1}	(C);
    
  \end{tikzpicture}
  
  One finds that $n_0(A) = 3$ and $S(A) = \{M_1, M_2\}$ with
  \begin{align*}
    M_1 &:= \delta'(Q',0) = \{q_0',q_1',q_2'\}\\
    M_2 &:= \delta'(Q',01) = \{q_0',q_3',q_4'\}.
  \end{align*}
  We construct now a naturally induced transducer $\mathcal{T}_{A}$ and start by defining $Q = \{q_0,q_1\}$ where
  \begin{align*}
    q_0 &= (q_0',q_1',q_2')\\
    q_1 &= (q_0',q_3',q_4').
  \end{align*}
  
  We find
  \begin{align*}
    \delta'(q_0, 0) = (q_1',q_0',q_2'), \quad
    \delta'(q_0, 1) = (q_0',q_4',q_3')\\
    \delta'(q_1, 0) = (q_1',q_0',q_2'), \quad
    \delta'(q_1, 1) = (q_0',q_1',q_2')
  \end{align*}
  and therefore
  \begin{align*}
    \delta(q_0,0) &= q_0, \quad \delta(q_0,1) = q_1\\
    \delta(q_1,0) &= q_0, \quad \delta(q_1,1) = q_0\\
    \lambda(q_0,0) &= (12), \quad
    \lambda(q_0,1) = (23)\\
    \lambda(q_1,0) &= (12), \quad
    \lambda(q_1,1) = id.
  \end{align*}
  Thus we constructed a transducer $\mathcal{T}_{A}$ (as described in the proof of \cref{prop:synch_trans}):\\
  
  \tikzset{elliptic state/.style={draw,ellipse}}
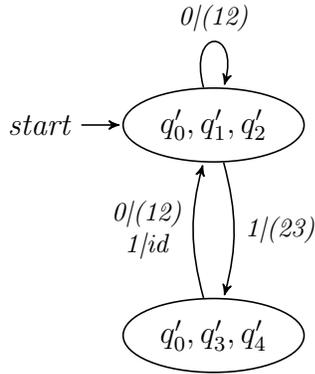
  \begin{tikzpicture}[->,>=stealth',shorten >=1pt,auto,node distance=2.8cm, semithick, bend angle = 15, cross line/.style={preaction={draw=white,-,line width=4pt}}]
    
    \node[elliptic state, initial]	(A)                    {$q_0', q_1', q_2'$};
    \node[elliptic state]	         (B) [below of=A] 	{$q_0', q_3', q_4'$};
    
    \path [every node/.style={font=\footnotesize}, pos = 0.5]
    (A) edge [loop above] node 		 	{0|(12)}      	(A)
	edge [bend left]  node 		 	{1|(23)} 	(B)
    (B) edge [bend left]  node [align=center]	{0|(12)\\1|id}	(A);
  \end{tikzpicture}
  
  One easily observes that this is a naturally induced transducer.
  \end{example}
  \begin{remark}
  It is sometimes useful to consider the automaton where each state occurs as often as it occurs in the subsets of $S(A)$ and "group" them according to $S(A)$.
  \end{remark}
  \begin{example}
  For the example above this gives the following automaton.\\
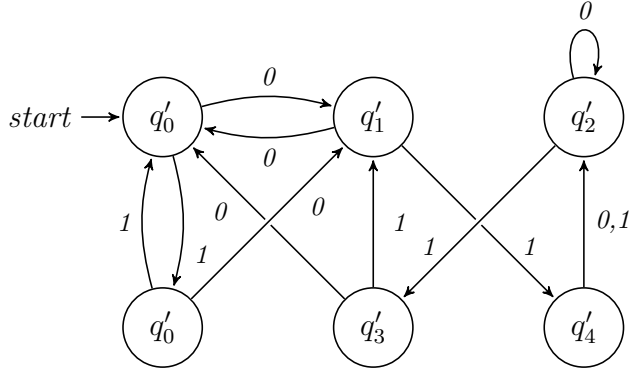
  \begin{tikzpicture}[->,>=stealth',shorten >=1pt,auto,node distance=2.8cm, semithick, bend angle = 15, cross line/.style={preaction={draw=white,-,line width=3pt}}]
    
    \node[state, initial](A)                    {$q_0'$};
    \node[state]         (B) [right of=A]	{$q_1'$};
    \node[state]         (C) [right of=B] 	{$q_2'$};
    \node[state]	 (A')[below of=A]	{$q_0'$};
    \node[state]         (D) [right of=A']	{$q_3'$};
    \node[state]         (E) [right of=D]       {$q_4'$};
    
    \path [every node/.style={font=\footnotesize}, pos = 0.66]
    (A) edge [bend left]  node 	[pos=0.5]	 {0}    (B)
	edge [bend left]  node 	[pos=0.6]	 {1} 	(A')
    (B) edge [bend left]  node 	[pos=0.5]	 {0}	(A)
	edge              node 	[right,xshift = 1mm]	 {1}	(E)
    (C) edge [loop above] node [pos=0.5] {0} 	(C)
	edge [cross line] node 	[xshift=-1mm,left]	 {1}	(D)
    (D) edge              node 	[pos = 0.7]	 {0}	(A)
	edge              node 	[pos = 0.5,right, xshift=1mm]	 {1}	(B)
    (E) edge              node [pos = 0.5,right]	 {0,1}	(C)
    (A')edge [cross line] node 	[pos = 0.6, right, xshift=2mm]	 {0}	(B)
	edge [bend left]  node 	[pos=0.5]	 {1}	(A);
    
  \end{tikzpicture}
  
\end{example}


We end this subsection with a technical lemma used later.
\begin{lemma}\label{le:transducer_regular}
  Let $\mathcal{T}_{A}$ be a naturally induced transducer.
  Then there exists $N_0 \in \N$ such that for all $q_1,q_2 \in Q$ and $n \geq N_0$ there exists $\bfw \in \Sigma^{n}$ such that $\delta(q_1,\bfw) = q_2$.
\end{lemma}
\begin{proof}
  Let $\bfw_q$ be a synchronizing word. We denote by 
  $dist(x,y) := \min\{n|\exists \bfw \in \Sigma^{n}: \delta(x,\bfw) = y\}$
  and define $N_0 := |\bfw_q| + \max_{q' \in Q} dist(q, q')$.\\
  Take $q_1,q_2 \in Q$ and $n \geq N_0$. There exists $\bfw_1$ from $q$ to $q_2$ of length $dist(q,q_2)$. 
  We consider $\bfw = 0^{n-|\bfw_q| - dist(q,q_2)} \bfw_q \bfw_1$ and see that this is a path from $q_1$ to $q_2$ of length $n$ 
  ($0^x$ denotes the word consisting of $x$ consecutive zeros).
\end{proof}

\subsection{Connection between an automaton and its induced transducer}

Property 4) of an induced transducer relates one step of $A$ to one step of $\mathcal{T}_{A}$ and $\lambda / T$. In this section we show how
arbitrary many steps of $A$ relate to $\mathcal{T}_{A}$.
Therefore we start with a technical lemma that describes how the action of a permutation to a state $q \in Q$ influences the behavior of the induced transducer.
\begin{lemma}\label{le:perm-lambda}
  Let $A$ be a DFA and $\mathcal{T}_{A}$ an induced transducer. Furthermore, let $\sigma \in S_{n_0}, a \in \Sigma$ and $q = (x_1,\ldots,x_{n_0}) \in Q$. It holds
  \begin{align*}
    \delta'(\sigma \cdot q,a) = (\sigma \circ \lambda(q,a)) \cdot \delta(q,a).
  \end{align*}
\end{lemma}
\begin{proof}
  We find for $1\leq k \leq n_0$
  \begin{align*}
    \pi_k(\delta'(\sigma \cdot (x_1,\ldots,x_{n_0}),a)) &= \delta'(x_{\sigma^{-1}(k)},a) = \pi_{\sigma^{-1}(k)}(\delta'(q,a))\\
	&= \pi_{\sigma^{-1}(k)}(\lambda(q,a) \cdot \delta(q,a)) = \pi_k(\sigma \cdot (\lambda(q,a) \cdot \delta(q,a))).
  \end{align*}
  Thus, it holds
  \begin{align*}
    \delta'(\sigma \cdot q,a) = \sigma \cdot (\lambda(q,a) \cdot \delta(q,a)) = (\sigma \circ \lambda(q,a)) \cdot \delta(q,a).
  \end{align*}
\end{proof}

Now we are ready to show a very important connection between an automaton and its induced transducer.

\begin{proposition}\label{pr:equality}
Let $A$ be a strongly connected automaton and $\mathcal{T}_{A}$ an induced transducer.
For every $\bfw \in \Sigma^{*}$ we have
  \begin{align*}
    \delta'(q_0',\bfw) = \pi_1(T(q_0,\bfw)\cdot \delta(q_0,\bfw)).
  \end{align*}
\end{proposition}
\begin{proof}
  We actually prove that for all $q \in Q$ and $\bfw \in \Sigma^{*}$ it holds that
  \begin{align*}
    \delta'(q,\bfw) = T(q,\bfw) \cdot \delta(q,\bfw),
  \end{align*}
  which obviously implies the statement.
  We use induction on the length of $\bfw$.
  \begin{itemize}
    \item $|\bfw| = 0$: Obviously $\delta'(q,\varepsilon) = q = (id(\delta(q,\varepsilon)))$ holds.
    \item $\bfw = \bfw'x$: We define $q_1:=\delta(q,\bfw')$ and find 
      \begin{align*}
	      \delta'(q,\bfw'x) &= \delta'(\delta'(q,\bfw'),x) = \delta'(T(q,\bfw')\cdot q_1,x) = (T(q,\bfw') \circ \lambda(q_1,x)) \cdot \delta(q_1,x)\\
		&= T(q,\bfw'x) \cdot \delta(q,\bfw'x)),
      \end{align*}
      where the third equality holds by \cref{le:perm-lambda}.
  \end{itemize}
  \hspace{-5mm}
\end{proof}
This Proposition allows us to reconstruct the output of $A$ by knowing the output of $\mathcal{T}_{A}$ and $T$.

\begin{example}
  We continue our previous example and find
  \begin{align*}
    \delta'(q_0',0110) &= q_2' = \pi_1 (T(q_0,0110) \cdot \delta(q_0,0110)) = \pi_1(((12)\circ (23)\circ id \circ(12)) \cdot q_0)\\
    &= \pi_1 ((13) \cdot q_0) = \pi_1(q_2',q_1',q_0') = q_2'.
  \end{align*}

\end{example}

As one has some freedom on how to choose the order of the elements of $Q$ it is natural to ask how this choice influences the induced transducer.

\begin{proposition}\label{pr:T_permute}
  Let $A$ be a DFA and $\mathcal{T}_{A} = (Q,\Sigma, \delta, q_0, \Delta, \lambda)$ an induced transducer.
  Let $\overline{\mathcal{T}}_{A}$ be another induced transducer obtained by changing the order of every tuple $q \in Q$ 
  by a permutation $\sigma_q$, i.e.,~$\overline{q} = \sigma_q \cdot q$ (where still $\pi_1(\sigma_{q_0}\cdot q_0) = q_0'$ holds).
  For $\bfw \in \Sigma^{*}$ it holds that
  \begin{align*}
    \overline{T}(\overline{q},\bfw) = \sigma_{q} \circ T(q,a) \circ \sigma_{\delta(q,\bfw)}^{-1}.
  \end{align*}
  Furthermore, if $\mathcal{T}_{A}$ is a naturally induced transducer, then so is $\overline{\mathcal{T}}_{A}$.
\end{proposition}
\begin{proof}
  Let $q \in Q, a \in \Sigma$. We define for all $q\in Q, a \in \Sigma$
  \begin{align*}
    \overline{q} &:= \sigma_{q} \cdot q\\
    \overline{\delta}(\overline{q},a) &:= \overline{\delta(q,a)}\\
    \overline{\lambda}(\overline{q},a) &:= \sigma_{q} \circ \lambda(q,a) \circ \sigma_{\delta(q,a)}^{-1},
  \end{align*}
  $\overline{Q} := \cup_{q \in Q} \overline{q}$ and $\overline{\Delta}$ is again the group generated by $\overline{\lambda}$.\\
  We claim that $\overline{\mathcal{T}}_{A} = (\overline{Q}, \Sigma, \overline{\delta}, \overline{q_0}, \overline{\Delta}, \overline{\lambda})$ is again an induced transducer.
  1)-3) follow immediately so it remains to show 4).
  Therefore, we compute by \cref{le:perm-lambda}
  \begin{align*}
    \delta'(\sigma_q \cdot q,a) &= (\sigma_q \circ \lambda(q,a)) \cdot \delta(q,a)\\
      &= (\sigma_q \circ \lambda(q,a) \circ \sigma_{\delta(q,a)}^{-1})\cdot(\sigma_{\delta(q,a)} \cdot \delta(q,a))\\
      &= \overline{\lambda}(\overline{q},a)\cdot \overline{\delta}(\overline{q},a).
  \end{align*}
  Thus $\overline{\mathcal{T}}_{A}$ is indeed an induced transducer and the stated equation follows easily by induction on the length of $\bfw$.
  
  The last statement is very easy to verify.
\end{proof}

\subsection{Length restrictions for naturally induced transducers}\label{sec:str1}
As $\Delta$ is generated by $\lambda$, one may assume that all elements of $\Delta$ can be obtained by $T$.
We show in this subsection that this is not true in general but for certain naturally induced transducers.
Furthermore, we will show that restrictions on the length of $\bfw$ give restrictions on what elements of $\Delta$ can be obtained by $T(q,\bfw)$.\\
The main result of this subsection -- which contains all the facts mentioned above -- is the following theorem.
\begin{theorem}\label{th:full_G}
  Let $A$ be a strongly connected automaton. Then there exists a minimal $d\in \N$, $m_0 \in \N$,
  a naturally induced transducer $\mathcal{T}_{A}$
  and a subgroup $G$ of $\Delta$ such that the following two conditions hold.
  \begin{itemize}
    \item For all $q \in Q, \bfw \in (\Sigma^{d})^{*}$ we have $T(q,\bfw)\in G$.
    \item For all $q,\overline{q}\in Q$ and $n \geq m_0$ it holds that
      \begin{align*}
	\{T(q,\bfw): \bfw \in \Sigma^{nd}, \delta(q,\bfw) = \overline{q}\} = G.
      \end{align*}
  \end{itemize}
  $d$ and $m_0$ only depend on $A$, but not on its initial state $q'_0$.
\end{theorem}
We provide an example of a primitive (and strongly connected) automaton for which $d>1$ holds.
\begin{example}
  Consider the following automaton.\\
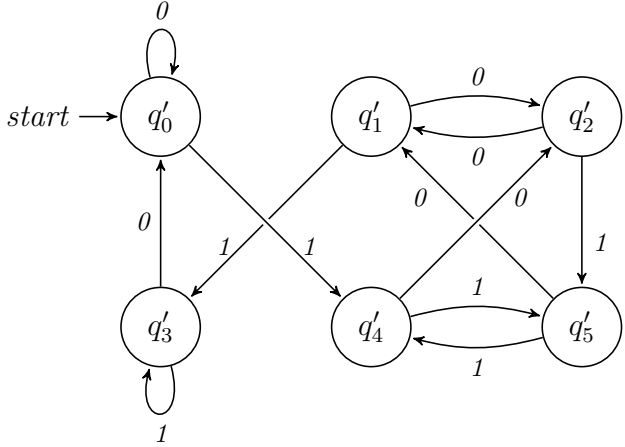
  \begin{tikzpicture}[->,>=stealth',shorten >=1pt,auto,node distance=2.8cm, semithick, bend angle = 15, cross line/.style={preaction={draw=white,-,line width=3pt}}]
    
    \node[state, initial](A)                    {$q_0'$};
    \node[state]         (B) [right of=A]	{$q_1'$};
    \node[state]         (C) [right of=B] 	{$q_2'$};
    \node[state]	 (D) [below of=A]	{$q_3'$};
    \node[state]         (E) [right of=D]	{$q_4'$};
    \node[state]         (F) [right of=E]       {$q_5'$};
    
    \path [every node/.style={font=\footnotesize}, pos = 0.66]
    (C) edge [bend left]  node [pos=0.5] {0}    (B)
	edge 		  node 		 {1} 	(F)
    (B) edge [bend left]  node [pos=0.5] {0}	(C)
	edge              node [left]	 {1}	(D)
    (A) edge [loop above] node [pos=0.5] {0} 	(A)
	edge [cross line] node [right]	 {1}	(E)
    (F) edge              node [left, xshift=-2mm]	 {0}	(B)
	edge [bend left]  node [pos=0.5] {1}	(E)
    (E) edge [cross line] node [right]	 {0}	(C)
	edge [bend left]  node [pos=0.5] {1}	(F)
    (D) edge              node [pos=0.5] {0}	(A)
	edge [loop below] node [pos=0.5] {1}	(D);
  \end{tikzpicture}
  
  This automaton is clearly primitive, i.e., its period is $1$, as $\delta(q_0,0) = q_0$.
  One choice for a naturally induced transducer is already indicated by the sketch above,
  \begin{align*}
    &q_0 = (q_0', q'_1, q'_2), \quad q_1 = (q'_3, q'_4, q'_5)\\
    &\delta(q_0,0) = q_0, \quad \delta(q_0,1) = q_1\\
    &\delta(q_1,0) = q_0, \quad \delta(q_1,1) = q_1\\
    &\lambda(q_0,0) = (23), \quad \lambda(q_0,1) = (12)\\
    &\lambda(q_1,0) = (23), \quad \lambda(q_1,1) = (23).
  \end{align*}
  One sees directly that $\sgn(\lambda(q,a)) = -1$ for all $a \in \Sigma$ and therefore 
  \begin{align*}
    T(q,\bfw) \in \left\{\begin{array}{cl}\{id,(123),(132)\} , & \mbox{for }\bfw \in (\Sigma^{2})^{*} \\ \{(12),(23),(13)\}, & \mbox{for }\bfw \in \Sigma(\Sigma^{2})^{*} \end{array}\right. 
  \end{align*}
  Indeed, one can show that $d=2, G = \{id,(123),(132)\}$ and $g_0 = (12)$, as all elements of $G$ appear as output for paths from $q_0$ to itself for length $4$ as well as for length $6$.
\end{example}

The other results of this subsection are rather technical; they will be used only in this and the following subsection.\\
According to \cref{prop:synch_trans} we assume in this section that $\mathcal{T}_{A}$ is a naturally induced transducer of a strongly connected automaton.
Thus, $\Delta$ is finite.\\
Paths $\bfw$ with the property 
\begin{align}\label{eq:good_path}
  \delta(q,\bfw) = q, T(q,\bfw) = id
\end{align}
will play an important role in this subsection and we start by exploring some of their properties.
\begin{remark}
  Obviously concatenating two paths $\bfw_1, \bfw_2$ fulfilling \eqref{eq:good_path} gives again a path fulfilling \eqref{eq:good_path}.
\end{remark}

We start this section by showing that every path can be extended to a path with this property.
\begin{lemma}\label{le:inv}
  For every $q \in Q, \bfw \in \Sigma^{*}$ there exists an \textbf{inverse path} $\overline{\bfw}^q \in \Sigma^{*}$ such that \eqref{eq:good_path} holds for $\bfw \overline{\bfw}^q$.
\end{lemma}

\begin{proof}
  Since $\mathcal{T}_{A}$ is strongly connected, there exists $\bfw' \in \Sigma^{*}$ such that\\
  $\delta(q,\bfw\bfw') = \delta(\delta(q,\bfw),\bfw') = q$.\\
  We define $g:=T(q,\bfw \bfw')$. Since $G$ is finite we know that $g^n = id$ for some $n \in \N$.
  Since $\delta(q,(\bfw \bfw')^n) = q$ and $T(q,(\bfw \bfw')^n) = id$ we know that the desired property holds for $\overline{\bfw}^q := \bfw' (\bfw \bfw')^{n-1}$.
\end{proof}

We define $M_q:=\{n\in \N|\exists \bfw \in \Sigma^{n}, \delta(q,\bfw) = q, T(q,\bfw) = id\}$ and it follows directly by \cref{le:inv} that $M_q$ is non-empty for all $q$.

\begin{lemma}\label{le:d}
  Let $A$ be a strongly connected automaton and $\mathcal{T}_{A}$ an induced transducer.
  There exists $d= d(\mathcal{T}_{A})\in \N$ such that for all $q \in Q: M_q \subset d \cdot \N$ and $d \cdot \N \setminus M_q$ is finite.
\end{lemma}
\begin{proof}
  The first remark in this subsection directly shows that $M_q$ is closed under addition for all $q \in Q$.
  Thus, it is easy to show that for every $q \in Q$ there exists $d_q$ such that $M_q \subset d_q \cdot \N$ and $d_q \cdot \N \setminus M_q$ is finite:\\
  We take $d_q := \gcd(M_q)$ which can be written as $d_q = \sum_{\ell<r} \alpha_{\ell} n_{\ell}$ for some $r\in \N$ and $n_{\ell} \in M_q$.
  Obviously, we find $M_q \subset d_q \cdot \N$.
  Let $m := \sum_{\ell<r} n_{\ell}$, then we can write any $N \in d_q\cdot \N$ as $rm + s d_q$ with $0\leq s < m/d_q$ and one finds for sufficiently large $N$ 
  - and, therefore, sufficiently large $r$ - that all coefficients
  of the $n_{\ell}$ are positive and, therefore, $N \in M_q$.
  
  It  remains to show that all $d_q$ coincide. 
  Let $q_1,q_2 \in Q$. 
  There exist $m \in \N$, $\bfw_1 \in \Sigma^{m}$ and $\bfw_2 \in \Sigma^{m+d_{q_1}}$ 
  such that \eqref{eq:good_path} holds for $\bfw_1,\bfw_2$ (for $q= q_1$).
  Furthermore, by \cref{le:transducer_regular} and \cref{le:inv}, there exist $\bfw, \overline{\bfw}^{q_2}$ such that 
  $\delta(q_2,\bfw) = q_1, \delta(q_1,\overline{\bfw}^{q_2}) = q_2$ and $T(q_2,\bfw \overline{\bfw}^{q_2}) = id$.
  Therefore, $\bfw \bfw_i \overline{\bfw}^{q_2}$ fulfill \eqref{eq:good_path}  for $i = 1,2$ (and $q = q_2$).
  These two paths belonging to $M_{q_2}$ have length $|\bfw| + |\overline{\bfw}^{q_2}| + m$ and $|\bfw| + |\overline{\bfw}^{q_2}| + m + d_{q_1}$ respectively. 
  As $M_{q_2} \subseteq d_{q_2} \N$ we see that $d_{q_2} \leq d_{q_1}$ holds.
  Since $q_1,q_2$ are arbitrary elements of $Q$ the statement follows.
\end{proof}
\begin{example}
  Consider the following naturally induced transducer with one state, where $k=2$.
  \begin{align*}
    \lambda(q_0,0) &= (23)\\
    \lambda(q_0,1) &= (123)
  \end{align*}
  One finds easily that $d = 1$ and $M_{q_0} = \N \setminus \{1\}$.
\end{example}

\begin{remark}
  Changing the order of the tuples $q \in Q$, i.e., considering a different naturally induced transducer, does not change $d(\mathcal{T}_{A})$ (since $1 \mapsto \sigma_q \circ id \circ \sigma_q^{-1} = id$).
  Thus we find that $d(\mathcal{T}_{A})$ only depends on $A$ and we write $d = d(A) = d(\mathcal{T}_{A})$.
\end{remark}

As mentioned before,
we are interested in which elements of $\Delta$ occur for paths with certain length. 
One might assume that some periodic behavior (with period $d$) might occur when restricting to paths of certain length and, furthermore, 
that for long enough paths some closure properties hold.

To find some precise statements, we define 
\begin{align*}
  M_{q \overline{q},\ell} := \{T(q,\bfw)|\bfw \in \Sigma^{\ell}, \delta(q,\bfw) = \overline{q}\},
\end{align*}
for non-negative $\ell$ and $M_{q \overline{q},\ell} = \emptyset$ for negative $\ell$.

$M_{q \overline{q},\ell}$ are all the group elements that occur when reading a word of length $\ell$ that corresponds to a path from $q$ to $\overline{q}$.
Furthermore, we define $G_{q \overline{q}}(\ell) := \bigcup_{\ell' \equiv \ell (d)} M_{q \overline{q},\ell'}$.
The following lemma shows that this union is actually a limit.
\begin{lemma}\label{le:m_00}
  For all $q,\overline{q} \in Q$ and for all $\ell$ such that $0 \leq \ell \leq d-1$ it holds that, $G_{q\overline{q}}(\ell) = M_{q\overline{q},\ell + k d}$ for almost all - i.e., all but finitely many - $k \in \N$. 
\end{lemma}
\begin{proof}
  It is sufficient to show that for each $g \in G_{q\overline{q}}(\ell)$ we have $g \in M_{q\overline{q},\ell +k d}$ for almost all $k \in \N$.\\
  For $g \in G_{q\overline{q}}(\ell)$, there exists $k_0 \in \N$ and $\bfw \in \Sigma^{\ell + k_0 d}$ such that $\delta(q,\bfw) = \overline{q}$ and $T(q,\bfw) = g$. 
  \cref{le:d} implies that we have for almost all $k' \in \N$:
  There exists $\bfw_{k'} \in \Sigma^{k'd}$ which fulfills \eqref{eq:good_path}.
  Since $\bfw_{k'} \bfw$ is a path from $q$ to $\overline{q}$ with weight $g$, we have for almost all $k \geq k_0$ (and therefore for almost all $k \in \N$), 
  $g \in M_{q\overline{q},\ell+kd}$.
\end{proof}
We are interested in how the different $G_{q\overline{q}}(\ell)$ are related.
We see easily that for $q_1,q_2,q_3 \in Q$, $G_{q_1q_2}(\ell_1) \cdot G_{q_2q_3}(\ell_2) \subseteq G_{q_1q_3}(\ell_1+\ell_2)$ holds, 
but we actually even find equality:

\begin{lemma}\label{le:G_prod}
  For all $\ell_1, \ell_2 \in \N$ and $q_1,q_2,q_3 \in Q$ it holds that
  \begin{align*}
    G_{q_1q_3}(\ell_1+\ell_2) = G_{q_1q_2}(\ell_1) \cdot G_{q_2q_3}(\ell_2).
  \end{align*}
  Furthermore, $|G_{q_0q_0}(0)| = |G_{q\overline{q}}(\ell)|$ holds for all $q,\overline{q} \in Q$ and $\ell$ such that $0 \leq \ell \leq d-1$.
\end{lemma}

\begin{proof}
  We start by showing $G_{q\overline{q}}(\ell)^{-1} = G_{\overline{q}q}(-\ell)$ for all $q,\overline{q} \in Q, \ell \in \N$.\\
  We know by \cref{le:inv} that for each $\bfw$ with $\delta(q,\bfw)=\overline{q}$, there exists $\overline{\bfw}^q$. 
  Thus we conclude that $|\bfw| + |\overline{\bfw}^q| \equiv 0 \bmod d$ and, therefore, $G_{q\overline{q}}(\ell)^{-1} \subseteq G_{\overline{q}q}(-\ell)$.
  By inverting both sides, we also find $G_{q\overline{q}}(\ell) \subseteq G_{\overline{q}q}(-\ell)^{-1}$ and since $q,\overline{q}$ and $\ell$ are chosen arbitrarily, 
  $G_{q\overline{q}}(\ell)^{-1} = G_{\overline{q}q}(-\ell)$
  holds for all $q,\overline{q}\in Q, \ell \in \N$.
  
  The inclusion $\supseteq$ is trivial.
  Let $g \in G_{q_1q_3}(\ell_1 + \ell_2)$. 
  By definition there exists $\bfw$ such that $\delta(q_1,\bfw) = q_3, T(q_1,\bfw) = g$ and $|\bfw| \equiv \ell_1 + \ell_2 \bmod d$. 
  \cref{le:transducer_regular} shows the existence of $\bfw_1$ with $\delta(q_1,\bfw_1) = q_2$ and $|\bfw_1| \equiv \ell_1 \bmod d$. 
  We find, by the first part of this proof, an inverse path $\overline{\bfw_1}^{q_1}$ from $q_2$ to $q_1$ such that $|\overline{\bfw_1}^{q_1}|\equiv -\ell_1 \bmod d$ holds.
  We see that $\bfw_2 := \overline{\bfw_1}^{q_1} \bfw$ is a path from $q_2$ to $q_3$ with $|\bfw_2| \equiv \ell_2 \bmod d$ and
  \begin{align*}
    T(q_1,\bfw_1 \bfw_2) &= T(q_2,\bfw_1) \cdot T(q_1,\bfw_2) = T(q_1,\bfw_1) \cdot T(q_2,\overline{\bfw_1}^{q_1}) \cdot T(q_1,\bfw) \\
    &= T(q_1,\bfw_1 \overline{\bfw_1}^{q_1}) \cdot g= g.
  \end{align*}
  This finishes the proof of the first statement.
  We see by the first part of this lemma, that $G_{q_1q_3}(\ell) \supseteq G_{q_1q_2}(\ell') \cdot g_2$ and $G_{q_1q_3}(\ell) \supseteq g_1 \cdot G_{q_2q_3}(\ell')$ 
  for $g_1 \in G_{q_1q_2}(\ell - \ell'), g_2 \in G_{q_2q_3}(\ell-\ell')$. Thus we find
  \begin{align*}
    |G_{q_1q_3}(\ell)| \geq |G_{q_1q_2}(\ell')|\\
    |G_{q_1q_3}(\ell)| \geq |G_{q_2q_3}(\ell')|
  \end{align*}
  for all $q_1,q_2,q_3 \in Q$ and $\ell, \ell' \in \N$.
We can use this fact to show that
\begin{align*}
  |G_{q_0 q_0}(0)| \geq |G_{q_0 q}(0)| \geq |G_{qq}(\ell)| \geq |G_{qq_0}(0)| \geq |G_{q_0 q_0}(0)|.
\end{align*}
Therefore we see that $|G_{qq}(\ell)| = |G_{q_0q_0}(0)|$ for all $q \in Q$ and $\ell \in \N$. To complete the proof we see that
\begin{align*}
  |G_{q_0q_0}(0)| = |G_{qq}(0)| \geq |G_{q\overline{q}}(\ell)| \geq |G_{\overline{q}\overline{q}}(0)| = |G_{q_0q_0}(0)|.
\end{align*}
\end{proof}
\begin{remark}
  \cref{le:G_prod} gives a strong structural results. 
  Take, for example, any element $g_2\in G_{q_2q_3}(\ell_2)$. As $|G_{q_1q_3}(\ell_1+\ell_2)| = |G_{q_1q_2}(\ell_1)|$, it follows that
  $G_{q_1q_3}(\ell_1 + \ell_2) = G_{q_1q_2}(\ell_1) \cdot g_2$.
  A similar result holds for $g_1 \in \G_{q_1q_2}(\ell_1)$.
\end{remark}

\begin{corollary}\label{co:G_qq_subgroup}
  $G_{qq}(0)$ is a subgroup of $\Delta$ for all $q \in Q$.
\end{corollary}

$G_{qq}(0)$ being a subgroup is the first example for the appearance of some closure property occurring in this setting.\\
In this subsection, we also prove that all $G_{q\overline{q}}(\ell)$ are cosets of $G_{q_0q_0}(0)$, but first we mention a slightly stronger result than \cref{le:m_00} which we need in the next subsection.

\begin{lemma}\label{le:m_0}
  There exists $m_0 = m_0(A)$ such that for all $q,\overline{q} \in Q$ and for all $\ell$ such that $0 \leq \ell \leq d-1, g \in G_{q\overline{q}}(\ell)$ and $k \geq m_0$ 
  there exist $\bfw_1,\bfw_2 \in \Sigma^{\ell +kd}$ such that
  $\bfw_1 \neq \bfw_2$ and $\delta(q,\bfw_i) = \overline{q}, T(q,\bfw_i) = g$.
\end{lemma}
\begin{proof}
  \cref{le:m_00} guarantees the existence of some $m_{1}$ such that 
  for all $q,\overline{q} \in Q, 0 \leq \ell \leq d-1, g \in G_{q\overline{q}}(\ell)$ and $k' \geq m_{1}$ exists $\bfw \in \Sigma^{\ell + k'd}$ such that
  $\delta(q,\bfw) = \overline{q}, T(q,\bfw) = g$.
  We will show that the desired property holds for $k = 2 m_{1}$, which implies the statement, since there exists for all $k' \geq m_{1}$ a path $\bfw \in \Sigma^{k' d}$ with
  $\delta(\overline{q},\bfw) = \overline{q}$ and $T(\overline{q},\bfw) = id$.\\
  \begin{itemize}
   \item Case $|Q| \geq 2$:\\
      We fix $q_1\neq q_2 \in Q$.
      \cref{le:transducer_regular} yields the existence of some paths $\bfw_1',\bfw_2'$ of length $m_1 d$ such that $\delta(q,\bfw_i') = q_i$.
      Of course, we find $T(q,\bfw_i') \in G_{q q_i}(0)$ which shows together with
      \cref{le:G_prod} that $T(q,\bfw_i')^{-1} g \in G_{q_i \overline{q}}(\ell)$.
      \cref{le:m_00} shows that there exist paths $\bfw_1'',\bfw_2''$ of length $\ell + m_1 d$ such that $\delta(q_i,\bfw_i'') = \overline{q}$ and $T(q,\bfw_i' \bfw_i'') = g$.
   \item Case $|Q| = 1, |G_{q_0 q_0}(0)| \geq 2$:\\
      Thus we have $Q = \{q_0\}$.
      We fix $g_1 \neq g_2 \in G_{q_0 q_0}(0)$.
      A combination of \cref{le:m_00} and \cref{le:G_prod} yields the existence of some paths $\bfw_1',\bfw_2'$ of length $m_1 d$ 
      and paths $\bfw_1'',\bfw_2''$ of length $\ell + m_1 d$ such that $T(q_0,\bfw_i') = g_i$ and $T(q_0,\bfw_i' \bfw_i'') = g$.
   \item Case $|Q| = |G_{q_0 q_0}(0)| = 1$:\\
      Obviously taking any two different $\bfw_i'\in \Sigma^{m_1 d}$ and $\bfw \in \Sigma^{\ell + m_1 d}$ such that
      $T(q_0,\bfw_i') = id, T(q_0,\bfw) = g$,
      gives $T(q_0,\bfw_i' \bfw) = g$.
  \end{itemize}
  Thus we have constructed in all cases two different paths with the desired properties.
\end{proof}

We have seen in \cref{pr:T_permute} that permuting the elements of $Q$ again gives an induced transducer.
We now show that by choosing the right permutations, we are able to make the actual structure of the induced transducer more apparent.

\begin{lemma}\label{le:id}
  Let $\mathcal{T}_{A}$ be a naturally induced transducer. 
  There exists a naturally induced transducer $\overline{\mathcal{T}}_{A}$ such that
  $id \in \overline{G}_{\overline{q_0} \overline{q}}(0)$ holds for each $\overline{q} \in \overline{Q}$.
\end{lemma}
\begin{proof}
  The idea is to find permutations $\sigma_q$ such that the induced transducer defined in \cref{pr:T_permute} has the desired properties.
  We take $\sigma_q \in G_{q_0 q}(0)$ for all $q \neq q_0$ and $\sigma_{q_0} = id$.
  By \cref{pr:T_permute}, we find
  $\overline{G}_{\overline{q_0 q}}(0) = G_{q_0 q} \cdot \sigma_{q}^{-1}$, which directly finishes the proof.
\end{proof}
We consider, from now on, only naturally induced transducers for which $id \in G_{q_0 q}(0)$ holds for all $q\in Q$.

We define $G := G_{q_0 q_0}(0)$ and find that the condition above is already sufficient to make $G_{q\overline{q}}(\ell)$ independent of $q,\overline{q}$, 
which makes the actual structure more apparent.


\begin{proposition}\label{pr:G_g_0}
  There exists $g_0 \in \Delta$ such that we have for all $q,\overline{q} \in Q, \ell \in \N$ 
  \begin{align*}
    G_{q\overline{q}}(\ell) = G \cdot g_0^{\ell} = g_0^{\ell} \cdot G.
  \end{align*}
\end{proposition}
\begin{remark}
  It easily follows that $d|ord(g_0)$ and for $d>1, g_0 \notin G$.\\
  We also find that $\Delta = G \cdot \{g_0^{\ell}: \ell \in \N\} = \{g_0^{\ell}:\ell \in \N\} \cdot G$,
  i.e.,~$G \vartriangleleft \Delta$.\\
  This proposition shows together with \cref{le:m_00} that $M_{q\overline{q},\ell}$ depends only on $(\ell \bmod d)$ for large $\ell$.
\end{remark}
\begin{proof}
  As a first step, we prove that for all $q,\overline{q} \in Q$ we have $G_{q\overline{q}}(0) = G$:\\
   We know by \cref{le:G_prod} that $G_{q_0 q}(0) = G_{q_0 q_0}(0) \cdot G_{q_0 q}(0)$ and without loss of generality by \cref{le:id}, $id \in G_{q_0q_0}(0)$.
   Thus we have $G_{q_0 q_0}(0) \cdot G_{q_0 q}(0) \supseteq G \cdot id$ and, by comparing the cardinalities, we see that
  $G_{q_0 q}(0) = G$. 
  Analogously, one finds that $G_{q_0 \overline{q}}(0) = G$. 
  \cref{le:G_prod} gives again $G_{q_0 \overline{q}}(0) = G_{q_0 q}(0) \cdot G_{q \overline{q}}(0)$ (i.e.,~$G = G \cdot G_{q \overline{q}}(0)$) 
  and one concludes as above that $G_{q\overline{q}}(0) = G$.
  
  Now we consider $G_{q_0 q_0}(1)$. We find by \cref{le:G_prod}
  \begin{align*}
    G_{q_0 q_0} (1) = G \cdot G_{q_0 q_0}(1) = G_{q_0 q_0}(1) \cdot G.
  \end{align*}
  We now select an arbitrary element $g_0$ of $G_{q_0 q_0}(1)$ and find
  \begin{align*}
    G_{q_0 q_0} (1) = G \cdot g_0 = g_0 \cdot G.
  \end{align*}
  
  One easily shows by induction that $G_{q_0 q_0} (\ell) = G \cdot g_0^{\ell} = g_0^{\ell} \cdot G$ holds for all $\ell \in \N$.
  
  It just remains to note that
  \begin{align*}
    G_{q\overline{q}}(\ell) = G_{qq_0}(0) \cdot G_{q_0 q_0}(\ell) \cdot G_{q_0 \overline{q}}(0) = G \cdot (g_0^{\ell} \cdot G) \cdot G = g_0^{\ell} \cdot G.
  \end{align*}
\end{proof}
Now we can prove \cref{th:full_G}.

\begin{proof}[Proof of \cref{th:full_G}: ]
  We choose all the variables as defined throughout this section.
  The statement of this theorem easily follows by a combination of \cref{le:m_00} and \cref{pr:G_g_0}.
\end{proof}

\begin{remark}
  For the period $p$ of $A$, one easily finds that $p | d$. However, $p=d$ need not hold as the example provided after \cref{th:full_G} shows.
\end{remark}

%

\subsection{Arithmetic restriction for naturally induced transducers}\label{sec:str2}
\cref{th:full_G} shows that all elements of $G$ occur when we restrict ourself to paths whose length is divisible by $d$.
We show in this subsection that -- similarly to \cref{sec:str1} -- restrictions on $[\bfw]_k \bmod (k^{d}-1)$ lead to restrictions on what elements of $G$ occur.\\
The main result of this subsection is the following theorem.
\begin{theorem}\label{th:dk_0}
  Under the same conditions as in \cref{th:full_G},
  there exist natural numbers $k_0, m'_0, \ell_0$ and $d'$ (where $d' | k^{d}-1$ and therefore $\gcd(d',k) = 1$) together with a
  naturally induced transducer $\mathcal{T}_{A}$ fulfilling the properties of \cref{th:full_G}, a subgroup $G_0$ of $G$
  and $g'_0 \in G$ fulfilling the following two conditions.
  \begin{itemize}
    \item For all $q,\overline{q}\in Q$ it holds that
      \begin{align*}
	\{T(q,\bfw): q \in Q, \bfw \in (\Sigma^{d k_0})^{*}, \delta(q,\bfw) = \overline{q}, [\bfw]_k \equiv \ell \bmod d'\} = G_0 \cdot g_0'^{\ell} = g_0'^{\ell} \cdot G_0.
      \end{align*}
    \item For all $q,\overline{q}\in Q$ there exists $d''(q,\overline{q}) | k^{\ell_0}$ such that for all  $g \in G_0$ and $m \geq m'_0$ holds
      \begin{align*}
	\gcd \{[\bfw]_k: \bfw \in \Sigma^{d k_0 m}, \delta(q,\bfw) = \overline{q}, T(q,\bfw) = g\} = d' \cdot d''(q,\overline{q}).
      \end{align*}
  \end{itemize}
  All of the variables only depend on $A$, but not on its initial state $q'_0$.
\end{theorem}

\begin{example}
		We continue the example from \cref{sec:str1}.
		The naturally induced transducer we considered was
		
		\tikzset{elliptic state/.style={draw,ellipse}}
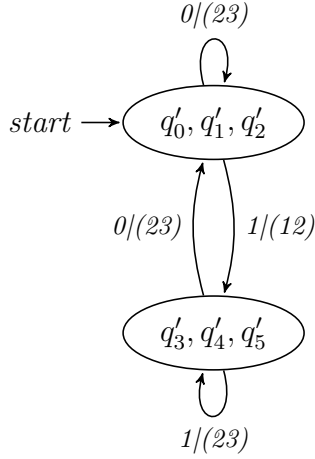
		\begin{tikzpicture}[->,>=stealth',shorten >=1pt,auto,node distance=2.8cm, semithick, bend angle = 15, cross line/.style={preaction={draw=white,-,line width=4pt}}]
			
			\node[elliptic state, initial]	(A)                    {$q_0', q_1', q_2'$};
			\node[elliptic state]	         (B) [below of=A] 	{$q_3', q_4', q_5'$};
			
			\path [every node/.style={font=\footnotesize}, pos = 0.5]
			(A) edge [loop above] node 		 	{0|(23)}      	(A)
		edge [bend left]  node 		 	{1|(12)} 	(B)
			(B) edge [bend left]  node 			{0|(23)}	(A)
		edge [loop below] node			{1|(23)}	(B);
		\end{tikzpicture}
  
  We have already seen that $d=2$ holds. 
  First we would like to mention that it is not necessary to know the values of $d'$ and $d''$ to construct the naturally induced transducer, as the proof of \cref{le:s_equal} shows.
  Now we derive $d'$ and $d''$. We find easily that $\bfw_{q_0} = 0, \bfw_{q_1} = 1$ are synchronizing words.
  Furthermore, a path ends in $q_0$ if and only if it ends with $0$ and a path ends in $q_1$ if and only if it ends with $1$.
  As we will see, $d''(q,\overline{q})$ corresponds to arithmetic restrictions for paths from $q$ to $\overline{q}$.
  Thus, we find $d''(.,q_0) = d''(.,q_1) = 2$. In general it is sufficient to consider all paths of length at most $\abs{Q}-1$ to determine $d''$.
  
  We compute 
  \begin{align*}
    \delta(q_1,00) = \delta(q_1,10) = q_0\\
    T(q_1,00) = T(q_1,10) = id.
  \end{align*}
  Thus we find $2|d(q_1, q_0)$, $d' = 1$ and $k_0 = 1$.

  In general it is more involved to determine $d'$. 
  One straightforward way is to determine $d_{id,\ell}^{q_0 q_0}$ (defined shortly) for increasing values of $\ell$ 
   until the conditions of \cref{th:dk_0} are fulfilled for $d' \cdot d''(q_0,q_0) = d_{id,\ell}^{q_0 q_0}$.
\end{example}

The rest of this subsection is again rather technical and only used to prove this theorem.

We only consider $G$ (i.e.,~weights corresponding to paths whose length is divisible by $d$) as just the same phenomena occur in the general situation.
To assure that we are only working on these elements, we only consider paths whose length is divisible by $d$.

\cref{le:m_0} assures, that we have for $\ell \geq m_0$
\begin{align*}
  \forall q,\overline{q}\in Q \quad \forall g \in G \quad \exists \bfw_1, \bfw_2 \in \Sigma^{d \ell}: \bfw_1 \neq \bfw_2, \delta(q,\bfw_i) = \overline{q}, T(q,\bfw_i) = g.
\end{align*}

Therefore, we define for $\ell \geq m_0, g \in G$ 
{\small
\begin{align*}
	d_{g,\ell}^{q\overline{q}} := \max\{m \in \N|\exists r: \forall \bfw \in \Sigma^{d \ell} \text{ such that }
\delta(q,\bfw) = \overline{q} \text{ and } T(q,\bfw) = g \Rightarrow [\bfw]_k \equiv r \bmod m\}.
\end{align*}
}%
An equivalent definition for $d_{g,\ell}^{q\overline{q}}$ is the greatest common divisor of all differences of numbers $[\bfw]_k$ corresponding to 
paths of length $d \ell$ from $q$ to $\overline{q}$ with weight $g$.
We will use one of the two definitions depending on the situation.
Our next goal is to show that $d_{g,\ell}^{q\overline{q}}$ converges to some $d(q,\overline{q})$ for all $g \in G$.
\begin{lemma}\label{le:d_0}
  Let $q,\overline{q} \in Q$. Then there exists $d(q,\overline{q}):= \lim_{\ell \to \infty} d_{id,\ell}^{q\overline{q}}$ and for all $\ell \geq m_0, g \in \Delta$ we have $d(q,\overline{q}) | d_{g,\ell}^{q\overline{q}}$ and 
  there exists $m_0'$ (not depending on $q,\overline{q}$) such that for all $\ell \geq m_0', g \in \Delta$ we have $d_{g,\ell}^{q\overline{q}} = d(q,\overline{q})$.
\end{lemma}

\begin{proof}
  We start by showing that for all $\ell,\ell' \geq m_0, q,\overline{q} \in Q$ and $g_1, g_2 \in G$ it holds that $d_{g_1,\ell+\ell'}^{q\overline{q}} | d_{g_2,\ell}^{q\overline{q}}$.
  \cref{th:full_G} shows that
  there exists $\bfw \in \Sigma^{d \ell'}$ such that $\delta(q,\bfw) = q$ and $T(q,\bfw) = g_1\cdot g_2^{-1}$. We find
  \begin{align*}
    d_{g_1,\ell+\ell'}^{q\overline{q}} 
	&= gcd(\{[\bfw_1]_k-[\bfw_2]_k: \bfw_i \in \Sigma^{d(\ell+\ell')} \text{ with } \delta(q,\bfw_i) = \overline{q}, T(q,\bfw_i) = g_1\})\\
	& | \quad gcd(\{[\bfw \bfw_1]_k-[\bfw \bfw_2]_k: \bfw_i \in \Sigma^{d \ell} \text{ with } \delta(q,\bfw \bfw_i) = \overline{q}, T(q,\bfw \bfw_i) = g_1\})\\
	&= gcd(\{[\bfw_1]_k-[\bfw_2]_k: \bfw_i \in \Sigma^{d \ell} \text{ with } \delta(q,\bfw_i) = \overline{q}, T(q,\bfw_i) = g_2\})\\
	&= d_{g_2,\ell}^{q\overline{q}}
  \end{align*}
Let $d_{g}^{q\overline{q}}$ denote the minimal value of $d_{g,\ell}^{q\overline{q}}$ for all $\ell$ (we choose $\ell_0 \geq m_0$ such that $d_{g,\ell_0}^{q\overline{q}} = d_{g}^{q\overline{q}}$)
and since
\begin{align*}
  d_{g}^{q\overline{q}} = d_{g,\ell + \ell_0}^{q \overline{q}} \leq d_{g,\ell_0}^{q\overline{q}} = d_{g}^{q\overline{q}}
\end{align*}
there exists $m_0'$ such that for all $\ell \geq m_0'$ it holds that
$d_{g,\ell}^{q\overline{q}} = d_{g}^{q\overline{q}}$.
It also follows directly that $d_{g_1}^{q\overline{q}} = d_{g_2}^{q\overline{q}}$ for all $g_1,g_2 \in \Delta$ and the result follows directly.
\end{proof}

Let $\bfw \in \Sigma^{\ell}$ and suppose $\delta(q,\bfw) = \overline{q}$, $T(q,\bfw) = g$, we then denote $r_{\ell}^{q\overline{q}}(g) := [\bfw]_k \bmod d(q,\overline{q})$ 
as we see directly that this definition does not depend on the choice of $\bfw$, but only on $q,\overline{q}$ and $g$.
More generally, it follows that for $\bfw \in \Sigma^{d \ell}, \delta(q,\bfw) = \overline{q}$ we have 
\begin{align*}
 r_{\ell}^{q\overline{q}} (T(q,\bfw)) \equiv [\bfw]_k \bmod d(q,\overline{q}).
\end{align*}

\begin{remark}
  By assumption any word containing a synchronizing word is again synchronizing, and by assumption $\mathcal{T}_{A}$ is synchronizing and strongly connected, 
  there exists a minimal $\ell_0 = \ell_0(A)$ such that
  for all $q \in Q$ there exists $\bfw_q\in \Sigma^{\ell_0}$ with
    \begin{align*}
      \forall \overline{q} \in Q: \delta(\overline{q},\bfw_q) = q.
    \end{align*}
\end{remark}

We find an important restriction on $d(q,\overline{q})$.

\begin{lemma}
  We have for every $q,\overline{q} \in Q$
  \begin{align*}
    k^{\ell_0}(k^d - 1) \equiv 0 \bmod d(q,\overline{q}).
  \end{align*}
\end{lemma}
\begin{proof}
  We fix $q,\overline{q}$ and start by considering the case where we concatenate a word $\bfw$ with some words $\bfw_1, \bfw_2$ from left and right, respectively,
  such that the weight does not change: 
  Let $\bfw \in \Sigma^{m_2}, \bfw_1,\bfw_2 \in \Sigma^{d m_1}$ ($m_1\geq m_0, m_2 \geq m_0'$) 
  such that $\delta(q,\bfw_1) = q, T(q,\bfw_1) = id$ and $\delta(\overline{q},\bfw_2) = \overline{q}, T(\overline{q},\bfw_2) = id$, see \cref{th:full_G}.
  We find by simple computations and \cref{le:d_0}
  \begin{align*}
		r_{m_2 + d m_1}^{q \overline{q}} (T(q,\bfw_1\bfw)) &= r_{d m_1 + m_2}^{q \overline{q}} (T(q,\bfw \bfw_2))\\
    [\bfw_1]_k k^{m_2} + [\bfw]_k &\equiv [\bfw]_k k^{d m_1} + [\bfw_2]_k \bmod d(q,\overline{q})\\
    [\bfw_1]_k k^{m_2} - [\bfw_2]_k &\equiv [\bfw]_k (k^{d m_1}-1) \bmod d(q,\overline{q}).
  \end{align*}
  The left side of the last equation only depends on $\bfw_1,\bfw_2$ and $m_2$.
  By choosing $m_2 \geq \ell_0 +1$, $\bfw = 0\bfw_{\overline{q}}$ or $\bfw = 1\bfw_{\overline{q}}$ and -- if $\bfw$ is not of sufficient length -- adding zeros from the left
  we find
  \begin{align*}
    [\bfw_{\overline{q}}]_k (k^{d m_1}-1) &\equiv (k^{\ell_0} + [\bfw_{\overline{q}}]_k)(k^{d m_1}-1) \bmod d(q,\overline{q})\\
    k^{\ell_0}(k^{d m_1}-1) &\equiv 0 \bmod d(q,\overline{q}).
  \end{align*}
  Comparing the results for $m_1$ and $m_1+1$ completes the proof.
  \end{proof}

  This allows us to decompose $d(q,\overline{q})$ into two co-prime factors $d'(q,\overline{q}), d''(q,\overline{q})$ such that\\ 
  $d'(q,\overline{q}) | (k^d-1)$ and $d''(q,\overline{q}) | k^{\ell_0}$.
  On the one hand $d''(q,\overline{q})$ corresponds to restrictions on how you can reach the state $\overline{q}$, on the other hand $d'(q,\overline{q})$ corresponds to some restrictions 
   $ \bmod \hspace{1.5mm} k^d-1$, which is of greater interest for us.

\begin{lemma}\label{le:d'}
  For every $q,\overline{q}\in Q$ it holds that $d'(A) := d'(q_0,q_0)= d'(q,\overline{q})$.
\end{lemma}
\begin{proof}
  Since we have already gained some knowledge about $d(q,\overline{q})$ we can adapt \cref{le:d_0} to our needs.\\
  Let $\ell, \ell_1, \ell_2 \geq m_0'$, $q_1,q_2,\overline{q_1},\overline{q_2}\in Q$ and paths $\bfw \in\Sigma^{d \ell_1}$ from $q_1$ to $\overline{q_1}$ with weight $id$, 
  $\bfw' \in \Sigma^{d \ell_2}$ from $\overline{q_2}$ to $q_2$ with weight $id$.
  It follows, as in \cref{le:d_0},
	{\small
  \begin{align*}
    d(q_1,q_2) &= d_{id,\ell_1+\ell+\ell_2}^{q_1,q_2} = gcd(\{[\bfw_1]_k-[\bfw_2]_k: \bfw_i \in \Sigma^{d(\ell_1 + \ell + \ell_2)} \text{ with } 
	\delta(q_1,\bfw_i) = q_2, T(q_1,\bfw_i) = id\})\\
    &| \quad gcd(\{[\bfw \bfw_1 \bfw']_k - [\bfw \bfw_2 \bfw']_k: \\
	  &\qquad \qquad \bfw_i \in \Sigma^{d \ell} \text{ with } \delta(q_1,\bfw \bfw_i \bfw') = q_2, T(q_1, \bfw \bfw_i \bfw') = id\})\\
    & = gcd(\{k^{d \ell_2}([\bfw_1]_k-[\bfw_2)]_k: \bfw_i \in \Sigma^{d \ell} \text{ with } \delta(\overline{q_1},\bfw_i) = \overline{q_2}, T(\overline{q_1},\bfw_i) = id\})\\
    & = k^{d \ell_2} d_{id,\ell}^{\overline{q_1},\overline{q_2}} = k^{d \ell_2} d(\overline{q_1},\overline{q_2})
  \end{align*}
	}%
  for arbitrary $q_1,q_2,\overline{q_1},\overline{q_2}$.
  Thus we find $d'(q_1,q_2)|d'(\overline{q_1},\overline{q_2})$.
  As changing the order of some state $q \in Q$ does not change $d(\mathcal{T}_{A})$ ($id \mapsto \sigma_q \circ id \circ \sigma_q^{-1}$), 
  we find that $d'$ only depends on $A$.
\end{proof}

As mentioned before, we are more interested in $d'$ and define $s^{q\overline{q}}_{\ell}(g) := r_{\ell}^{q\overline{q}}(g) \bmod d'$. 
We find the following important properties:
\begin{lemma}\label{le:s_add}
  For all $q_1,q_2,q_3 \in Q, m_1,m_2 \geq m_0'$, and $g_1, g_2 \in G$ follows
  \begin{align*}
    s^{q_1,q_3}_{d(m_1+m_2)}(g_1\cdot g_2) = s^{q_1,q_2}_{d m_1}(g_1) + s^{q_2,q_3}_{d m_2}(g_2).
  \end{align*}
\end{lemma}
\begin{proof}
	We find by \cref{th:full_G} that there exists $\bfw_1 \in \Sigma^{d m_1}$ and $\bfw_2 \in \Sigma^{d m_2}$ such that
	$\delta(q_1,\bfw_1) = q_2, T(q_1,\bfw_1) = g_1$ and $\delta(q_2,\bfw_2) = q_3, T(q_2,\bfw_2) = g_2$.
	This gives, 
	\begin{align*}
		s^{q_1,q_2}_{d m_1}(g_1) &\equiv [bfw_1]_k \bmod d'\\
		s^{q_2,q_3}_{d m_2}(g_2) &\equiv [bfw_2]_k \bmod d'\\
		s^{q_1,q_3}_{d(m_1+m_2)}(g_1\cdot g_2) &\equiv [bfw_1]_k k^{d m_2} + [bfw_2]_k \bmod d'.
	\end{align*}
  Thus, the lemma follows directly from the fact that $k^{d m_2}\equiv 1 \bmod k^{d}-1$.
\end{proof}

\begin{lemma}
  There exists $k_0 \in \N$ such that for every $\ell \geq m_0'$ we have $s_{d \ell}^{q_0q_0}(id) = 0 \Leftrightarrow \ell \equiv 0 \bmod k_0$.
\end{lemma}
\begin{proof}
  We find by \cref{le:s_add} $s_{d(\ell_1 + \ell_2)}^{q_0q_0}(id) = s_{d \ell_1}^{q_0q_0}(id) + s_{d \ell_2}^{q_0q_0}(id) \bmod d'$.
  The statement follows by the same arguments as we used in the proof of \cref{le:d}.
\end{proof}
\begin{remark}
  One can actually prove that the Lemma above holds for all $\ell \in \N$ for which $s_{d \ell}^{q_0 q_0}(id)$ is properly defined.
  Furthermore, $k_0 | d'$ holds.
\end{remark}

We further follow the ideas of \cref{sec:str1} and find the following result.
\begin{lemma}\label{le:s_equal}
  There exists a naturally induced transducer $\overline{\mathcal{T}}_{A}$ such that
  $\overline{s}_{d k_0 m_0'}^{\overline{q_0 q}} (id) = 0$ holds for all $\overline{q} \in \overline{Q}$.
\end{lemma}
\begin{proof}
  We want to find some permutations $\sigma_q$ such that applying them gives an induced transducer with the desired properties.
  Let $q \neq q_0$, take $\bfw \in \Sigma^{d k_0 m_0'}$ such that $[\bfw]_k \equiv 0 \bmod d', \delta(q_0,\bfw) = q$ and define $\sigma_{q} = T(q_0,\bfw)$. 
  Without loss of generality we restrict ourself to $\bfw = \bfw' \bfw_q$, where $\bfw_q$ is again a synchronizing word -- we can easily choose a suitable $\bfw'$.\\
  By choosing $\sigma_{q_0} = id$ we find by using \cref{pr:T_permute} that
  $\overline{T}(\overline{q_0},\bfw) = T(q_0,\bfw) \cdot T(q_0,\bfw)^{-1} = id$.\\
  Note that applying these permutations does not change the property $id \in G_{q_0 q}(0)$ used in \cref{sec:str1}
\end{proof}
We consider from now on only naturally induced transducers such that $s_{k_0 m_0'}^{q_0q}(id) = 0$ holds for all $q\in Q$.\\
We find a result similar to the one in \cref{sec:str1} concerning $M_{q\overline{q},\ell}$, which shows that $s_{d\ell}^{q\overline{q}}(g)$
only depends on $\ell \bmod k_0$ and $g$.
\begin{proposition}
  For all $q,\overline{q} \in Q, g \in G$ and $\ell \in \N, \ell \geq m_0'$ we have 
  \begin{align*}
    s_{\ell}^{q\overline{q}}(g) = s_{\ell \bmod k_0}(g),
  \end{align*}
  where $s_{\ell}(g) := s_{\ell \bmod k_0}(g) := s_{k_0 m_0' + (\ell \bmod k_0)}^{q_0q_0}(g)$.
\end{proposition}
\begin{proof}
  We find by \cref{le:s_add} that
  \begin{align*}
    0 = s_{2 k_0 m_0'}^{q_0q_0}(id) = s_{k_0 m_0'}^{q_0 q}(id) + s_{k_0 m_0'}^{q q_0}(id) = s_{k_0 m_0'}^{q q_0}(id).
  \end{align*}
  Let us now assume $\ell \geq 5 k_0 m_0'$. We have
  \begin{align*}
    s_{\ell}^{q\overline{q}}(g) &= s_{k_0 m_0}^{q q_0}(id) k^{\ell-k_0m_0} + s_{\ell-2 k_0m_0}^{q_0q_0}(g)k^{k_0m_0} + s_{k_0 m_0}^{q_0\overline{q}}(id)\\
	  &= s_{\ell-2 k_0 m_0}^{q_0q_0}(g) = s_{k_0 m_0 + (\ell \bmod k_0)}^{q_0q_0}(g) + s_{\ell - (\ell \bmod k_0) - 3 k_0 m_0}^{q_0q_0}(id) = s_{\ell}(g).
  \end{align*}
  Let now $\ell \geq m_0'$. We find
  \begin{align*}
    s_{\ell}^{q\overline{q}}(g) = s_{\ell}^{q\overline{q}}(g) + 5 s_{\ell}^{\overline{q}\overline{q}}(id) = s_{\ell + 5 k_0 m_0'}^{q\overline{q}}(g) = s_{\ell \bmod k_0}(g),
  \end{align*}
  which finishes the proof.  
\end{proof}


We define $G_{\ell} := \{g \in G: s_0(g) = \ell\}$. 

\begin{corollary}
  $G_{0}$ is a subgroup of $G$.
\end{corollary}
\begin{proof}
  The statement follows directly by \cref{le:s_add}.
\end{proof}
Thus we can prove now \cref{th:dk_0}.

\begin{proof}[Proof of \cref{th:dk_0}: ]
  Of course we choose $k_0, m'_0, \ell_0, d$ and $G_0$ as we defined them throughout this section.
  To prove the first part we choose $g_0'$ as an arbitrary element of $G_{1}$ and the proof follows by the same arguments as in the proof of \cref{th:full_G}.
  The second part is just the result of \cref{le:d'}.
\end{proof}

\subsection{Reduction to special naturally induced transducers}\label{sec:reduce}
We want to show how to reduce the general setting of an automatic sequence to automata with special properties.

\begin{proposition}\label{pr:components_d_k_0}
  Let $\mathbf{a}=(a_n)_{n\in\N}$ be a $k$-automatic sequence. There exists $p \in \N$ and $A = (Q',\{0,\ldots,k^{p}-1\},\delta',q_0')$ such that
  $\mathbf{a}$ is generated by $A$ and for every automaton $A_i$ that is a restriction to a final component
  of $A$, it holds that $d(A_i) = k_0(A_i) = 1$.
\end{proposition}
\begin{proof}
  Let $\overline{A} = (\overline{Q},\{0,\ldots,k-1\},\overline{\delta},\overline{q_0})$ be an automaton that generates $\mathbf{a}$.
  We can assume without loss of generality that $\overline{\delta}(\overline{q_0},0) = \overline{q_0}$ and every state $\overline{q}\in \overline{Q}$ is reachable from $\overline{q_0}$.
  Let $\overline{Q}_1,\ldots,\overline{Q}_{\ell}$ be the final components of $\overline{A}$.
  We let $\overline{A}_i$ denote the automaton that corresponds to the restriction of $\overline{A}$ to $\overline{Q}_i$, with arbitrary initial state.\\
  We define $p = \lcm(d(\overline{A}_1)k_0(\overline{A}_1),\ldots,d(\overline{A}_{\ell})k_0(\overline{A}_{\ell}))$ 
  and note that $d(\overline{A}_i)$ and $k_0(\overline{A}_i)$ do not depend on the initial state of $A_i$ -- compare \cref{th:full_G} and \cref{th:dk_0}.
  The idea is now to take the ``$p$-th power" of $\overline{A}$ which will give the desired property.
  
  Therefore we define $A = (\overline{Q},\{0,\ldots,k^{p}-1\},\delta',\overline{q_0})$ where $\delta'$ is the extension of $\overline{\delta}$ to 
  $\{0,\ldots,k-1\}^p \cong \{0,\ldots,k^{p}-1\}$ from letters to words of length $p$.
  We see that $A$ generates $\mathbf{a}$ as $\overline{\delta}(\overline{q_0},0) = \overline{q_0}$ ensures that adding leading zeros does not change the output (of $\overline{A}$).
	
  One easily observes that every final component of $A$ is contained in a final component of $\overline{A}$:
  The strongly connected components of a directed graph form a new directed acyclic graph where every node corresponds to a strongly connected component.
  There exists an edge between two such nodes if and only if there is an edge between two states of these strongly connected components.
  Every final component of $A$ is part of a final component of $\overline{A}$, since there exists a path (whose length is divisible by $p$)
  from every strongly connected component, which is not closed under $\overline{\delta}$ to a final component.
  Therefore, we can restrict ourself to consider only a final component $Q'_i$.
	
  It remains to show $d(A_i) = k_0(A_i) = 1$.
  For this purpose, we want to describe how a naturally induced transducer of $A_i$ looks like.
  Let $\mathcal{T}_{\overline{A}_j} = (Q,\{0,\ldots,k-1\}, \delta, q_0, \Delta, \lambda)$ be a naturally induced transducer of $\overline{A}_j$.
  As $d(\overline{A}_j) k_0(\overline{A}_j) | p$ we know by \cref{th:full_G} that for all $q_1,q_2\in Q$ and $g\in G$ there exists
  $\bfw \in \{0,\ldots,k^p-1\}^{*}$ such that $\delta(q_1,\bfw) = q_2$ and $T(q_1,\bfw) = g$.
  This means, in particular, that there exists $\bfw \in \{0,\ldots,k^p-1\}^{*}$ such that $\delta(q_1,\bfw) = q_2$ and $T(q_1,\bfw) = id$.
  Thus \cref{pr:equality} shows that we can describe the states of $A_i$ as the elements corresponding to some coordinates of $\mathcal{T}_{\overline{A}_j}$:
  $\pi_i(q_1)$ and $\pi_i(q_2)$ belong to the same strongly connected component for all $q_1,q_2\in Q$ as $\bfw$ satisfies $\delta'(\pi_i(q_1),\bfw) = \pi_i(q_2)$.
  This implies that
  \begin{align*}
    Q'_i = \{\pi_{\ell}(q)|q\in Q, \ell \in I\},
  \end{align*}
  holds for some index set $I\subseteq\{1,\ldots,n_0\}$.
  We claim that $\mathcal{T} := (Q_{I},\Sigma^{p},\widetilde{\delta}_{I}, (q_0)_{I}, G_{I}, \widetilde{\lambda}_{I})$ provides a naturally induced transducer of $A_i$.
  Here we again denote by $\widetilde{\delta}$ the extension of $\delta$ from letters to words, $\widetilde{\lambda}$ coincides with $T$ for words of length $p$
  (for $\mathcal{T}_{A}$) and $G$ is defined as in \cref{th:full_G}.
  Furthermore, we use $(.)_{I}$ to denote the projection to the coordinates of $I$.
  \cref{th:full_G} assures that $\mathcal{T}_{\overline{A}}$ is an induced transducer.
  To see that $\mathcal{T}_{\overline{A}}$ is synchronizing we just have to take a synchronizing word $\bfw_{q_0}$ of $\mathcal{T}_{A}$ 
  and add leading zeros such that it is a word whose length is divisible by $p$.\\
  However Property 6) may not hold, but we already discussed in \cref{sec:aut_to_trans} that this can be fixed easily.\\
  By the construction of $\mathcal{T}$ we see easily that $d(A_i) = k_0(A_i) = 1$ holds.
\end{proof}

This Proposition can be simplified under suitable conditions.

\begin{corollary}\label{cor:aut_to_trans_schleife}
  Let $\overline{A} = (Q',\Sigma, \delta', q_0')$ be a strongly connected automaton, such that\\
  $\delta'(q_0',0) = q_0'$.
  Then there exists a strongly connected automaton $A$ such that $d(A) = k_0(A) = 1$ and
  the automatic sequences generated by $A$ and $\overline{A}$ coincide.
\end{corollary}
\begin{proof}
  We consider the proof of \cref{pr:components_d_k_0} and see directly that $A$ -- as considered in the proof of \cref{pr:components_d_k_0} -- 
  has only one strongly connected component.
  It just remains to note that the automatic sequences generated by $A$ and $\overline{A}$ coincide since adding leading zeros does not change the output.
\end{proof}

\begin{example}
We continue the example from \cref{sec:str1}.
  The naturally induced transducer we considered was
  
  \tikzset{elliptic state/.style={draw,ellipse}}
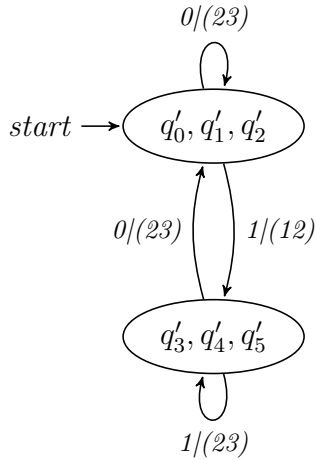
  \begin{tikzpicture}[->,>=stealth',shorten >=1pt,auto,node distance=2.8cm, semithick, bend angle = 15, cross line/.style={preaction={draw=white,-,line width=4pt}}]
    
    \node[elliptic state, initial]	(A)                    {$q_0', q_1', q_2'$};
    \node[elliptic state]	         (B) [below of=A] 	{$q_3', q_4', q_5'$};
    
    \path [every node/.style={font=\footnotesize}, pos = 0.5]
    (A) edge [loop above] node 		 	{0|(23)}      	(A)
	edge [bend left]  node 		 	{1|(12)} 	(B)
    (B) edge [bend left]  node 			{0|(23)}	(A)
	edge [loop below] node			{1|(23)}	(B);
  \end{tikzpicture}
	
and we have already seen that $d = 2, k_0 = 1$. Thus we consider the automaton that is the ''square`` of the original automaton, i.e., we concatenate two steps of the original automaton to one step of the new automaton, and the corresponding naturally induced transducer:

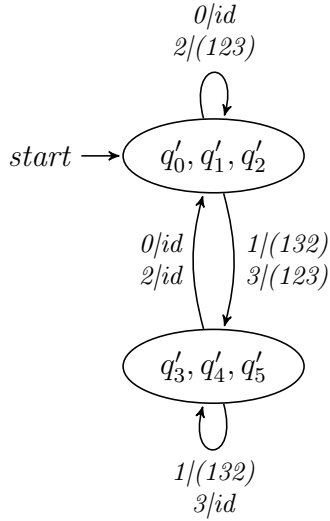
\tikzset{elliptic state/.style={draw,ellipse}}
  \begin{tikzpicture}[->,>=stealth',shorten >=1pt,auto,node distance=2.8cm, semithick, bend angle = 15, cross line/.style={preaction={draw=white,-,line width=4pt}}]
    
    \node[elliptic state, initial]	(A)                    {$q_0', q_1', q_2'$};
    \node[elliptic state]	         (B) [below of=A] 	{$q_3', q_4', q_5'$};
    
    \path [every node/.style={font=\footnotesize}, pos = 0.5]
    (A) edge [loop above] node 	[align=center]	 	{0|id\\2|(123)}      	(A)
	edge [bend left]  node 	[align=center]	 	{1|(132)\\3|(123)} 	(B)
    (B) edge [bend left]  node 	[align=center]		{0|id\\2|id}	(A)
	edge [loop below] node	[align=center]		{1|(132)\\3|id}	(B);
  \end{tikzpicture}

  We find then easily that for this transducer $d = k_0 = 1$ holds.
\end{example}

\section{Reduction of \cref{th:moebius} and \cref{th:prime}}\label{sec:reduction}
We reduce in this section the two main Theorems~\ref{th:moebius} and \ref{th:prime} to statements concerning naturally induced transducers.

\subsection{M\"obius Randomness Principle}
First we present a lemma that reduces the Sarnak conjecture for the symbolic dynamical system associated with an automatic sequence 
to a M\"obius randomness law for (possibly different) automatic sequences.

\begin{lemma}\label{le:dynamical_system_uniform}
  Suppose that for every automatic sequence $(a_n)_{n\in\N}$ with values in $\C$
  \begin{align}
    \sum_{n\leq N} \mu(n) a_{n+r} = o(N),
  \end{align}
  uniformly for $r \in \N$. Then \cref{th:moebius} holds.
\end{lemma}
\begin{proof}
  We now fix one automatic sequence $\mathbf{b} = (b_n)_{n\in\N}$ that takes values in $\C$.
  Let $(X,S)$ be the symbolic dynamical system associated with $\mathbf{b}$.
  \cite[Lemma 2.1]{synchronizing} shows that it is sufficient to show that for every $j\geq 1$ and every $g:\C^j \to \C$ it holds that
  \begin{align*}
    \sum_{n\leq N} \mu(n) g(b_{n+r},b_{n+r+1},\ldots,b_{n+r+j-1}) = o(N)
  \end{align*}
  uniformly for $r\in\N$.
  However, \cite[Theorem 5.4.4, Lemma 5.2.6 and Lemma 4.3.9]{AllouchShallit} shows that $(b_{n+\ell})_{n\in \N}$ is again an automatic sequence for all $\ell \in \N$, 
  the product of two automatic sequences is again an automatic sequence, and therefore the product of $j$ automatic sequences is again an automatic sequence, 
  and also that $(g(b_n))_{n\in \N}$ is again an automatic sequence.
	Therefore, $a_n = g(b_{n+r},b_{n+r+1},\ldots,b_{n+r+j-1})$ defines again an automatic sequence and hence the statement follows.
\end{proof}

To further reduce this result we need some ideas of representation theory.
We cover the most important definitions and notations briefly.
A $d$ dimensional representation $D$ is a continuous homomorphism $D: G \rightarrow U_d$, where we let $U_d$ denote the group of unitary $d\times d$ matrices
over $\C$. A representation $D$ is called irreducible if there exists no non-trivial subspace $V \subset U_d$ such that $D(g)V \subseteq V$ holds for all $g\in G$.
Furthermore, we will use a scalar-product for functions $f_1,f_2: G \to \C$
\begin{align*}
	<f_1,f_2> := \int f_1 \overline{f_2} d\boldsymbol{\mu} = \frac{1}{\abs{G}} \sum_{g \in G} f_1(g) \overline{f_2(g)},
\end{align*}
where $\boldsymbol{\mu}$ denotes the Haar measure.
In \cref{sec:RS1}, we show the following result.

\begin{proposition}\label{pr:rep_estimate}
  Let $\mathcal{T}_{A}$ be a naturally induced transducer by $A$ and suppose that $d(A) = 1, k_0(A) = 1$ holds.
  Let, furthermore, $D$ be a unitary irreducible representation of $G$, 
  $r,\lambda_1, \lambda_2 \in \N, 0<b <k^{\lambda_1}$ and $m<k^{\lambda_2}$.
  It holds
  \begin{align}\label{eq:goal}
    \norm{\sum_{\substack{n<N\\n\equiv m \bmod k^{\lambda_2}}} \chi_{1/k^{\lambda_1}}\rb{\frac{n+r-bk^{\nu-\lambda_1}}{k^{\nu}}}D(T(q_0,(n+r)_k)) \mu(n)}_F = o(N),
  \end{align}
  uniformly for $r$ where $\nu$ is the unique integer satisfying $k^{\nu-1}\leq N < k^{\nu}$.
\end{proposition}
\begin{remark}
  $n \equiv m \bmod k^{\lambda_2}$ obviously fixes the last $\lambda_2$ digits of $n$.
  Furthermore, $\chi_{1/k^{\lambda_1}}(.)$ fixes the digits $\nu-\lambda_1,\ldots,\nu-1$ of $n$.
\end{remark}

Our goal of this subsection is to show that \cref{pr:rep_estimate} implies \cref{th:moebius}.
%

  We fix any $\varepsilon>0$ and need to show that there exists $N_0$ such that for $N \geq N_0(\varepsilon)$ 
  \begin{align*}
    \abs{S(r,N)} \leq \varepsilon N 
  \end{align*}
  holds for all $r$, where
  \begin{align*}
    S(r,N) := \sum_{n<N} \mu(n) a_{n+r}.
  \end{align*}

Let us first present the most important ideas for this proof.
Obviously, we can only work with naturally induced transducers, when we restrict ourself to the final components of $A$.
Therefore we fix some of the first digits of $(n+r)$  which allows us to work with a strongly connected component of $Q'$.
Thereafter, we also fix the last digits of $(n+r)$ to fix $\delta(q,(n+r)_k)$.
It then remains to find suitable estimates for $T(q,.)$.

\begin{remark}
  It would be very convenient to apply \cite[Lemma 3.2]{synchronizing} to only work with strongly connected automata, but this Lemma does not provide uniformity in $r$.
\end{remark}

More precisely, \cref{pr:components_d_k_0} shows that there exists a DFAO $A = (Q',\Sigma, \delta', q_0', \tau)$ such that $\Sigma = \{0,\ldots,k-1\}, \delta'(q_0', 0) = q_0'$ 
and $a_n = \tau(\delta'(q_0',(n)_k))$.
Let $Q'_1,\ldots, Q'_{\ell}$ be the strongly connected components of $Q'$ that are closed under $\delta'$. Furthermore, let $A_i$ be the restriction of $A$ to $Q'_i$. 
\cref{pr:components_d_k_0} states that $d(A_i) = k_0(A_i) = 1$ holds.

We define the naturally induced transducers for these components by
\begin{align*}
  \mathcal{T}_{A_i} = (Q_i,\Sigma,\delta_i,\Delta_i,\lambda_i)
\end{align*}
and let $\bfw_i$ be a synchronizing word of $\mathcal{T}_{A_i}$ (we intentionally avoid to define the initial state here).

\begin{restatable}{lemma}{setDens}
\label{le:set_dens}
  The set 
	{\small
	\begin{align*}
		M:=\{n\in \N|\exists \bfv_1,\bfv_2 \in \Sigma^{*}: (n)_k = \bfv_1 \bfv_2 \wedge \forall q' \in Q': \delta'(q',\bfv_1) \in \cup Q'_i \wedge 
		    \forall i: \bfw_i\ \text{ is a subword of } \bfv_2\}
	\end{align*}
	}%
	has density $1$.
\end{restatable}
The motivation for introducing the set $M$ is that whenever you read a word corresponding to some $n \in M$, you end up in one final component $Q'_i$ -- this is achieved by $\bfv_1$.
The purpose of $\bfv_2$ is not so obvious but will be more apparent shortly.
\begin{remark}
  Let $n \in M$. By the definition of $M$ it follows that the considered properties hold also for unreduced digital representations of $n$ (i.e.,~with leading zeros).
\end{remark}
As the proof of \cref{le:set_dens} is rather technical, we postpone it to \cref{sec:technical}.

  We assume for simplicity that $\abs{a_n} \leq 1$ holds for all $n \in \N$.
  We let $\nu$ denote the unique integer satisfying $k^{\nu-1}\leq N < k^{\nu}$. 
  
  The first idea is to remove the digits of $n+r$ at positions $\nu,\nu+1,\ldots$, i.e., $(\floor{(n+r)/k^{\nu}})_k$, and restrict ourself to $(n+r) \bmod k^{\nu}$ as these integers form (at most) two intervals
  with combined length at least $k^{\nu-1}$.
  Therefore, we want to detect the digits of $n+r$ at the positions $\nu-\lambda_1,\ldots,\nu-1$ 
  for some fixed value $\lambda_1$ -- only depending on $\varepsilon$ -- that we will choose shortly:
  \begin{align*}
    \abs{\sum_{n<N} \mu(n) a_{n+r}} &= \abs{\sum_{n<N} \sum_{b<k^{\lambda_1}} \chi_{k^{-\lambda_1}} \rb{\frac{n+r-bk^{\nu- \lambda_1}}{k^{\nu}}} \mu(n) a_{n+r}}\\
    &\leq \sum_{b<k^{\lambda_1}} \abs{\sum_{n<N} \chi_{k^{-\lambda_1}} \rb{\frac{n+r-bk^{\nu-\lambda_1}}{k^{\nu}}} \mu(n) a_{n+r}}\\
    &\leq \sum_{\substack{b<k^{\lambda_1}\\b\in M}} \abs{\sum_{n<N} \chi_{k^{-\lambda_1}} \rb{\frac{n+r-bk^{\nu-\lambda_1}}{k^{\nu}}} \mu(n) a_{n+r}}
	+ \sum_{\substack{b<k^{\lambda_1}\\b \notin M}} k^{\nu-\lambda_1}.
  \end{align*}
  Choosing $\lambda_1 (< \nu)$ such that $\abs{\{b<k^{\lambda_1}| b\notin M\}} \leq \frac{\varepsilon}{3 k}k^{\lambda_1}$ gives
  \begin{align*}
    \abs{S(r,N)} &\leq \sum_{\substack{b<k^{\lambda_1}\\b\in M}} \abs{S_1(b,r,N)}
	+ \frac{\varepsilon}{3} N
  \end{align*}
  where 
  \begin{align*}
    S_1 := S_1(b,r,N) := \sum_{n<N} \chi_{k^{-\lambda_1}}\rb{\frac{n+r-bk^{\nu-\lambda_1}}{k^{\nu}}}\mu(n) a_{n+r}.
  \end{align*}

  We rewrite $r = r_1 + k^{\nu} r_2$ with $r_1 < k^{\nu}$ and $r_1,r_2 \in \N$, and split the sum over $n$ into two parts.
  We let $(n)_k^t$ denote the unique word $\bfw$ of length $t$ such that  $[\bfw]_k \equiv n \bmod k^{t}$ and find
  \begin{align*}
    \abs{S_1} &\leq \abs{\sum_{\substack{n<N\\n+r_1<k^{\nu}}} \chi_{k^{-\lambda_1}} \rb{\frac{n+r-bk^{\nu-\lambda_1}} {k^{\nu}}} \mu(n) \tau(\delta'(q'_0,(n+r)_k))}\\
      &\qquad + \abs{\sum_{k^{\nu}-r_1 \leq n < N} \chi_{k^{-\lambda_1}} \rb{\frac{n+r-bk^{\nu-\lambda_1}} {k^{\nu}}} \mu(n) \tau(\delta'(q'_0,(n+r)_k))}\\
      & = \abs{\sum_{\substack{n<N\\n<k^{\nu}-r_1}} \chi_{k^{-\lambda_1}} \rb{\frac{n+r-bk^{\nu-\lambda_1}}{k^{\nu}}} \mu(n) 
	\tau(\delta'(q'_0,(r_2)_k (b)_k^{\lambda_1} (n+r)_k^{\nu-\lambda_1}))}\\
      &\qquad + \abs{\sum_{k^{\nu}-r_1 \leq n < N} \chi_{k^{-\lambda_1}} \rb{\frac{n+r-bk^{\nu-\lambda_1}}{k^{\nu}}} \mu(n) 
	\tau(\delta'(q'_0,(r_2+1)_k (b)_k^{\lambda_1} (n+r)_k^{\nu-\lambda_1}))}.
  \end{align*}

  Thus we find
  \begin{align}\label{eq:S_1_to_S_2}
  \begin{split}
    \abs{S_1} &\leq \abs{S_2(b,r_1,\min (N,k^{\nu}-r_1),\delta'(q'_0,(r_2)_k))} + \abs{S_2(b,r_1,N,\delta'(q'_0,(r_2+1)_k))}\\
    &\qquad + \abs{S_2(b,r_1,k^{\nu}-r_1,\delta'(q'_0,(r_2+1)_k))},
  \end{split}
  \end{align}
  where
  \begin{align*}
    S_2(b,r',N',q') := \sum_{n<N'} \chi_{k^{-\lambda_1}} \rb{\frac{n+r'-bk^{\nu-\lambda_1}}{k^{\nu}}} \mu(n) \tau(\delta'(q',(b)_k^{\lambda_1} (n+r')_k^{\nu-\lambda_1}))
  \end{align*}
  for $N'+r' <k^{\nu}$.
  
  We simplify \eqref{eq:S_1_to_S_2} to
  \begin{align}\label{eq:S_1_to_S_2_simplified}
    \abs{S_1(b,r,N)} \leq 3 \max_{\substack{r',N'\in \N\\r'+N'<k^{\nu}}} \max_{q' \in Q'} \abs{S_2(b,r',N',q')}
  \end{align}

  Now we work on an estimate for $S_2$.
  We can rewrite $(b)_k^{\lambda_1} = \bfv_1 \bfv_2$ where $\delta'(q',\bfv_1) \in Q'_i$ for some $i$.
  
  We denote by $q'_1 = \delta'(q',\bfv_1)$ and $A(q'_1) = (Q'_i,\Sigma,\delta'_i,q'_1,\tau_i)$ the restriction of $A$ to $Q'_i$, 
  i.e.,~$\delta'_i = \delta'_i |_{Q'_i\times \Sigma}, \tau_i = \tau |_{Q'_i}$).
  Let $\mathcal{T}_{A(q'_1)} = (Q,\Sigma,\delta,q_0,\Delta,\lambda)$ be a naturally induced transducer of $A(q'_1)$.
  We define $M_{\lambda_2}$ as the set of integers $m$ such that $(m)_k^{\lambda_2}$ is a synchronizing words of $\mathcal{T}_{A(q'_1)}$ and note that for $[\bfw_0]_k \in M_{\lambda_2}, \bfw \in \Sigma^{*}$ it holds that
  $\delta(q_0,\bfw \bfw_0) = \delta(q_0,\bfw_0)$.
  \cite[Lemma 2.2]{synchronizing} shows that there exist $c,\eta>0$ such that 
  $k^{\lambda_2} - |M_{\lambda_2}| \leq c k^{(1-\eta) \lambda_2}$ holds for all $\lambda_2 \in \N$.
  \cref{pr:equality} shows that
  \begin{align*}
    \abs{S_2(b,r',N',q')} &= \biggl| \sum_{n<N'} \chi_{k^{-\lambda_1}} \rb{\frac{n+r'-bk^{\nu-\lambda_1}}{k^{\nu}}} \mu(n) \\
      &\qquad \tau(\pi_1(T(q_0,\bfv_2 (n+r')_k^{\nu-\lambda_1}) \delta(q_0,\bfv_2 (n+r')_k^{\nu-\lambda_1})))\biggr | \\
    &= \biggl| \sum_{m<k^{\lambda_2}} \sum_{\substack{n<N'\\n+r'\equiv m \bmod k^{\lambda_2}}} \chi_{k^{-\lambda_1}} \rb{\frac{n+r' - bk^{\nu-\lambda_1}}{k^{\nu}}} \mu(n)\\
      &\qquad \tau(\pi_1(T(q_0,\bfv_2 (n+r')_k^{\nu-\lambda_1}) \delta(q_0,\bfv_2 (n+r')_k^{\nu-\lambda_1}))) \biggr| \\
    &\leq \sum_{m\in M_{\lambda_2}} \biggl| \sum_{\substack{n<N'\\n\equiv m -r' \bmod k^{\lambda_2}}} \chi_{k^{-\lambda_1}} \rb{\frac{n+r' - bk^{\nu-\lambda_1}}{k^{\nu}}} \mu(n)\\
      &\qquad \tau(\pi_1(T(q_0,\bfv_2 (n+r')_k^{\nu-\lambda_1}) \delta(q_0,(m)_k^{\lambda_2}))) \biggr| + c k^{(1-\eta) \lambda_2} k^{\nu-\lambda_1-\lambda_2}.
  \end{align*}
  We choose $\lambda_2$ such that $c k^{-\eta \lambda_2} \leq \frac{\varepsilon}{9k}$ (and we also assume that $\lambda_1 + \lambda_2 < \nu$) and find
  \begin{align*}
    \abs{S_2(b,r',N',q')} \leq \sum_{m \leq k^{\lambda_2}} \sum_{q'_1 \in \cup Q'_i} \abs{S_3(b,r',N',m,\mathcal{T}_{A(q'_1)})} + \frac{\varepsilon N}{9} k^{-\lambda_1},
  \end{align*}
  where
  \begin{align*}
    S_3(b,r',N',m,\mathcal{T}_{A(q'_1)}) &:= \sum_{\substack{n<N'\\n\equiv m \bmod k^{\lambda_2}}} \chi_{k^{-\lambda_1}}\rb{\frac{n+r'-bk^{\nu-\lambda_1}}{k^{\nu}}} \mu(n)\\
      &\qquad \qquad \tau(\pi_1(T(q_0,\bfv_2 (n+r')_k^{\nu-\lambda_1}) \delta(q_0,(m)_k^{\lambda_2}))).
  \end{align*}
  Note that $\lambda_1$ and $\lambda_2$ are independent of $N$ and $N'$.

  We define $q_1:= \delta(q_0,(b)_k)$ and find $T(q_0,(b)_k (n+r')_k^{\nu-\lambda_1}) = T(q_0,(b)_k)\circ T(q_1,(n+r')_k^{\nu-\lambda_1})$ as well as
  $T(q_0,\bfv_2 (n+r')_k^{\nu-\lambda_1}) = T(q_0,\bfv_2) \circ T(q_1,(n+r')_k^{\nu-\lambda_1})$, since $\bfv_2$ is synchronizing. 
  This gives in total
  \begin{align*}
    T(q_0,\bfv_2 (n+r')_k^{\nu-\lambda_1}) = T(q_0,\bfv_2) \circ T(q_0,(b)_k)^{-1} \circ T(q_1,(b)_k (n+r')_k^{\nu-\lambda_1}).
  \end{align*}
  We define for $\bfw \in \Sigma^{*}$ a function $f_{\bfw,b}: G \rightarrow \C$
	\begin{align*}
		f_{\bfw,b}(\sigma) := \tau(\pi_1(T(q_0,\bfv_2) \circ (T(q_0,(b)_k)^{-1} \circ \sigma) \cdot(\delta(q_0,\bfw)))).
	\end{align*}
  Thus we find $\tau(\pi_1(T(q_0,\bfv_2 (n+r')_k^{\nu-\lambda_1}) \cdot \delta(q_0,(m)_k^{\lambda_2}))) = f_{(m)_k^{\lambda_2},b}(T(q_0,(b)_k (n+r')_k^{\nu-\lambda_1}))$.
  We denote by $\mathcal{F}(\mathcal{T}_{A(q'_1)}):= \{f_{\bfw,b}: \bfw \in \Sigma^{*},b\in \N\}$. 
  Note that $\abs{\mathcal{F}(\mathcal{T}_{A(q'_1)})}\leq \abs{Q}\cdot \abs{G}$ holds, as $\delta(q_0,\bfw)$ can take at most $\abs{Q}$ and $T(q_0,\bfv_2) \circ T(q_0,(b)_k)^{-1}$
  can take at most $\abs{G}$ values.
  We find

  \begin{align*}
    S_3(b,r',N',m,\mathcal{T}_{A(q'_1)}) = \sum_{\substack{n<N'\\n\equiv m \bmod k^{\lambda_2}}} \chi_{k^{-\lambda_1}}\rb{\frac{n+r'-bk^{\nu-\lambda_1}}{k^{\nu}}} 
	\mu(n) f_{(m)_k^{\lambda_2},b}(T(q_0,(n+r')_k)).
  \end{align*}

  With $T(n) := T(q_0,(n)_k)$ we find

  \begin{align*}
    \abs{S_3(b,r',N',m,\mathcal{T}_{A(q'_1)})} \leq \max_{f\in\mathcal{F}(\mathcal{T}_{A(q'_1)})} \abs{\sum_{\substack{n<N'\\n\equiv m \bmod k^{\lambda_2}}} 
	\chi_{k^{-\lambda_1}}\rb{\frac{n+r'-bk^{\nu-\lambda_1}}{k^{\nu}}} \mu(n) f(T(n+r'))}.
  \end{align*}

%

  Consider a finite group $G$. It is well-known that there only exist finitely many equivalence classes of unitary and irreducible representations of $G$ 
  (see for example \cite[Part I, Section 2.5]{rep_finite}). The Peter-Weyl Theorem (see for example \cite[Chapter 4, Theorem 1.2]{representations}) 
  states that the entry functions of irreducible representations (suitably renormalized) form a orthonormal basis of $L^2(G)$.
  Thus we can express any function $f:G\to \C$ by these $M_0$ entry functions:
  
  Let $D^{(m')} = (d_{ij}^{(m')})_{ij}$ be representations of the equivalence classes mentioned above and $f: G\to\C$.
  Then, there exist $c_{\ell}$ such that
  \begin{align*}
    f(g) = \sum_{\ell < M_0} c_{\ell} d_{i_{\ell}j_{\ell}}^{(m_\ell)}(g)
  \end{align*}
  holds for all $g\in G$. Thus we find
  \begin{align*}
    \abs{S_3} &\leq \max_{D \in \mathcal{D}_{q'_1}} c \norm{\sum_{\substack{n<N'\\n\equiv m \bmod k^{\lambda_2}}} \chi_{k^{-\lambda_1}} 
	\rb{\frac{n+r'-bk^{\nu-\lambda_1}}{k^{\nu}}} \mu(n) D(T(n+r'))}_F,
  \end{align*}
  where $\norm{.}_F$ denotes the Frobenius norm, $c$ is the maximum of all possible values of $c_{\ell} \cdot M_0$ and $\mathcal{D}_{q'_1}$ is 
  the finite set of all possible appearing representations for $\mathcal{F}(\mathcal{T}_{A(q'_1)})$.
  
  If we denote by $T_{q'_1}$ the function $T$ for $\mathcal{T}_{A(q'_1)}$ we find in total
	{\small
  \begin{align*}
    \abs{\sum_{n<N} \mu(n) a_{n+r}} &\leq \frac{2}{3} \varepsilon N + c \sum_{b<k^{\lambda_1}} \sup_{\substack{r',N'\in \N\\N'+r'<k^{\nu}}} \sum_{m<k^{\lambda_2}}\\
      &\qquad  \sum_{q'_1\in \cup Q'_i} \sum_{D \in \mathcal{D}_{q'_1}} \norm{\sum_{\substack{n<N'\\n\equiv m \bmod k^{\lambda_2}}} \chi_{k^{-\lambda_1}} 
	\rb{\frac{n+r'-bk^{\nu-\lambda_1}}{k^{\nu}}} \mu(n) D(T_{q'_1}(n+r'))}_F.
  \end{align*}
	}%
  Note that $\mathcal{D}_{q'_1}$ is finite and only depends on the automatic sequence $\mathbf{a}$.

  \begin{align*}
    \abs{S(r,N)} &\leq \frac{2}{3} \varepsilon N + c \sum_{b<k^{\lambda_1}} \sum_{m<k^{\lambda_2}} \sum_{q'_1 \in \cup Q'_i} \sum_{D \in \mathcal{D}_{q'_1}}\\
      &\qquad  \sup_{r'\in \N} \max_{N' < k^{\nu}} \norm{\sum_{\substack{n<N'\\n\equiv m \bmod k^{\lambda_2}}} \chi_{k^{-\lambda_1}} 
	\rb{\frac{n+r'-bk^{\nu-\lambda_1}}{k^{\nu}}} \mu(n) D(T_{q'_1}(n+r'))}_F,
  \end{align*}
  where $\lambda_1,\lambda_2$ and $\mathcal{D}_{q'_1}$ only depend on $\varepsilon$.
  This finally proves the following proposition, which recapitulates the results of this subsection.
  \begin{proposition}
    Assume \cref{pr:rep_estimate} holds. Then \cref{th:moebius} holds.
  \end{proposition}

\subsection{Prime Number Theorem}
We use the same notations as in the previous section and some similar ideas.
We have seen that representations play an important role.
For the prime number theorem we have to distinguish some special representations.\\
Recall that $G_{\ell} := \{g \in G: s_0(g) = \ell\}$ for $\ell = 0,\ldots,d'(A)-1$.

\begin{restatable}{lemma}{Dl}\label{le:Dl}
  There exist $d'$ representations $D_0,\ldots,D_{d'-1}$ defined by
  \begin{align*}
    D_{\ell}(g) := \e\rb{\frac{\ell \cdot s_0(g)}{d'}},
  \end{align*}
  for $\ell = 0,\ldots d'-1$, where $s_0$ is defined by $s_0(T(q,\bfw)) \equiv [\bfw]_k \bmod d'$.
\end{restatable}
\begin{proof}
  The proof follows directly by \cref{th:dk_0} and \cref{th:full_G}.
\end{proof}
\begin{remark}
  These representations play an important role in this section.
\end{remark}

Note that $D_{\ell}(T(q_0,\bfw)) = \e\rb{(\ell \cdot [\bfw]_k) / d'}$ for $\ell = 0,\ldots, d'-1, \bfw \in \Sigma^{*}$.

In \cref{sec:RS1} we will prove the following Proposition~\ref{pr:prime_Fourier} and in \cref{subsec:dist_T_primes} deduce from it the more technical \cref{le:dist_in_G} below.
The rest of this subsection is devoted to prove that this lemma implies \cref{th:prime}.

\begin{proposition}\label{pr:prime_Fourier}
  Let $D$ be a unitary, irreducible representation of $G$ different from $D_{\ell}$. 
  Then there exists some $\eta >0$ such that
  \begin{align*}
    \norm{\frac{1}{k^{\nu}} \sum_{p<k^{\nu}} D(T(q_0,(p)_k))\e(-pt)}_2 \ll k^{-\nu \eta}
  \end{align*}
  holds uniformly in $t\in \R$.
\end{proposition}

\begin{restatable}{lemma}{distG}
\label{le:dist_in_G}
  Let $\lambda \in \N$, $\mathcal{T}$ be a naturally induced transducer with function $T$.
  For every $g\in G$ there exists $f_g$ such that
  \begin{align*}
    \frac{1}{\pi(x;a,k^{\lambda})} \sum_{\substack{p<x\\p\equiv a \bmod k^{\lambda}}} \ind_{[T(q_0,(p)_k) = g]} = f_g + o(1) \text{ for } x\to\infty
  \end{align*}
  holds (uniformly) for all $a<k^{\lambda}$ such that $(a,k) = 1$.
\end{restatable}

As mentioned before, the rest of this subsection is devoted to prove
\begin{proposition}\label{pr:prime_reduce}
	Suppose \cref{pr:prime_Fourier} holds. Then \cref{th:prime} holds.
\end{proposition}
\begin{proof}
We start as in the previous subsection and find by using \cref{cor:aut_to_trans_schleife} and \cref{pr:equality} that
\begin{align*}
  \frac{1}{\pi(x)}\sum_{p\leq x} \ind_{[a_p = b]} &= 
	  \frac{1}{\pi(x)} \rb{\sum_{\substack{a\leq k^{\lambda}\\(a,k^{\lambda}) = 1}}\quad \sum_{\substack{p\leq x\\ p\equiv a \bmod k^{\lambda}}} \ind_{[a_p = b]} + O(k)}\\
	  &= \frac{1}{\pi(x)} \rb{\sum_{\substack{a\leq k^{\lambda}\\(a,k^{\lambda}) = 1}} \quad \sum_{\substack{p\leq x\\ p\equiv a \bmod k^{\lambda}}} 
	      \ind_{[\tau(\pi_1(T(q_0,(p)_k)\cdot \delta(q_0,(p)_k))) = b]} + O(k)},
\end{align*}
where the error term arises from primes that are not co-prime to $k^{\lambda}$ and therefore not co-prime to $k$.

We note that $(a,k^{\lambda})=1$ holds if and only if $(\varepsilon_0(a),k) = 1$ where $\varepsilon_0(x)$ denotes the least significant digit of $x$ in base $k$.
We denote again by $M_{\lambda}:= \{n< k^{\lambda}: (n,k) = 1 \text{ and } (n)_k \text{ is synchronizing}\}$.
One finds easily that $|\{n\leq k^{\lambda}: (n,k) = 1, n\notin M_{\lambda}\}| = O(k^{\lambda(1-\eta)})$ for some $\eta >0$, as in the previous section.

We fix $\lambda$ for now to find the necessary estimates. Thereafter we let $\lambda$ grow. 
This then gives the desired result.
\begin{align*}
  \frac{1}{\pi(x)}\sum_{p\leq x} \ind_{[a_p = b]} &=
      \frac{1}{\pi(x)} \sum_{a \in M_{\lambda}} \sum_{\substack{p\leq x\\ p\equiv a \bmod k^{\lambda}}} 
	      \ind_{[\tau(\pi_1(T(q_0,(p)_k)\cdot \delta(q_0,(p)_k))) = b]}\\
	      &\qquad \qquad + \frac{\pi(x;1,k^{\lambda})}{\pi(x)} O(k^{\lambda(1-\eta)}) + O\rb{\frac{k^{\lambda}}{\pi(x)}}\\
      &= \frac{1}{\pi(x)} \sum_{a \in M_{\lambda}} \sum_{\substack{p\leq x\\ p\equiv a \bmod k^{\lambda}}} 
	      \ind_{[\tau(\pi_1(T(q_0,(p)_k)\cdot \delta(q_0,(a)_k))) = b]} + O(k^{-\lambda\eta}) + O\rb{\frac{k^{\lambda}}{\pi(x)}}\\
      &= \frac{1}{\pi(x)} \sum_{q\in Q} \sum_{\substack{a \in M_{\lambda}\\ \delta(q_0,(a)_k) = q}} \sum_{\substack{p\leq x\\ p\equiv a \bmod k^{\lambda}}} 
	      \ind_{\tau([\pi_1(T(q_0,(p)_k)\cdot q)) = b]} + O(k^{-\lambda\eta}) + O\rb{\frac{k^{\lambda}}{\pi(x)}}\\
      &= \frac{1}{\pi(x)} \sum_{q\in Q} \sum_{\substack{a<k^{\lambda}\\(a,k) = 1}} \ind_{[\delta(q_0,(a)_k) = q]} \sum_{\substack{p\leq x\\ p\equiv a \bmod k^{\lambda}}} 
	    \ind_{[\tau(\pi_1(T(q_0,(p)_k)\cdot q)) = b]}\\
	    &\qquad \qquad + O(k^{-\lambda\eta}) + O\rb{\frac{k^{\lambda}}{\pi(x)}}.    
\end{align*}

Thus we are interested in the distribution of $T(q_0,(p)_k)$ along primes in arithmetic progressions, which is covered by \cref{le:dist_in_G}.

By writing
\begin{align*}
  \sum_{\substack{p\leq x\\ p\equiv a \bmod k^{\lambda}}} \ind_{[\tau(\pi_1(T(q_0,(p)_k)\cdot q)) = b]} = 
	\sum_{g\in G} \ind_{[\tau(\pi_1(g\cdot q)) = b]} \sum_{\substack{p\leq x\\ p\equiv a \bmod k^{\lambda}}} \ind_{[T(q_0,(p)_k) = g]}
\end{align*}

we find that
\begin{align*}
  \sum_{\substack{p\leq x\\ p\equiv a \bmod k^{\lambda}}} \ind_{[\tau(\pi_1(T(q_0,(p)_k)\cdot q)) = b]} = \pi(x;1,k^{\lambda}) f_{q,b} + o(\pi(x;1,k^{\lambda}))
\end{align*}
holds, where
\begin{align*}
  f_{q,b} = \sum_{g \in G} f_g \ind_{[\tau(\pi_1(g \cdot q)) = b]}.
\end{align*}

Therefore, we are interested in the quantity 
\begin{align*}
\frac{1}{\varphi(k^{\lambda})}\abs{\{a <k^{\lambda}: (\varepsilon_0(a),k)=1, \delta(q_0,(a)_k) = q\}}.
\end{align*}
We want to show that it has a limit (for $\lambda \to \infty$). 
We start by rewriting it as
\begin{align*}
  &\frac{1}{\varphi(k)}\sum_{\substack{a_0 <k\\ (a_0,k)=1}} \frac{1}{k^{\lambda-1}}\abs{\{a <k^{\lambda-1}: \delta(q_0,(a\cdot k + a_0)_k) = q\}}\\
  &\qquad = \frac{1}{\varphi(k)} \sum_{\substack{a_0 <k\\ (a_0,k)=1}} \sum_{q'\in Q} \ind_{[\delta(q',a_0) = q]} 
		      \frac{1}{k^{\lambda-1}} \abs{\{a<k^{\lambda-1}: \delta(q_0,(a)_k) = q'\}}.
\end{align*}

Thus it is sufficient to show that $\frac{1}{k^{\lambda-1}} |\{a<k^{\lambda-1}: \delta(q_0,(a)_k) = q'\}|$ has a limit.
However, $\mathcal{T}$ is synchronizing and strongly connected and therefore primitive. 
It is a well-known result that the densities exist in this setting (see for example \cite{AllouchShallit}).
This allows us to write
\begin{align*}
  |\{a <k^{\lambda}: (\varepsilon_0(a),k)=1, \delta(q_0,(a)_k) = q\}| = \varphi(k^{\lambda}) f_q + o(\varphi(k^{\lambda})),
\end{align*}
where $f_q$ does not depend on $\lambda$.
Thus we find in total
\begin{align*}
  \frac{1}{\pi(x)} \sum_{p\leq x} \ind_{[a_p = b]} &= \frac{1}{\pi(x)} \sum_{q\in Q} (\varphi(k^{\lambda}) f_q + o(\varphi(k^{\lambda}))) 
  (\pi(x;1,k^{\lambda}) f_{q,b} + o(\pi(x;1,k^{\lambda}))) + O(k^{-\lambda\eta})\\
      &= \sum_{q\in Q} f_q f_{q,b} + \frac{1}{\pi(x;1,k^{\lambda})} o(\pi(x;1,k^{\lambda})) + \frac{1}{\varphi(k^{\lambda})} o(\varphi(k^{\lambda})) + O(k^{-\lambda \eta}).
\end{align*}

To show that the sum of the three error terms is smaller than $\varepsilon$, we choose $\lambda$ such that the second and third error term are bounded by $\varepsilon/3$.
Then we find for $x$ large enough that the first error term (for this given $\lambda$) is also bounded by $\varepsilon/3$ 
(note that $k^{\lambda-1}\leq \varphi(k^{\lambda})\leq k^{\lambda}$).
Consequently we find
\begin{align*}
  \lim_{x\to\infty}\frac{1}{\pi(x)} \sum_{p\leq x} \ind_{[a_p = b]} = \sum_{q\in Q} f_q f_{q,b},
\end{align*}
which -- finally -- finishes the proof of \cref{pr:prime_reduce}.

\end{proof}

\begin{remark}
  The error terms can be made explicit. We find for example that the error term in \cref{le:dist_in_G} is actually of the form $O(k^{-\lambda \eta'})$.
  The dominant error term seems to correspond to the distribution of primes in arithmetic progressions.
\end{remark}

\begin{remark}
	As mentioned earlier, \cref{th:prime} covers block-additive functions. 
	However, to find the corresponding result for block-additive functions the group structure degenerates to an additive structure and many arguments simplify
	(e.g.,~ one does not need to use group representations).
\end{remark}

\section{A general M\"obius Principle}\label{sec:RS1}
The goal of this section is to show \cref{pr:rep_estimate} and \cref{pr:prime_Fourier}.
To prove these results we use a generalization of the results in \cite{mauduit_rivat_rs}. 
We start this section by stating the altered definitions and results.
In \cref{sec:Fourier} we show that for all but very special representations \cref{def:2} is fulfilled by $f(n) = D(T(n+r))$, uniformly in $r$.
In \cref{sec:carry} we show that for every representation \cref{def:1} is fulfilled by $f(n) = D(T(n+r))$, uniformly in $r$.
We leave the full proofs of the generalized results of \cite{mauduit_rivat_rs} to \cref{sec:technical}, as they are technical and very similar to the original proofs.

We consider a function $f:\N \to U_d$ where $U_d$ denotes the set of unitary $d\times d$ matrices.
Let $k\in\N$.
We let $f_{\lambda}$ denote the $k^{\lambda}$-periodic function defined by
\begin{align*}
  \forall n \in \{0,\ldots,k^{\lambda}-1\},\quad \forall m\in \Z, f_{\lambda}(n+m k^{\lambda}) = f(n).
\end{align*}
Furthermore, we define
$f_{\mu,\lambda}(n) := f_{\lambda}(n) f_{\mu}(n)^{H}$ for $\mu\leq \lambda$.

It is necessary to use matrix valued functions instead of complex valued functions as in \cite{mauduit_rivat_rs}, as we are working with representations.
\begin{definition}\label{def:1}
  A function $f:\N\rightarrow U_d$ has the carry property if there exists $\eta >0$ such that uniformly for $(\lambda,\alpha, \rho)\in \N^3$
  with $\rho<\lambda$, the number of integers $0\leq \ell <k^{\lambda}$ such that there exists $(n_1,n_2) \in \{0,\ldots,k^{\alpha}-1\}^2$ with
  \begin{align}\label{eq:carry_violation}
    f(\ell k^{\alpha} + n_1 + n_2)^{H} f(\ell k^{\alpha} + n_1) \neq f_{\alpha + \rho}(\ell k^{\alpha} + n_1 + n_2)^{H} f_{\alpha + \rho}(\ell k^{\alpha} + n_1)
  \end{align}
  is at most $O(k^{\lambda-\eta \rho})$ where the implied constant may depend only on $k$ and $f$.
\end{definition}
\begin{remark}
	One can obviously exchange \eqref{eq:carry_violation} with
	\begin{align*}
		f(\ell k^{\alpha} + n_1)^{H} f(\ell k^{\alpha} + n_1 + n_2) \neq f_{\alpha + \rho}(\ell k^{\alpha} + n_1)^{H} f_{\alpha + \rho}(\ell k^{\alpha} + n_1 + n_2).
	\end{align*}
\end{remark}
We introduce a set of functions with uniformly small Discrete Fourier Transforms as in \cite{mauduit_rivat_rs}:
\begin{definition}\label{def:2}
  Given a non-decreasing function $\gamma: \R \rightarrow \R$ satisfying $\lim_{\lambda \rightarrow \infty} \gamma(\lambda) = +\infty$ and $c>0$ we denote by 
  $F_{\gamma, c}$ the set of functions $f:\N\rightarrow U_d$ such that for $(\alpha,\lambda)\in \N^2$ with $\alpha \leq c \lambda$ and $t\in \R$:
  \begin{align}
    \norm{k^{-\lambda} \sum_{u<k^{\lambda}} f(u k^{\alpha}) \e(-ut)}_F \leq k^{-\gamma(\lambda)}.
  \end{align}

\end{definition}
One obvious difference to \cite{mauduit_rivat_rs} is that we consider matrix valued functions.
Moreover, the original definition of the carry property requires $\eta = 1$.
Nevertheless, we find results similar to those obtained in \cite{mauduit_rivat_rs} in this more general setting.
The results of \cite{mauduit_rivat_rs} have already been generalized to matrix valued functions in \cite{drmota2014}
and to a weaker carry property (although still more restrictive than \cref{def:1}) in \cite{hanna}.
We provide a full proof for the following theorems in \cref{sec:technical}.
\begin{theorem}\label{thm:prime}
  Let $\gamma:\R\to\R$ be a non-decreasing function satisfying $\lim_{\lambda\to\infty}\gamma(\lambda) = + \infty$, and $f:\N\to U_d$ be a function satisfying
  \cref{def:1} for some $\eta \in (0,1]$ and $f\in \mathcal{F}_{\gamma,c}$ for some $c\geq 10$ in \cref{def:2}.
  Then for any $\theta \in \R$ we have
  \begin{align}
    \norm{\sum_{n\leq x} \Lambda(n)f(n)\e(\theta n)} \ll c_1(k) (\log x)^{c_2(k)} x k^{-\eta \gamma(2\floor{(\log x)/(80 \log k)}/20},
  \end{align}
  with the same constants as in \cite{mauduit_rivat_rs}.
\end{theorem}
\begin{theorem}\label{thm:mobius}
  Let $\gamma:\R\to\R$ be a non-decreasing function satisfying $\lim_{\lambda\to\infty}\gamma(\lambda) = + \infty$, and $f:\N\to U_d$ be a function satisfying
  \cref{def:1} for some $\eta \in (0,1]$ and $f\in \mathcal{F}_{\gamma,c}$ for some $c\geq 10$ in \cref{def:2}.
  Then for any $\theta \in \R$ we have
  \begin{align}
    \norm{\sum_{n\leq x} \mu(n)f(n)\e(\theta n)} \ll c_1(k) (\log x)^{\frac{9}{4} + \frac{1}{4}\max(\omega(k),2)} x k^{-\eta \gamma(2\floor{(\log x)/(80 \log k)}/20},
  \end{align}
  with the same constants as in \cite{mauduit_rivat_rs}.
\end{theorem}

\subsection{Fourier Estimates}\label{sec:Fourier}
One of the most difficult parts of this approach is to find sharp enough bounds for the Fourier terms.
The goal of this subsection is to show that \cref{def:2} holds for all but very special unitary irreducible representations where $f(n) = D(T(n+r))$ and 
$\gamma(\lambda) = \eta' \lambda - c'$ for some $\eta' >0$ and $c' \in \R$.
The proof is rather technical. However, it justifies some results of \cref{sec:induced_trans}.

Let $\mathcal{T}_{A}$ be a naturally induced transducer of $A$ and suppose that $d(A) = 1, k_0(A) = 1$ holds.
We remind the reader of the following special representations.

\Dl*

Note that $D_{\ell}(T(q_0,\bfw)) = \e\rb{(\ell \cdot [\bfw]_k) / d'}$ for $\ell = 0,\ldots, d'-1, \bfw \in \Sigma^{*}$.

We try to find exponentially decreasing bounds for expressions similar to
\begin{align*}
  \norm{\frac{1}{k^{\lambda}}\sum_{u<k^{\lambda}}D(T(q_0,(u)_k))\e(-ut)}_2,
\end{align*}
uniformly in $t$, where $D$ is a unitary, irreducible representation of $G$.

We already see that $D = D_{\ell}$ is a special case for which we are not able to find exponentially decreasing bounds as
{\small
\begin{align*}
  \frac{1}{k^{\lambda}}\sum_{u<k^{\lambda}}D_{\ell}(T(q_0,(u)_k))\e(-ut) = \frac{1}{k^{\lambda}} \sum_{u<k^{\lambda}} \e\rb{u\rb{\frac{\ell}{d'}-t}}
	\end{align*}
	}%
gives $1$ for $t = \ell/d'$, which is the trivial bound.
Nevertheless, we are able to find exponentially decreasing bounds for all other unitary and irreducible representations of $G$:

\begin{theorem}\label{thm:Fourier}
  Let $D$ be a unitary and irreducible representation of $G$ different from $D_0,\ldots,D_{d'-1}$.
  Then there exists $\eta >0$ such that
  \begin{align}\label{eq:eta}
    \norm{\frac{1}{k^{\lambda}}\sum_{u<k^{\lambda}}D(T(q,(uk^{\alpha}+r)_k))\e(-ut)}_2 \ll k^{-\eta \lambda}
  \end{align}
  holds uniformly for $t \in \R$, $q\in Q, r\in \N$ and $\alpha \in \N$.
\end{theorem}
The proof is carried out throughout this subsection.
We define
\begin{align*}
  \phi_{\lambda,\alpha}^{q}(t,r) := \frac{1}{k^{\lambda}} \sum_{u<k^{\lambda}} D(T(q,(uk^{\alpha}+r)_k)) \e(-ut)
\end{align*}
and for $r<k^{\alpha}$
\begin{align*}
  \psi_{\lambda,\alpha}^{q}(t,r) := \frac{1}{k^{\lambda}} \sum_{u<k^{\lambda}} D(T(q,(u)_k^{\lambda} (r)_k^{\alpha})) \e(-ut).
\end{align*}

We see that these two definitions look very similar, but it turns out to be much easier to deal with $\psi$.

\begin{restatable}{lemma}{phiPsi}
  Let $D$ be a unitary representation of $G$. Assume that there exists $\eta >0$ such that
  \begin{align*}
    \norm{\psi_{\lambda,\alpha}^{q}(t,r)}_2 \ll k^{-\eta \lambda}
  \end{align*}
  holds uniformly for $q \in Q,\lambda,\alpha \in \N, t \in \R$ and $r <k^{\alpha}$.
  Then \cref{thm:Fourier} holds for $D$.
\end{restatable}
\begin{proof}
	We postpone this proof to \cref{sec:technical}.
\end{proof}

Therefore, it is sufficient to prove the following proposition.
\begin{proposition}\label{pr:psi}
  Let $D$ be a unitary and irreducible representation of $G$ different from $D_0,\ldots,D_{d'-1}$.
  Then there exists $\eta >0$ such that 
  \begin{align*}
    \frac{1}{k^{\lambda}} \norm{\sum_{u<k^{\lambda}}D(T(q,(u)_k^{\lambda} (r)_k^{\alpha})) \e(-ut)}_2 \ll k^{-\eta \lambda}
  \end{align*}
  holds uniformly for $q \in Q, \lambda, \alpha \in \N, t\in \R$ and $r<k^{\alpha}$.
\end{proposition}

We use the recursive structure of transducers to find recurrences for these Fourier terms.

\begin{lemma}\label{le:phi_rec}
  Let $q \in Q$, $\lambda,\alpha \in \N,t \in \R$ and $r<k^{\alpha}$.
  It holds for all $m < \lambda$
  \begin{align}\label{eq:phi_rec}
    \psi_{\lambda,\alpha}^{q}(t,r) = \frac{1}{k^m}\sum_{\varepsilon < k^m} D(T(q,(\varepsilon)_k^{m}))\e(-\varepsilon (k^{\lambda-m}t)) \psi_{\lambda-m,\alpha}^{\delta(q,\varepsilon)}(t,r).
  \end{align}
\end{lemma}
\begin{proof}
  We know that $T(q,(\varepsilon)_k^{m}(u')_k^{\lambda-m}(r)_k^{\alpha}) = T(q,(\varepsilon)_k^{m})\cdot T(\delta(q,(\varepsilon)_k^m),(u')_k^{\lambda-m} (r)_k^{\alpha})$ for 
  $u' <k^{\lambda-m}$.
  By distinguishing the $m$ most significant digits of $u$, we find 
	{\small
  \begin{align*}
    \psi_{\lambda, \alpha}^{q}(t) &= \frac{1}{k^{\lambda}} \sum_{\varepsilon < k^m}\sum_{u' < k^{\lambda-m}} D(T(q,(\varepsilon)_k^{m}(u')_k^{\lambda-m}(r)_k^{\alpha})) 
	  \e(-(\varepsilon k^{\lambda-m}+u')t)\\
    &= \frac{1}{k^{m}} \sum_{\varepsilon < k^m} D(T(q,(\varepsilon)_k^{m})) \e(-\varepsilon (k^{\lambda-m}t)) \\
    &\qquad \frac{1}{k^{\lambda-m}} \sum_{u'<k^{\lambda-m}} D(T(\delta(q,(\varepsilon)_k^{m}),(u')_k^{\lambda-m}(r)_k^{\alpha}))\e(-u't).
  \end{align*}
	}%
\end{proof}

The main idea is to find $m \in \N, \varepsilon_1,\varepsilon_2<k^m$ such that two or more terms on the right side of \eqref{eq:phi_rec} cancel -- at least partially.
This means that we want to find $\varepsilon_1,\varepsilon_2$ such that
$\delta(q,\varepsilon_1) = \delta(q,\varepsilon_2)$ and 
\begin{align*}
	\norm{D(T(q,\varepsilon_1))\e(-\varepsilon_1 k^{\lambda-m} t) + D(T(q,\varepsilon_2))\e(-\varepsilon_2 k^{\lambda-m} t)}_2 < 2.
\end{align*}
We split the proof into two parts -- depending on the dimension of $D$.
\begin{lemma}\label{le:psi_d>1}
  Let $D$ be a unitary and irreducible representation of $G$ of dimension at least $2$. Then \cref{pr:psi} holds for $D$.
\end{lemma}
\begin{proof}
  We need to show that there exists $\eta>0$ such that \cref{pr:psi} holds.
  \cref{th:full_G} shows that there exists for every $g \in G$ a path $\bfw_g \in \Sigma^{m_0}$ such that $\delta(q,\bfw_g) = q, T(q,\bfw_g) = g$ holds.
  We find using \cref{le:phi_rec} that
  \begin{align*}
    \norm{\psi_{\lambda,\alpha}^{q}(t,r)}_2 &= \frac{1}{k^{m_0}}\norm{\sum_{\varepsilon < k^{m_0}} D(T(q,(\varepsilon)_k^{m_0}))\e(-\varepsilon (k^{\lambda-m_0}t)) 
      \psi_{\lambda-m_0,\alpha}^{\delta(q,(\varepsilon)_k^{m_0})}(t,r)}_2\\
    &\leq \frac{1}{k^{m_0}} \norm{\sum_{g \in G} D(T(q, \bfw_g)) \e(-[\bfw_g]_k k^{\lambda-m_0}t) 
      \psi_{\lambda-m_0,\alpha}^{q}(t,r)}_2\\
    &\qquad + \frac{1}{k^{m_0}} \sum_{\substack{\varepsilon<k^{m_0}\\ \varepsilon \neq [\bfw_g]_k \forall g \in G}} \norm{D(T(q,(\varepsilon)_k^{m_0})) 
	      \e(-\varepsilon (k^{\lambda - m_0} t))
      \psi_{\lambda-m_0,\alpha}^{\delta(q,(\varepsilon)_k^{m_0})}(t,r)}_2\\
    &\leq \frac{1}{k^{m_0}} \norm{\sum_{g \in G} D(T(q,\bfw_g)) \e(-[\bfw_g]_k k^{\lambda-m_0} t)}_2 \cdot \norm{\psi_{\lambda-m_0,\alpha}^{q}(t,r)}_2\\
    &\qquad + \frac{k^{m_0} - |G|}{k^{m_0}} \max_{q'\in Q} \norm{\psi_{\lambda-m_0,\alpha}^{q'}(t,r)}_2\\
    &\leq \frac{1}{k^{m_0}} \rb{k^{m_0} - |G| + \norm{\sum_{g\in G}D(g) \e(-[\bfw_g]_k k^{\lambda-m_0} t)}_2} \max_{q'\in Q} \norm{\psi_{\lambda-m_0,\alpha}^{q'}(t,r)}_2.
  \end{align*}

  Thus it is sufficient to show that for all $t'\in \R$ holds 
  \begin{align}\label{eq:ineq_d>2}
    \norm{\sum_{g\in G}D(g) \e(-[\bfw_g]_k t')}_2 < |G|.
  \end{align}
  Since the left side is a periodic and continuous function in $t'$, this implies that there exists $\eta'>0$ such that
  \begin{align*}
    \norm{\sum_{g\in G}D(g) \e(-[\bfw_g]_k t')}_2 \leq |G| - \eta'
  \end{align*}
  for all $t'\in \R$.
  This gives in total
  \begin{align*}
    \norm{\psi_{\lambda,\alpha}^{q}(t,r)}_2 \leq \rb{1-\frac{\eta'}{k^{m_0}}} \max_{q'\in Q} \norm{\psi_{\lambda-m_0,\alpha}^{q'}(t,r)}_2
  \end{align*}
  and the statement follows easily as $\eta'$ only depends on $G$ and $D$ (but not on $\lambda,\alpha, t'$ or $r$).\\
  What follows is a variation of the proof of \cite[Lemma 4]{invertible}:
  
  Let us assume -- on the contrary -- that there exists $t' \in \R$ such that 
  \begin{align*}
    \norm{\sum_{g\in G}D(g) \e(-[\bfw_g]_k t')}_2 = |G|.
  \end{align*}
  This holds if and only if there exists $\mathbf{0} \neq \mathbf{y} \in \C^{d}$ such that
  \begin{align*}
    \norm{\sum_{g\in G}D(g) \e(-[\bfw_g]_k t') \mathbf{y}}_2^2 &= \sum_{g_1,g_2\in G} \langle D(g_1)\e(-[\bfw_{g_1}]_k t') \mathbf{y}, D(g_2) \e(-[\bfw_{g_2}]_k t') \mathbf{y} \rangle = |G|^2 \norm{\mathbf{y}}_2^2
  \end{align*}
  However, the Cauchy-Schwarz inequality implies 
  \begin{align}\label{eq:cs}
  \begin{split}
    &\abs{\langle D(g_1)\e(-[\bfw_{g_1}]_k t') \mathbf{y}, D(g_2) \e(-[\bfw_{g_2}]_k t') \mathbf{y} \rangle} \\
	&\qquad \qquad \leq \norm{D(g_1)\e(-[\bfw_{g_1}]_k t') \mathbf{y}}_2^2 \norm{D(g_2) \e(-[\bfw_{g_2}]_k t') \mathbf{y}}_2^2 = \norm{\mathbf{y}}_2^2.
  \end{split}
  \end{align}
  For equality to hold in \eqref{eq:cs} it is necessary that the $D(g_i)\e(-[\bfw_{g_i}]_k t')$ are linearly dependent.
  Since we find for $g_1=g_2 = id$ the summand $\langle \e(-[\bfw_{id}]_k t') \mathbf{y},\e(-[\bfw_{id}]_k t') \mathbf{y} \rangle$ we obtain for all $g\in G$
  \begin{align*}
    D(g) \e(-[\bfw_{g}]_k t') \mathbf{y} = \e(-[\bfw_{id}]_k t') \mathbf{y},
  \end{align*}
  i.e.,~$\mathbf{y}$ is an eigenvector of all $D(g), g \in G$.
  We define $W = span(\mathbf{y})$ and find $D(g) W \subseteq W$ for all $g \in G$.
  This means that $D$ would be reducible which yields a contradiction. 
  
  Therefore \eqref{eq:ineq_d>2} holds, which concludes this proof.
  \end{proof}

\begin{lemma}\label{le:psi_d=1}
  Let $D$ be a one-dimensional, unitary and irreducible representation of $G$ different from $D_{0},\ldots,D_{d'-1}$. Then \cref{pr:psi} holds for $D$.
\end{lemma}
\begin{proof}
  
  Our goal is to show that there exists some $\eta'>0$ (only depending on $G$ and $D$) such that
  \begin{align}\label{eq:ineq_d=1}
    \norm{\psi_{\lambda,\alpha}^{q}(t,r)}_2 \leq (1-\eta') \max_{\widetilde{q}\in Q} \norm{\psi_{\lambda-m_0'- \ell_0,\alpha}^{\widetilde{q}}(t,r)}_2
  \end{align}
  holds for all $t\in\R$.
  This implies again \cref{pr:psi} for some $\eta > 0$ as in the proof of the previous case.

We need two different estimates for this step and start to work on the first estimate.\\
  We start by using the properties of $G$, i.e.,~$\{g \in G: s_0(g) = \ell\} = G_0 \cdot (g_0')^{\ell}$, to find restrictions for $D$.\\
  Assume that $D(g) = 1$ for all $g \in G_0$.
  It follows that $D(g_0') = \e(\ell/d')$ for some $\ell < d'$ since $(g_0')^{d'} \in G_0$
  and therefore we see directly that $D = D_{\ell}$.
  Thus we know that there exists $g \in G_0$ such that $D(g) \neq 1$.\\
  We use \cref{le:m_0} to see that there exist $\bfw_{id},\bfw_{g}\in \Sigma^{m_0'}$ such that $\delta(q,\bfw_{id}) = \delta(q,\bfw_{g}) = q$ and 
  $T(q,\bfw_{id}) = id, T(q,\bfw_g) = g$ holds.
  By using \cref{le:phi_rec} (with $m = m_0'$), we find (as in the proof for representations of dimension $\geq 2$)
  \begin{align*}
    \norm{\psi_{\lambda,\alpha}^{q}(t,r)}_2 \leq \frac{1}{k^{m_0'}} &\left \| D(T(q, \bfw_{id})) \e(-[\bfw_{id}]_k k^{\lambda - m_0'} t) 
		\psi_{\lambda - m_0',\alpha}^{\delta(q,\bfw_{id})}(t,r)\right.\\
		& \quad + \left. D(T(q,\bfw_{g})) \e(-[\bfw_{g}]_k k^{\lambda - m_0'} t) 
		\psi_{\lambda - m_0',\alpha}^{\delta(q,\bfw_{g})}(t,r)\right\|_2\\
		& + \frac{k^{m_0'}-2}{k^{m_0'}} \max_{q'\in Q} \norm{\psi_{\lambda-m_0',\alpha}^{q'}(t,r)}_2.
  \end{align*}
  We see that the first term of the right side equals
  \begin{align*}
    &\frac{1}{k^{m_0'}} \norm{1\cdot \e(-[\bfw_{id}]_k k^{\lambda-m_0'}t) + D(g) \cdot \e(-[\bfw_{g}]_k k^{\lambda-m_0'}t)}_2 \norm{\psi_{\lambda-m_0',\alpha}^{q}(t,r)}_2\\
    &\quad = \frac{1}{k^{m_0'}} \norm{1 + D(g) \e(-([\bfw_{g}]_k-[\bfw_{id}]_k) k^{\lambda-m_0'}t)}_2 \norm{\psi_{\lambda-m_0',\alpha}^{q}(t,r)}_2.
  \end{align*}
  
	Now we use \cref{le:phi_rec} again with $m= \ell_0$ to find in total
	{\small
  \begin{align}\label{eq:estimate_1}
    \norm{\psi_{\lambda, \alpha}^{q}(t,r)}_2 \leq \frac{k^{m_0'}-2 + \abs{1+D(g)\e(-([\bfw_{g}]_k-[\bfw_{id}]_k) k^{\lambda-m_0'}t)}}{k^{m_0'}} \max_{\widetilde{q}\in Q} 
	  \norm{\psi_{\lambda-\ell_0-m_0',\alpha}^{\widetilde{q}}(t,r)}_2.
  \end{align}
	}%

    To find the second estimate, we use \cref{le:phi_rec} with $m = \ell_0$ and find
  \begin{align*}
    \norm{\psi_{\lambda,\alpha}^{q}(t,r)}_2 \leq \max_{\widetilde{q}\in Q}\norm{\psi_{\lambda-\ell_0,\alpha}^{\widetilde{q}}(t,r)}_2.
  \end{align*}
  For convenience we assume that the maximal value of the right side is attained at $q$.
    \cref{th:dk_0} shows $d'\cdot d''(q,q) = gcd\{[\bfw_1]_k -[\bfw_2]_k: \bfw_i \in \Sigma^{m_0'} \text{ with } \delta(q,\bfw_i) = q, T(q,\bfw_i) = id\}$.
    Therefore, there exist $N \in \N, \mu_{i,j}\in \Z^{N}\times \Z^{N}$ such that $d'\cdot d''(q,q) = \sum_{i,j<N} \mu_{i,j} ([\bfw_i]_k - [\bfw_j]_k)$ 
    where $\bfw_i,\bfw_j \in \Sigma^{m_0'}$ such that $\delta(q,\bfw_i) = \delta(q,\bfw_j) = q$ and $T(q,\bfw_i) = T(q,\bfw_j) = id$.
    By \cref{le:phi_rec} and the same arguments we used above, we find
    \begin{align*}
      \norm{\psi_{\lambda-\ell_0,\alpha}^{q}(t,r)}_2 &\leq \frac{1}{k^{m_0'}} \norm{\sum_{i < N} D(T(q,\bfw_{i})) \e(-[\bfw_{i}]_k k^{\lambda-\ell_0-m_0'} t) \psi_{\lambda-\ell_0-m_0',\alpha}^{q}(t,r)}_2\\
	&\qquad + \frac{k^{m_0'}-N}{k^{m_0'}} \max_{\widetilde{q} \in Q} \norm{\psi_{\lambda-\ell_0-m_0',\alpha}^{\widetilde{q}}(t,r)}_2\\
	&= \frac{1}{k^{m_0'}} \abs{\sum_{i<N} \e(-[\bfw_{i}]_k k^{\lambda-\ell_0-m_0'}t)} \norm{\psi_{\lambda-\ell_0-m_0',\alpha}^{q}(t,r)}_2\\
	&\qquad + \frac{k^{m_0'}-N}{k^{m_0'}} \max_{\widetilde{q}\in Q} \norm{\psi_{\lambda-\ell_0-m_0',\alpha}^{\widetilde{q}}(t,r)}_2.
    \end{align*}
  Thus, we find in total
  \begin{align}\label{eq:estimate_2}
    \norm{\psi_{\lambda,\alpha}^{q}(t,r)}_2 \leq \frac{k^{m_0'}-N + \abs{\sum_{i<N} \e(-[\bfw_{i}]_k k^{\lambda-\ell_0-m_0'}t)}}{k^{m_0'}} \max_{\widetilde{q}\in Q} 
	\norm{\psi_{\lambda-\ell_0-m_0',\alpha}^{\widetilde{q}}(t,r)}_2.
  \end{align}

  We combine \eqref{eq:estimate_1} and \eqref{eq:estimate_2} and find
	{\small
  \begin{align*}
    \norm{\psi_{\lambda,\alpha}^{q}(t,r)}_2 &\leq \rb{1-\frac{2-\abs{1+D(g)\e(-([\bfw_{g}]_k-[\bfw_{id}]_k) k^{\lambda-m_0'}t)} + N - \abs{\sum_{i<N} \e(-[\bfw_{i}]_k k^{\lambda-\ell_0-m_0'}t)}}{2 k^{m_0'}}}\\
	    &\qquad \qquad \max_{\widetilde{q}\in Q}\norm{\psi_{\lambda-\ell_0-m_0',\alpha}^{\widetilde{q}}(t,r)}_2.
  \end{align*}
	}%

  It remains to show that
  \begin{align}\label{eq:not_equal}
    \abs{1+D(g)\e(-([\bfw_{g}]_k-[\bfw_{id}]_k)k^{\lambda-m_0'}t)} + \abs{\sum_{i<N}\e(-[\bfw_i]_k k^{\lambda-\ell_0-m_0'}t)} < N+2
  \end{align}
  holds for all $t\in \R$ as the left side is a periodic and continuous function in $t$.
  We distinguish the following two cases.\\
  At first let us assume $d't k^{\lambda - m_0'} \in \Z $:\\
  Since $g \in G_0$ we know that $[\bfw_{id}] \equiv [\bfw_{g}] \bmod d'$ and, therefore, the first term of the left side of equation \eqref{eq:not_equal} 
  simplifies to $\abs{1+D(g)}$.
  By the definition of $g$, we find 
  that $\abs{1+D(g)}<2$ and, therefore, equation \eqref{eq:not_equal} holds.\\
  The remaining case is $d't k^{\lambda - m_0'} \notin \Z $:\\
  Let us assume
  \begin{align*}
    \abs{\sum_{i<N} \e(-[\bfw_i]_k k^{\lambda-\ell_0-m_0'} t)} = N.
  \end{align*}
  This implies $([\bfw_{i}]_k - [\bfw_{j}]_k) k^{\lambda-\ell_0-m'_0}t \in \Z$ for all $i,j<N$.
  As $d' \cdot d''(q_0,q_0)$ is a linear combination of $([\bfw_i]_k-[\bfw_j]_k)$, we find $d' \cdot d''(q,q) k^{\lambda-\ell_0-m_0'} t \in \Z$.
  This yields a contradiction since $d''(q,q) | k^{\ell_0}$.
  Therefore, we have
  \begin{align*}
    \abs{\sum_{i<N} \e(-[\bfw_i]_k k^{\lambda-\ell_0 - m_0'} t)} < N
  \end{align*}
  and, consequently, equation \eqref{eq:not_equal} holds, which finishes the proof.
\end{proof}

\subsection{Carry Lemma}\label{sec:carry}
We show in this subsection that the function $f(n) = D(T(q_0,(n+r)_k))$ has the carry property -- uniformly in $r$.
\begin{lemma}
  \cref{def:1} holds -- uniformly in $r$ -- for $f (n) = D(T(q_0,(n+r)_k))$ where $D$ is a unitary and irreducible representation of $G$, $\eta$ is given by \cite[Lemma 2.2]{synchronizing}, i.e., $\eta = \frac{\log(k^{\ell_0}/(k^{\ell_0}-1))}{\log(k^{\ell_0})}$,
  and the implied constant does not depend on $r$.
\end{lemma}
\begin{proof}
  We fix $\lambda, \rho$ and $\alpha$. We rewrite $r = r_1 k^{\alpha} + r_2$ where $r_1,r_2 \in \N, r_2 <k^{\alpha}$.
  We want to distinguish the form of $\ell + r_1$, i.e.,~the maximal number $t$ such that the digits at position $\alpha,\ldots,\alpha+t$ 
  of $(\ell+r_1) k^{\alpha} + n_1 + n_2 +r_2$ can be affected by the carry of $n_1+n_2+r_2$.
  As $n_1+n_2+r_2\leq 3k^{\alpha}-3$ limits the carry to $0,1,$ or $2$, we define $t:= \max (\nu_k(\ell+r_1+1),\nu_k(\ell+r_1+2))$.
  Thus we know that the digits at position $\alpha + t +i$ of $(\ell k^{\alpha} + n_1 + n_2 + r)$ and $(\ell k^{\alpha} + n_1 + r)$ coincide for $i = 1,2,\ldots$ 
  and are equal to the digits at position $t + i$ of $\ell+r_1$.
  We fix $t$ such that $t<\rho$ and count the number of integers $\ell$ such that \eqref{eq:carry_violation} holds for some $n_1,n_2$.
  As the terms corresponding to the digits of index $\alpha + t + i$ cancel for $i\geq 1$, we find
  \begin{align*}
    &D(T(q_0,(\ell k^{\alpha} + n_1 + n_2 + r)_k))^{H} D(T(q_0,(\ell k^{\alpha} + n_1 + r)_k)) \\
    &\quad = D(T(\delta(q_0,(\floor{(\ell + r_1)/k^{t}})_k),(\ell k^{\alpha} + n_1 + n_2 + r)_k^{\alpha+t}))^{H} \\
    &\qquad \qquad D(T(\delta(q_0,(\floor{(\ell + r_1)/k^{t}})_k),(\ell k^{\alpha} + n_1 + r)_k^{\alpha + t}))
  \end{align*}
  and
  \begin{align*}
    &D(T_{\alpha + \rho}(q_0,(\ell k^{\alpha} + n_1 + n_2 + r)_k))^{H} D(T_{\alpha + \rho}(q_0,(\ell k^{\alpha} + n_1 + r)_k)) \\
    &\quad = D(T(\delta(q_0,(\floor{(\ell + r_1)/k^{t}})_k^{\rho-t},(\ell k^{\alpha} + n_1 + n_2 + r)_k^{\alpha+t}))^{H} \\
    &\qquad \qquad D(T(\delta(q_0,(\floor{(\ell + r_1)/k^{t}})_k^{\rho-t}),(\ell k^{\alpha} + n_1 + r)_k^{\alpha + t})).
  \end{align*}
  Thus, the number of integers $\ell$ (with fixed $t$ and $r$) for which $\eqref{eq:carry_violation}$ holds for some $n_1,n_2$ can be bounded by the number
  of integers $\ell$ such that $((\ell+r_1)/k^{t})_k^{\rho-t}$ is not synchronizing -- 
  otherwise $\delta(q_0,(\floor{(\ell+r_1)/k^{t}})_k) = \delta(q_0,(\floor{(\ell+r_1)/k^{t}})_k^{\rho-t})$ would imply that \eqref{eq:carry_violation} does not hold.
  By \cite[Lemma 2.2]{synchronizing}, we know that this number is bounded by $O(k^{(\rho-t)(1-\eta)} k^{\lambda-\rho})$ for some $\eta >0$ where the constant only depends on $k$ and $A$.
  Note that there are only two possibilities for the digits with index $0,\ldots,t-1$.
  By summing over $t$ we find that the number of such $\ell$ is bounded by
  \begin{align*}
    c \sum_{t\leq \rho} k^{\lambda-\rho} k^{(\rho-t)(1-\eta)} + k^{\lambda-\rho} &\leq c k^{\lambda-\eta \rho} \sum_{t\leq \rho} k^{-t(1-\eta)} + k^{\lambda-\rho}\\
    &\leq c k^{\lambda- \eta \rho} \frac{1}{1-k^{-(1-\eta)}} + k^{\lambda- \rho}
  \end{align*}
  and the result follows easily as $c$ does not depend on $\lambda,\rho, \alpha$ and $r$.
\end{proof}

\subsection{Proof of \cref{pr:rep_estimate} and \cref{pr:prime_Fourier}}
We are now ready to show \cref{pr:rep_estimate} and, thereafter, \cref{pr:prime_Fourier}.

The proof for \cref{pr:rep_estimate} differs when $D = D_{\ell}$ and we start by considering this specific case.

\begin{lemma}\label{le:D_l_mue}
  For all $\ell < d'$, $\lambda_1,\lambda_2,r \in \N$ and $b < k^{\lambda_1},m<k^{\lambda_2}$ we have 
  \begin{align*}
    \norm{\sum_{\substack{n<N\\n\equiv m \bmod k^{\lambda_2}}} \chi_{k^{-\lambda_1}}\rb{\frac{n+r-bk^{\nu-\lambda_1}}{k^{\nu}}} D_{\ell} (T(q_0,(n+r)_k)) \mu(n)}_F = o(N)
  \end{align*}
  uniformly in $r\in \N$.
\end{lemma}
\begin{proof}
  We find by the characterizing property of $D_{\ell}$ (i.e.,~$D_{\ell}(T(q_0,(n)_k)) = \e(\ell n/d')$)
  \begin{align*}
    &\norm{\sum_{\substack{n<N\\n\equiv m \bmod k^{\lambda_2}}} \chi_{k^{-\lambda_1}}\rb{\frac{n+r-bk^{\nu-\lambda_1}}{k^{\nu}}} D_{\ell}(T(q_0,(n+r)_k)) \mu(n)}_F\\
    &\qquad = \abs{\sum_{\substack{n<N\\n\equiv m \bmod k^{\lambda_2}}} \chi_{k^{-\lambda_1}}\rb{\frac{n+r-bk^{\nu-\lambda_1}}{k^{\nu}}} \e\rb{(n+r)\frac{\ell}{d'}} \mu(n)}\\
    &\qquad \leq \sum_{\substack{m'<k^{\lambda_2}d'\\m'\equiv m\bmod k^{\lambda_2}}} \abs{\e\rb{m'\frac{\ell}{d'}}}
	\abs{\sum_{n<N} \chi_{k^{-\lambda_1}}\rb{\frac{n+r-bk^{\nu-\lambda_1}}{k^{\nu}}} \ind_{[n+r\equiv m' \bmod k^{\lambda_2}d']} \mu(n)}.
  \end{align*}
  As
  \begin{align*}
    \sum_{n<N} \chi_{k^{-\lambda_1}}\rb{\frac{n+r-bk^{\nu-\lambda_1}}{k^{\nu}}}
  \end{align*}
  is indeed a sum over at most two intervals included in $[1,\ldots,N-1]$, the result follows from the well-known result
  \begin{align*}
    \sum_{\substack{n<N\\n\equiv r \bmod s}} \mu(n) = o(N).
  \end{align*}
\end{proof}

\begin{proof}[Proof of \cref{pr:rep_estimate}:]
  For $D = D_{\ell}$ \cref{le:D_l_mue} gives the desired result.
  Suppose from now on $D \neq D_{\ell}$ for all $\ell <d'$.
  We rewrite the left side of \eqref{eq:goal} using exponential sums and obtain using Vaaler-approximation (see for example \cite[Lemma 1]{mauduit_rivat_rs})
  {\small
  \begin{align*}
    &\norm{\sum_{n<N} \frac{1}{k^{\lambda_2}} \sum_{h<k^{\lambda_2}} \e\rb{h\frac{n-m}{k^{\lambda_2}}} \sum_{n<N} \chi_{k^{-\lambda_1}}\rb{\frac{n+r-bk^{\nu-\lambda_1}}{k^{\nu}}}
	D(T(n+r)) \mu(n)}_F\\
    &\quad \leq \frac{1}{k^{\lambda_2}} \sum_{h<k^{\lambda_2}} \norm{\sum_{n<N} \e\rb{n\frac{h}{k^{\lambda_2}}} A_{k^{-\lambda_1},H}\rb{\frac{n+r-bk^{\nu-\lambda_1}}{k^{\nu}}} D(T(n+r)) \mu(n)}_F\\
    &\qquad + \frac{1}{k^{\lambda_2}} \sum_{h<k^{\lambda_2}} \sum_{n<N} B_{k^{-\lambda_1},H}\rb{\frac{n+r-bk^{\nu-\lambda_1}}{k^{\nu}}} \norm{\e\rb{n\frac{h}{k^{\lambda_2}}} 
    D(T(n+r))\mu(n)}_F\\
    &\quad \ll \frac{1}{k^{\lambda_2}} \sum_{h<k^{\lambda_2}} \norm{\sum_{n<N} \e\rb{n\frac{h}{k^{\lambda_2}}} \sum_{\abs{h'}\leq H} a_{h'}(k^{-\lambda_1},H) 
	\e\rb{h'\frac{n+r-bk^{\nu-\lambda_1}}{k^{\nu}}} D(T(n+r)) \mu(n)}_F\\
    &\qquad + \sum_{n<N} \sum_{\abs{h'}\leq H} b_{h'}(k^{-\lambda_1},H) 
	\e\rb{h'\frac{n+r-bk^{\nu-\lambda_1}}{k^{\nu}}},
  \end{align*}
  }%
  where the implied constant only depends on the dimension of $D$ and $\abs{a_{h'}(k^{-\lambda_1},H)} \leq \min\rb{k^{-\lambda_1}, \frac{1}{\pi \abs{h'}}}, 
      \abs{b_{h'}(k^{-\lambda_1},H)} \leq \frac{1}{H+1}$ holds. We assume that $k^{\lambda_1} \leq H \leq k^{\nu}$.
      
  Thus, the first summand above is bounded by
  \begin{align*}
    &\frac{1}{k^{\lambda_2}} \sum_{h<k^{\lambda_2}} \sum_{\abs{h'}< H} \abs{a_{h'}(k^{-\lambda_1},H)} 
	\norm{\sum_{n<N} \e\rb{n \rb{\frac{h}{k^{\lambda_2}} + \frac{h'}{k^{\nu}}}} D(T(n+r)) \mu(n)}_F\\
    &\qquad \ll \abs{\sum_{h'<k^{\lambda_1}} k^{-\lambda_1} + \sum_{k^{\lambda_1}<h'<H} \frac{1}{h'}} \sup_{\theta \in \R} \norm{\sum_{n<N} \e(\theta n) D(T(n+r)) \mu(n)}_F\\
    &\qquad \ll (1+\log(H)) \sup_{\theta \in \R} \norm{\sum_{n<N} \e(\theta n) D(T(n+r)) \mu(n)}_F
  \end{align*}
  However, as $D(T(n+r))$ fulfills \cref{def:1} and \cref{def:2} uniformly in $r$, we can apply \cref{thm:mobius}.
  This gives for some $\eta'>0$ the asymptotic estimate $(1+\log(H)) N^{1-\eta'}$.

  The second summand is positive and can be estimated by classical estimates involving geometric series (see for example \cite{mauduit_rivat_rs}).
  \begin{align*}
    &\sum_{\abs{h'}\leq H} b_{h'}(k^{-\lambda_1},H) 
	\sum_{n<N} \e\rb{h'\frac{n+r-bk^{\nu-\lambda_1}}{k^{\nu}}} \\
    &\quad \leq  \sum_{\abs{h'}\leq H} \frac{1}{H} \abs{\sum_{n<N} \e\rb{h'\frac{n}{k^{\nu}}}}\\
    &\quad \leq  \frac{1}{H} \sum_{\abs{h'}\leq k^{\nu}} \min\rb{N, \abs{\sin(\pi h' k^{-\nu})}}\\
    &\quad \ll \frac{1}{H} (N + k^{\nu} \log(k^{\nu}))\\
    &\quad \ll \frac{N \log(N)}{H},
  \end{align*}
  where the constant is absolute.
  
  Thus choosing for example $H = \floor{N^{1/2}}$ finishes the proof.
\end{proof}

\begin{proof}[Proof of \cref{pr:prime_Fourier}:]
By classical partial summation one finds the following well-known result (see for example \cite[Lemma 11]{mauduit_rivat_2010}).

If $f:\N \to \C$ is such that $\abs{f(n)}\leq 1$ for all $n \in \N$ then
\begin{align}
  \abs{\sum_{p\leq x} f(p)} \leq \frac{2}{\log x} \max_{t\leq x} \abs{\sum_{n\leq t} \Lambda(n) f(n)} + O(\sqrt{x}).
\end{align}
This result can easily be adapted to matrix valued functions.
Thus we find
\begin{align*}
  \norm{\frac{1}{\pi(k^{\nu})} \sum_{p<k^{\nu}} D(T(q_0,(p)_k)) \e(pt)}_F &\ll \frac{1}{\log k^{\nu}} \frac{1}{\pi(k^{\nu})} \max_{t\leq k^{\nu}} \norm{\sum_{n\leq t} \Lambda(n) D(T(q_0,(n)_k)) \e(nt)}_{F}\\
	&\qquad + O(k^{-\nu/2})
\end{align*}

As $D(T(n))$ fulfills \cref{def:1} and \cref{def:2}, we can apply \cref{thm:prime} and find that the right side is asymptotically bounded by
\begin{align*}
  \max_{t\leq k^{\nu}} t^{-\eta'} + O(k^{-\nu/2})
\end{align*}
for some $\eta'>0$ -- if we only choose $\eta'$ strictly smaller than $\eta$ to remove the other terms.
This gives directly the desired result.
\end{proof}

\section{Technical Results}\label{sec:technical}
In this section, we will provide the proofs for some results mentioned earlier, where the proof is technical and seemed to disturb the flow of reading.
\subsection{Number of words with special properties}
The first result stated that some properties occur ``often'' in the digital representation of numbers.
We considered an automaton with final components $Q'_i$.
Furthermore, we let $\bfw_i$ denote the synchronizing word for the naturally induced transducer corresponding to the automaton that is the restriction of the original automaton to $Q'_i$.
\setDens*
\begin{proof}
  One sees easily that it suffices to show that
  \begin{align*}
    \lim_{\nu\to\infty} \frac{M \cap [0,k^{2\nu})}{k^{2\nu}} = 1
  \end{align*}
  holds.
	Without loss of generality we assume that $\bfw_1$ has no leading zeros.
  We define $M_1 := \{n \in \N| \forall q' \in Q': \delta'(q',(n)_k) \in \cup Q'_i\}$, $M_{2,\nu} := \{n<k^{\nu}| \bfw_1\bfw_2\ldots\bfw_{\ell} \text{ is a subword of } (n)_k^{\nu}\}$ and $M_2 := \lim_{\nu\to\infty} M_{2,\nu}$.
	The idea is to show that both sets have density one.
  
  Noting that $\abs{(M \cap [0,k^{2 \nu}))} \geq \abs{ (M_1 \cap [0,k^{\nu}))} \cdot \abs{(M_{2,\nu})}$ gives then immediately the desired result.
	
	Thus, it just remains to show that $M_1$ and $M_2$ have density $1$.\\
  Therefore we define $M_{1,q'}:= \{n \in \N| \delta'(q',(n)_k) \in \cup Q'_i\}$ and see that $M_1 = \cap M_{1,q'}$.
  Thus it suffices to show that $M_{1,q'}$ has density $1$.
  
  One finds easily that for each $q'_1 \in Q'$ there exists $\bfw_{q'_1}\in\{0,\ldots,k-1\}^{*}$ such that $\delta'(q'_1,\bfw_{q'_1})$ belongs to $Q'_i$ for some $i$.
  We take $m$ to be the maximum of these lengths. As each $Q'_i$ is closed under $\delta'$ we may assume that each of these paths $\bfw_{q'}$ is exactly of length $m$.
  We split $Q'$ into two different parts, i.e.,~$\overline{Q'} = \cup Q_i, \widetilde{Q'} = Q'\setminus \overline{Q'}$.
  We show that 
  \begin{align}\label{eq:number_paths_strongly_connected}
    \abs{ \{\bfw \in \{0,\ldots,k-1\}^{n}: \delta'(q',\bfw) \in \widetilde{Q'}\}} \leq k^{m} k^{n(1-\eta)}
  \end{align}
  where $\eta$ is defined by
  $\eta = 1- \log(k^{m}-1)/\log(k^{m})>0$ so that
  \begin{align*}
    k^{m}-1 = k^{m(1-\eta)}.
  \end{align*}
  We show \eqref{eq:number_paths_strongly_connected} by induction on $n$.
  The statement is trivial for $n\leq m$.
  We find for $n > m$ (as $\overline{Q'}$ is closed under $\delta'$)
  \begin{align*}
    &| \{\bfw \in \Sigma^{n}: \delta'(q',\bfw)\in \widetilde{Q'}\}| = \\
      &\qquad =\sum_{q'_1 \in \widetilde{Q'}} |\{\bfw_1 \in \Sigma^{n-m}: \delta'(q',\bfw) = q'_1\}|\cdot
	    |\{\bfw_2 \in \Sigma^{m}: \delta'(q'_1,\bfw_2)\in \widetilde{Q'}\}|\\
      &\qquad \leq \sum_{q'_1 \in \widetilde{Q'}} |\{\bfw_1 \in \Sigma^{n-m}: \delta'(q',\bfw) = q'_1\}|\cdot (k^{m}-1)\\
      &\qquad = |\{\bfw \in \Sigma^{n-m}: \delta'(q',\bfw) \in \widetilde{Q'}\}|\cdot k^{m(1-\eta)} \leq k^{m} k^{(n-m)(1-\eta)} k^{m(1-\eta)}.
  \end{align*}
	
	\cite[Lemma 2.2]{synchronizing} shows 
  \begin{align*}
    \lim_{\nu \to \infty} \frac{M_{2,\nu}}{k^{\nu}} = 1,
  \end{align*}
  which finishes the proof.
\end{proof}

\subsection{Limiting distribution of $T$ along primes}\label{subsec:dist_T_primes}
The following proof shows that\\ 
$(T(q_0,(p)_k))_{p\in \mathcal{P}(a,k^{\lambda})}$ has a limiting distribution.

Therefore, we need the following very important result from representation theory.
\begin{lemma}
  Let $G$ be a compact group and $\boldsymbol{\nu}$ a regular normed Borel measure in $G$. Then a sequence $(x_n)_{n \geq 0}$ is $\boldsymbol{\nu}$-uniformly distributed
  in $G$, i.e., $\frac{1}{N} \sum_{n < N}\delta_{x_n} \rightarrow \boldsymbol{\nu}$, if and only if 
  \begin{align}\label{eq:rep_eq}
    \lim_{N\rightarrow \infty}\frac{1}{N} \sum_{n<N} D(x_n) = \int_{G} D d\boldsymbol{\nu}
  \end{align}
  holds for all irreducible unitary representations $D$ of $G$.
\end{lemma}
A proof is given for example in \cite{representations}.

\distG*
\begin{proof}
The sequence we are concerned with is $(T(q_0,(p)_k))_{p \in \mathcal{P}(a,k^{\lambda})}$.

We start by computing the left side of \eqref{eq:rep_eq} for $D = D_{\ell}$ for our specific sequence. Recall that $d'|k-1$ holds.

\begin{align*}
  &\lim_{x\rightarrow \infty}\frac{1}{\pi(x;a,k^{\lambda})} \sum_{\substack{p<x\\p\equiv a \bmod k^{\lambda}}} D_{\ell}(T(q_0,(p)_k))\\
     &\quad = \lim_{x\rightarrow \infty}\frac{1}{\pi(x;a,k^{\lambda})} \sum_{\substack{p<x\\p\equiv a \bmod k^{\lambda}}} \e(\ell p/d') \\
     &\quad = \lim_{x\rightarrow \infty}\frac{1}{\pi(x;a,k^{\lambda})} \sum_{b<d'} \e(\ell b/d') \sum_{\substack{p<x\\p\equiv a \bmod k^{\lambda}\\ p \equiv b \bmod d'}} 1\\
     &\quad = \frac{1}{\varphi(d')} \sum_{\substack{b<d'\\(b,d')=1}}  \e(\ell b/d')\\
     &\quad = \frac{1}{\varphi(d')} S(\ell,0;d').
\end{align*}
Here $S(a,b;c)$ denotes the Kloosterman (or Ramanujan) sum, defined by
\begin{align*}
  S(a,b;c) := \sum_{\substack{x<c\\(x,c) = 1}} \e\rb{\frac{ax + b\bar{x}}{c}},
\end{align*}
where $\bar{x}$ denotes the multiplicative inverse of $x$ modulo $c$.

Furthermore, we want to show that 
{\small
\begin{align*}
  \lim_{x\rightarrow \infty}\frac{1}{\pi(x;a,k^{\lambda})} \sum_{\substack{p<x\\p\equiv a \bmod k^{\lambda}}} D(T(q_0,(p)_k)) = 0
\end{align*}
}%
holds for $D\notin \{D_{\ell}: \ell = 0,\ldots,d'-1\}$.
We use a standard technique of analytic number theory to restrict ourself to sums with $x = k^{\nu}$.
Let $\nu$ be the unique integer such that $k^{\nu-1}\leq x < k^{\nu}$. We find (see for example~\cite{GrahamKolesnik})
{\small
\begin{align*}
  \norm{\frac{1}{\pi(x;a,k^{\lambda})}\sum_{\substack{p\leq x\\ p\equiv a \bmod k^{\lambda}}} D(T(q_0,(p)_{k}))}_F \ll 
  \nu \sup_{\theta \in \R} \norm{\frac{1}{\pi(x;a,k^{\lambda})}\sum_{\substack{p\leq k^{\nu}\\ p\equiv a \bmod k^{\lambda}}}\e(\theta p) D(T(q_0,(p)_{k}))}_F .
\end{align*}
}%

We can rewrite the condition $p \equiv a \bmod k^{\lambda}$ using exponential sums and find that the right side is bounded by
\begin{align*}
  &\nu \frac{1}{\pi(k^{\nu-1};a,k^{\lambda})} \sup_{\theta \in \R} \norm{\frac{1}{k^{\lambda}} \sum_{h<k^{\lambda}} \sum_{p\leq k^{\nu}} \e\rb{h\frac{p-a}{k^{\lambda}}} \e(\theta p) D(T(q_0,(p)_{k}))}_F\\
  &\qquad \leq \nu \frac{1}{\pi(k^{\nu-1};a,k^{\lambda})} \sup_{\theta \in \R} \max_{h<k^{\lambda}} \norm{\sum_{p\leq k^{\nu}} \e\rb{p\rb{\frac{h}{k^{\lambda}}+\theta}} D(T(q_0,(p)_{k}))}_F\\
  &\qquad \leq \nu \frac{1}{\pi(k^{\nu-1};a,k^{\lambda})} \sup_{\theta' \in \R} \norm{\sum_{p \leq k^{\nu}} \e(p \theta') D(T(q_0,(p)_{k}))}_F.
\end{align*}

Thus we find -- by using \cref{pr:prime_Fourier} -- that
\begin{align*}
  \lim_{x\rightarrow \infty}\frac{1}{\pi(x;a,k^{\lambda})} \sum_{\substack{p<x\\p\equiv a \bmod k^{\lambda}}} D(T(q_0,(p)_k)) 
	=\left\{\begin{array}{cl} \frac{1}{\varphi(d')} S(\ell,0;d'), & \mbox{for }D = D_{\ell} \\ 0, & \mbox{otherwise,} \end{array}\right.
\end{align*}

We define the function $f$ by
\begin{align*}
  f(g) &:= 1 + \frac{1}{\varphi(d')} \overline{S(1,0;d')} D_{1}(g) + \ldots + \frac{1}{\varphi(d')} \overline{S(d'-1,0;d')} D_{d'-1}(g)\\
      &= \sum_{\ell = 0}^{d'-1} \frac{1}{\varphi(d')} \overline{S(\ell,0;d')} D_{\ell}(g).
\end{align*}
We recall that $D_{\ell}(g) := \e\rb{\frac{s_0(g) \ell}{d'}}$ and, therefore, we find for $v < d', D_{\ell}(T(nd'+v)) = \e\rb{\frac{v \ell}{d'}}$.
By using this fact, we can determine $f(T(nd',v))$.
\begin{align*}
  f(T(nd'+v)) &= \frac{1}{\varphi(d')} \sum_{\ell = 0}^{d'-1} \overline{S(\ell,0;d')} \e\rb{\frac{\ell v}{d'}}\\
      &= \frac{1}{\varphi(d')} \sum_{\ell = 0}^{d'-1} \sum_{\substack{x<d'\\(x,d')=1}} \e\rb{-\frac{\ell x}{d'}} \e\rb{\frac{\ell v}{d'}}\\
      &= \frac{1}{\varphi(d')} \sum_{\substack{x<d'\\(x,d')=1}} \sum_{\ell = 0}^{d'-1} \e\rb{\ell \frac{v-x}{d'}}\\
      &= \frac{1}{\varphi(d')} \sum_{\substack{x<d'\\(x,d')=1}} d' \ind_{[x=v]}.
\end{align*}

In total we find
\begin{align*}
  f(T(nd'+v)) =\left\{\begin{array}{cl} \frac{d'}{\varphi(d')}, & \mbox{for }(v,d') = 1\\ 0, & \mbox{otherwise,} \end{array}\right. 
\end{align*}
i.e.,~$f$ is a positive function.
We can actually rewrite this by using \cref{th:dk_0}:
\begin{align*}
	f(g) =\left\{\begin{array}{cl} \frac{d'}{\varphi(d')}, & \mbox{for }(s_0(g),d') = 1\\ 0, & \mbox{otherwise,} \end{array}\right. .
\end{align*}
Thus we can define 
\begin{align*}
  d\boldsymbol{\nu} = f d\boldsymbol{\mu},
\end{align*}
where $\boldsymbol{\mu}$ denotes the Haar-measure of $G$.

Our goal is to show that $(T(q_0,(p)_k))_{p\in \mathcal{P}(a,k^{\lambda})}$ is $\boldsymbol{\nu}$-uniformly distributed in $G$. 

Let $\{D^{\alpha} = (d_{ij}^{\alpha})_{i,j\leq n_{\alpha}}, \alpha \in \mathcal{A}\}$
be a complete set of pairwise inequivalent irreducible unitary representations and set
$e_{ij}^{\alpha}(g) = \sqrt{n_{\alpha}}d_{ij}^{\alpha}(g)$. 
Note that $A$ is finite as we are actually considering a finite group $G$.
Recall that the set $\{e_{ij}^{\alpha}\}$ forms a complete orthonormal system in the Hilbert space $L^{2}(G)$. 
We obtain for $D_{\ell}$ that
\begin{align*}
  \int_{G} D_{\ell} f d\boldsymbol{\mu} = \frac{1}{\varphi(d')} \sum_{m<d'}S(m,0;d') <D_{\ell}|D_{m}> = \frac{1}{\varphi(d')} S(\ell,0;d').
\end{align*}
For all other representations $D^{\alpha} = (d_{ij}^{\alpha})_{1\leq i,j\leq n_{\alpha}}$, we find
\begin{align*}
  \int_{G} d_{ij}^{\alpha} f d\boldsymbol{\mu} = \frac{1}{\varphi(d')} \sum_{m<d'}S(m,0;d') <d_{ij}^{\alpha}|D_{m}> = 0.
\end{align*}

This proves that $(T(q_0,(p)_k))_{p \in \mathcal{P}(a,k^{\lambda})}$ is $\boldsymbol{\nu}$-uniformly distributed, where $\boldsymbol{\nu}$ does not depend on $a$.
\end{proof}

\subsection{Relation of Fourier terms}
We prove the following relation between very similarly defined terms that appeared when we estimated the Fourier terms.
\phiPsi*
\begin{proof}
  We start by proving
  \begin{align}\label{eq:phi_psi}
    \norm{\phi_{\lambda,\alpha}^{q}(t,r)}_2 \leq \sum_{1\leq j \leq \lambda} \max_{q' \in Q} \norm{\psi_{\lambda-j,\alpha}^{q'}(t,r \bmod k^{\alpha})}_2 \frac{1}{k^{j}} 
	  + \frac{1}{k^{\lambda}}.
  \end{align}
  We rewrite $r = r_1 k^{\alpha + \lambda} + r_2 k^{\alpha} + r_3$ where $r_3 <k^{\alpha}$ and $r_2 < k^{\lambda}$, 
  such that $uk^{\alpha} + r = r_1 k^{\alpha+\lambda} + (u+r_2) k^{\alpha} + r_3$.
	We distinguish the cases $r_1>0$ and $r_1 = 0$ as this determines whether the number of digits of $(u k^{\alpha}+r)$ depends on $u$.
  Let us first consider the case $r_1>0$.
  We find by distinguishing $u+r_2 <k^{\lambda}$ and $u+r_2 \geq k^{\lambda}$
  \begin{align*}
    \phi_{\lambda,\alpha}^{q}(t,r) &= \frac{1}{k^{\lambda}} \sum_{\substack{u<k^{\lambda}\\u+r_2 <k^{\lambda}}} D(T(q,(r_1)_k (u+r_2)_k^{\lambda} (r_3)_k^{\alpha})) \e(-ut)\\
	&\quad + \frac{1}{k^{\lambda}} \sum_{\substack{u<k^{\lambda}\\u+r_2 \geq k^{\lambda}}} D(T(q,(r_1+1)_k (u+r_2 -k^{\lambda})_k^{\lambda} (r_3)_k^{\alpha})) \e(-ut)\\
    &= \frac{1}{k^{\lambda}} D(T(q,(r_1)_k)) \sum_{r_2 \leq u' < k^{\lambda}} D(T(\delta(q,(r_1)_k),(u')_k^{\lambda} (r_3)_k^{\alpha})) \e(-(u'-r_2)t)\\
	&\quad + \frac{1}{k^{\lambda}} D(T(q,(r_1+1)_k)) \sum_{u'<r_2} D(T(\delta(q,(r_1+1)_k),(u')_k^{\lambda} (r_3)_k^{\alpha})) \e(-(u'-r_2+k^{\lambda})t).
  \end{align*}

  This gives
  \begin{align*}
    \norm{\phi_{\lambda,\alpha}^{q}(t,r)}_2 &\leq \max_{q' \in Q} \frac{1}{k^{\lambda}}\norm{\sum_{r_2 \leq u' < k^{\lambda}} D(T(q',(u')_k^{\lambda} (r_3)_k^{\alpha})) \e(-u't)}_2\\
	&\quad + \max_{q'\in Q} \frac{1}{k^{\lambda}} \norm{\sum_{u'<r_2} D(T(q',(u')_k^{\lambda} (r_3)_k^{\alpha})) \e(-u't)}_2.
  \end{align*}
  
  The case $r_1 = 0$ can be treated similarly and we find in this case
  \begin{align}\label{eq:rec_phi_neu}
  \begin{split}
    \norm{\phi_{\lambda,\alpha}^{q}(t,r)}_2 &\leq \max_{q' \in Q} \frac{1}{k^{\lambda}}\norm{\sum_{r_2 \leq u' < k^{\lambda}} D(T(q',(u')_k (r_3)_k^{\alpha})) \e(-u't)}_2\\
	&\quad + \max_{q'\in Q} \frac{1}{k^{\lambda}} \norm{\sum_{u'<r_2} D(T(q',(u')_k^{\lambda} (r_3)_k^{\alpha})) \e(-u't)}_2.
  \end{split}
  \end{align}
  We work from now on just with the case $r_1 = 0$ as the case $r_1 >0$ works similarly.
  We rewrite $r_2 = r_2' k^{\lambda-1} + r_2''$ with $r_2' <k, r_2'' < k^{\lambda-1}$ and find by distinguishing the most significant digits of $u'$
  the following upper bound for the right side of \eqref{eq:rec_phi_neu}:
  \begin{align*}
    &\frac{1}{k} \sum_{r_2' + 1 \leq u_1' < k} \max_{q' \in Q} \frac{1}{k^{\lambda-1}}\norm{\sum_{u_2' < k^{\lambda-1}} 
	  D(T(q',(u_2')_k^{\lambda-1} (r_3)_k^{\alpha})) \e(-u_2't)}_2\\
	&\quad + \frac{1}{k} \max_{q' \in Q} \frac{1}{k^{\lambda-1}}\norm{\sum_{r_2'' \leq u_2' <k^{\lambda-1}} D(T(q',(u_2')_k (r_3)_k^{\alpha})) \e(-u_2't)}_2\\
	&\quad +  \frac{1}{k} \sum_{u_1'<r_2'} \max_{q'\in Q} \frac{1}{k^{\lambda-1}}\norm{\sum_{u_2'<k^{\lambda-1}} D(T(q',(u_2')_k^{\lambda-1} (r_3)_k^{\alpha})) \e(-u_2't)}_2\\
	&\quad + \frac{1}{k} \max_{q' \in Q} \frac{1}{k^{\lambda-1}}\norm{\sum_{u_2'<r_2''} D(T(q',(u_2')_k^{\lambda-1} (r_3)_k^{\alpha})) \e(-u_2't)}_2.
  \end{align*}
  The first and third lines give
  \begin{align*}
    \frac{k-1}{k} \psi_{\lambda-1,\alpha}^{q}(t,r_3)
  \end{align*}
	and we find in total that the right side of \cref{eq:rec_phi_neu} is bounded by
	\begin{align*}
		\frac{k-1}{k} \psi_{\lambda-1,\alpha}^{q}(t,r_3) &+ \frac{1}{k} \max_{q' \in Q} \frac{1}{k^{\lambda-1}}\norm{\sum_{r_2'' \leq u_2' <k^{\lambda-1}} D(T(q',(u_2')_k (r_3)_k^{\alpha})) \e(-u_2't)}_2\\
		&\quad + \frac{1}{k} \max_{q' \in Q} \frac{1}{k^{\lambda-1}}\norm{\sum_{u_2'<r_2''} D(T(q',(u_2')_k^{\lambda-1} (r_3)_k^{\alpha})) \e(-u_2't)}_2.
	\end{align*}
  By applying this step inductively we find for $1\leq \ell \leq \lambda$ the bound
	\begin{align*}
		\sum_{1\leq j \leq \ell} \frac{k-1}{k^j} \psi_{\lambda-j,\alpha}^{q}(t,r_3) &+ \frac{1}{k^{\ell}} \max_{q' \in Q} \frac{1}{k^{\lambda-\ell}}\norm{\sum_{r_2^{(\ell)} \leq u_2' <k^{\lambda-\ell}} D(T(q',(u_2')_k (r_3)_k^{\alpha})) \e(-u_2't)}_2\\
		&\quad + \frac{1}{k^{\ell}} \max_{q' \in Q} \frac{1}{k^{\lambda-\ell}}\norm{\sum_{u_2'<r_2^{(\ell)}} D(T(q',(u_2')_k^{\lambda-\ell} (r_3)_k^{\alpha})) \e(-u_2't)}_2.
	\end{align*}
	\eqref{eq:phi_psi} is obtained by taking $\ell = \lambda$.
  Thus it just remains to use the bound of $\psi_{\lambda,\alpha}^{q}(t,r)$
  \begin{align*}
    \norm{\phi_{\lambda,\alpha}^{q}(t,r)}_2 &\leq \sum_{1\leq j \leq \lambda} \max_{q' \in Q} \norm{\psi_{\lambda-j,\alpha}^{q'}(t,r \bmod k^{\alpha})}_2 \frac{1}{k^{j}} + \frac{1}{k^{\lambda}}\\
	&\leq c \sum_{j\leq\lambda}k^{-(\lambda-j)\eta} k^{-j} + k^{-\lambda}\\
	&\leq c k^{-\lambda \eta} \frac{k^{-(1-\eta)}}{1-k^{-(1-\eta)}} + k^{-\lambda}.
  \end{align*}

\end{proof}

\subsection{Proof of \cref{thm:mobius} and \cref{thm:prime}}\label{sec:RS2}
 The proof of \cref{thm:prime} and \cref{thm:mobius} is completely analogous to the corresponding proof in \cite{mauduit_rivat_rs} and relies on 
 bounds for sums of type I and sums of type II.
In this part of the paper we will comment on how to adapt the original proof to our situation.
We will only describe the important changes briefly (and assume that the reader is familiar with \cite{mauduit_rivat_rs}).
First we comment on how to adapt the auxiliary results of \cite{mauduit_rivat_rs}.

For a generalization of \cite[Lemma 3]{mauduit_rivat_rs} see \cite{invertible}.
\cite[Lemma 6]{mauduit_rivat_rs} as well as the Cauchy-Schwarz inequality can be easily adapted to matrix-valued functions, where we use the Frobenius norm instead of the absolute value.
These results change at most by a factor $\sqrt{d}$.\\
\cite[Lemma 8]{mauduit_rivat_rs} can be adapted to our case where the proof stays unchanged:
\begin{lemma}\label{le:cut_1}
  Let $f:\N\rightarrow U_d$ satisfy \cref{def:1} with $\eta >0$. For $(\mu,\nu,\rho)\in\N^3$ with $2\rho < \nu$ the set $\mathcal{E}$ of pairs 
  $(m,n)\in \{k^{\mu-1},\ldots,k^{\mu}-1\} \times \{k^{\nu-1},\ldots,k^{\nu}-1\}$ such that there exists $\ell <k^{\mu+\rho}$ with 
  $f(mn+k)^{H}f(mn) \neq f_{\mu+2\rho}(mn+k)^{H}f_{\mu+2\rho}(mn)$ satisfies
  \begin{align}
    \card \mathcal{E} \ll (\log k) k^{\mu + \nu - \eta \rho}.
  \end{align}
\end{lemma}
For $a \in \Z$ and $\kappa \in \N$ we denote by $\mathbf{r}_{\kappa}(a)$ the unique integer $r \in \{0,\ldots,k^{\kappa}-1\}$ such that $a \equiv r \bmod k^{\kappa}$.
More generally for integers $0 \leq \kappa_1 \leq \kappa_2$ we denote by $\mathbf{r}_{\kappa_1,\kappa_2}(a)$ the unique integer $r \in \{0,\ldots,k^{\kappa_2-\kappa_1}-1\}$ such that $a = s k^{\kappa_2} + r k^{\kappa_1} + t$ for some $t \in \{0,\ldots,k^{\kappa_1}-1\}$ and $s \in \Z$.
\cite[Lemma 9]{mauduit_rivat_rs} can be adapted as well:
\begin{lemma}
  Let $f:\N \rightarrow U_d$ satisfying \cref{def:1} and $(\mu, \nu, \mu_0,\mu_1,\mu_2)\in \N^5$ with $\mu_0\leq \mu_1\leq \mu \leq \mu_2$, 
  $\mu \leq \nu$ and $2(\mu_2-\mu)\leq \mu_0$. For $(a,b,c) \in \N^3$ the set $\mathcal{E}(a,b,c)$ of pairs $(m,n) \in \{k^{\mu-1},\ldots,k^{\mu}-1\}\times\{k^{\nu-1},\ldots,k^{\nu}-1\}$
  such that
	{\small
  \begin{align*}
    &f_{\mu_2}(mn+am+bn+c)^{H} f_{\mu_2}(k^{\mu_0} \mathbf{r}_{\mu_0,\mu_2}(mn+am+bn+c))\\
    &\qquad \neq f_{\mu_1}(mn+am+bn+c)^{H} f_{\mu_1}(k^{\mu_0} \mathbf{r}_{\mu_0,\mu_2}(mn+am+bn+c))
  \end{align*}
	}%
  satisfies 
  \begin{align}
    \card \mathcal{E}(a,b,c) \ll \max(\tau(k),\log k) \mu_2^{\omega(k)} k^{\mu + \nu + \eta(\mu_0-\mu_1)}.
  \end{align}

\end{lemma}

\subsubsection{Sums of Type I}
We take a non-decreasing function $\gamma: \R \rightarrow \R$ satisfying $\lim_{\lambda \to \infty} \gamma(\lambda) = + \infty$, $c\geq 2$ and 
$f: \N \to U_d$ be a function satisfying \cref{def:1} and $f \in \mathcal{F}_{\gamma, c}$ according to \cref{def:2}.
Let
\begin{align}\label{eq:MN_relation}
  1 \leq M \leq N \text{ such that } M\leq(MN)^{1/3}.
\end{align}
Let $\mu$ and $\nu$ be the unique integers such that
\begin{align*}
  k^{\mu-1} \leq M < k^{\mu} \text{ and } k^{\nu-1} \leq N < k^{\nu}.
\end{align*}

Let $\theta \in \R$, an interval $I(M,N) \subseteq[0,MN]$ and 
\begin{align}
  S_{I}(\theta) = \sum_{\frac{M}{q} < m \leq M} \norm{\sum_{n: mn \in I(M,N)} f(mn)\e(\theta mn)}_F.
\end{align}

\begin{proposition}\label{pr:1}
  Assuming \eqref{eq:MN_relation} and with $c\geq 2$, we have -- uniformly for $\theta \in \R$ -- that
  \begin{align}
    S_{I}(\theta) \ll (\log k)^{5/2} (\mu + \nu)^{2} k^{\mu + \nu - \frac{\eta}{2} \gamma\rb{\frac{\mu + \nu}{3}}}.
  \end{align}
\end{proposition}
\begin{proof}[Proof (Sketch):]
	The proof proceeds as in \cite{mauduit_rivat_rs} up to \cite[equation (32)]{mauduit_rivat_rs}.
	We find
	\begin{align*}
		\widehat{f_{\mu+\nu}}(t) &= \frac{1}{k^{\mu + \nu - \alpha}} \sum_{v < k^{\mu + \nu - \alpha}} f(v k^{\alpha}) \e\rb{-\frac{vt}{k^{\mu + \nu-\alpha}}}\\
			&\qquad \frac{1}{k^{\alpha}}\sum_{u\leq k^{\alpha}} f(v k^{\alpha})^{H} f(u + v k^{\alpha}) \e\rb{-\frac{ut}{k^{\mu + \nu}}}.
	\end{align*}
	Here we needed to ensure the correct order of the terms.
	From now on the proof does not change and we just have to take the factor $\eta$ into account when using \cref{def:1}.
	We find for example
	\begin{align*}
		\card\widetilde{\mathcal{W}_{\alpha}} \ll k^{\mu + \nu - \eta \rho_1}.
	\end{align*}
	Thereafter one only needs to keep the order of terms as the values of $f$ are now matrices and thus the order of terms is important.
	However, this does not change any important properties and all arguments of \cite{mauduit_rivat_rs} still hold.
	By taking $\eta$ into account, we find
	{\small
	\begin{align*}
		\norm{S_{I,2}''(M,d)}_F\ll (\log k)^{1/2} k^{\mu + \nu - \eta \rho_1/2}\\
		S_{I,2}'(\theta') \ll (\log k)^{1/2} \sum_{1\leq d \leq M} \frac{k^{-\eta \rho_1/2}}{d} \ll \mu (\log k)^{3/2} k^{-\eta \rho_1/2}.
	\end{align*}
	}%
	Choosing
	\begin{align*}
		\rho_1 = \gamma\rb{\frac{\mu + \nu}{3}} \frac{2}{1+\eta}
	\end{align*}
	then gives the desired result.
\end{proof}

\subsubsection{Sums of Type II}
We take $\gamma: \R\to\R$ a non-decreasing function satisfying $\lim_{\lambda\to\infty}\gamma(\lambda) = +\infty, c \geq 10$ and $f:\N \to U_d$ be a function satisfying 
\cref{def:1} and $f\in\mathcal{F}_{\gamma,c}$ in \cref{def:2}. Let $M,N$ such that $1\leq M\leq N$. We let $\mu$ and $\nu$ denote the unique integers such that
\begin{align*}
  k^{\mu-1}\leq M < k^{\mu} \text{ and } k^{\nu-1} \leq N < k^{\nu}.
\end{align*}
Let us assume that
\begin{align}\label{eq:mu_nu_2}
  \frac{1}{4}(\mu + \nu) \leq \mu \leq \nu \leq \frac{3}{4}(\mu + \nu).
\end{align}
We assume -- as in \cite{mauduit_rivat_rs} -- that the multiplicative dependency of the variables in the type II sums has been removed by the classical method described (for example) in 
\cite[Section 5]{mauduit_rivat_2010}.
Let $\theta\in\R, a_m \in \C, b_n \in \C$ with $\abs{a_m}\leq 1, \abs{b_n}\leq 1$ and
\begin{align*}
  S_{II}(\theta) = \sum_{m}\sum_{n} a_m b_n f(mn) \e(\theta mn)
\end{align*}
where we sum over $m \in (M/k,M]$ and $n \in (N/k,N]$. We will prove
\begin{proposition}\label{pr:2}
  Assuming \eqref{eq:mu_nu_2} and $c\geq 10$, uniformly for $\abs{a_m}\leq 1,\abs{b_n}\leq 1$ and $\theta\in\R$, we have
  \begin{align}
    \norm{S_{II}(\theta)}_F \ll \max(\tau(k) \log k, \log ^3 k)^{1/4}(\mu + \nu)^{\frac{1}{4}(1+\max(\omega(k),2))} k^{\mu + \nu - \frac{\eta}{20}\gamma(2\floor{\mu/15})}.
  \end{align}
\end{proposition}
\begin{proof}
The proof of the corresponding result in \cite{mauduit_rivat_rs} is the most difficult part -- quite long and complicated.
We will try to focus only on the necessary changes and keep it as short as possible, as no important new ideas are needed.\\
We find by using the corresponding results:
\begin{align*}
  \norm{S_{II}(\theta)}_{F}^2 \ll \frac{M^2N^2}{R} + \frac{MN}{R} \sum_{1\leq r<R}\rb{1-\frac{r}{R}} \tr(S_1(r))
\end{align*}
with
\begin{align*}
  S_1(r) = \sum_{m}\sum_{n\in I(N,r)} b_{n+r}\overline{b_n} f(mn + mr) f(mn)^{H} \e(\theta mr),
\end{align*}
where $I(N,r) = (N/k,N-r]$. Let
\begin{align}
  \mu_2 = \mu + 2 \rho.
\end{align}
As we are only interested in $\tr(S_1(r))$, we can -- by well-known properties of the trace -- exchange the order of the matrices
\begin{align}
  \tr(S_1(r)) &= \sum_{m}\sum_{n\in I(N,r)} b_{n+r} \overline{b_n} \tr(f(mn+mr) f(mn)^{H}) \e(\theta mr)\\
    &= \sum_{m}\sum_{n\in I(N,r)} b_{n+r} \overline{b_n} \tr(f(mn)^{H} f(mn+mr)) \e(\theta mr)\\
    &= \tr\rb{\sum_{m} \sum_{n\in I(N,r)} b_{n+r} \overline{b_n} f(mn)^{H} f(mn+mr) \e(\theta mr)}.
\end{align}
We denote from now on by $\widetilde{S}$ the corresponding sum where we exchange the order of the matrices, e.g.~:
\begin{align}
  \widetilde{S}_1(r) = \sum_{m}\sum_{n\in I(N,r)} b_{n+r}\overline{b_n} f(mn)^{H} f(mn + mr) \e(\theta mr).
\end{align}
If $f$ satisfies the carry property described in \cref{def:1}, then by \cref{le:cut_1} the number of pairs $(m,n)$ for which
$f(mn)^{H} f(mn+mr) \neq f_{\mu_2}(mn)^{H} f_{\mu_2}(mn+mr)$ is $O(k^{\mu + \nu - \eta \rho})$. Hence
\begin{align}
  \tr(S_1(r)) = \tr(\widetilde{S}_1(r)) = \tr(S'_1(r)) + O(k^{\mu + \nu - \eta\rho})
\end{align}
where
\begin{align*}
  S'_1(r) &= \sum_{m}\sum_{n\in I(N,r)} b_{n+r}\overline{b_n} f_{\mu_2}(mn)^{H} f_{\mu_2}(mn+mr) \e(\theta mr).
\end{align*}
Note that it was important to change the order of matrices first to apply \cref{le:cut_1}.
Using again the Cauchy-Schwarz inequality for the summation over $r$ and $\abs{\tr(A)} \leq \norm{A}_F \sqrt{d}$, leads to
\begin{align}
  \norm{S_{II}(\theta)}_F^{4} \ll \frac{M^4N^4}{R^2}+\frac{M^2N^2}{R^2}R\sum_{1\leq r < R}\norm{S_1'(r)}_F^2.
\end{align}
When applying the Van-der-Corput inequality for the summation over $m$ (see for example \cite{invertible})
we need again to keep the correct order of terms. We find
\begin{align*}
  \sum_{1\leq r < R} \norm{S_1'(r)}_F^2 \ll \frac{M^2N^2R}{S} + \frac{MN}{S} \tr(\widetilde{S_2})
\end{align*}
with
\begin{align*}
  \widetilde{S_2} = \sum_{1\leq r<R}\sum_{1\leq s < S} \rb{1-\frac{s}{S}} \e(\theta k^{\mu_1} rs) \widetilde{S'_2}(r,s)
\end{align*}
where
\begin{align*}
  \widetilde{S'_2}(r,s) &= \sum_{m}\sum_{n} f_{\mu_2}((m+sk^{\mu_1})n)^{H} f_{\mu_2}((m+sk^{\mu_1})(n+r))f_{\mu_2}(m(n+r))^{H} f_{\mu_2}(mn)\\
      &= \sum_{m}\sum_{n} f_{\mu_1}((m+sk^{\mu_1})n)^{H} f_{\mu_1,\mu_2}((m+sk^{\mu_1})n)^{H} f_{\mu_1,\mu_2}((m+sk^{\mu_1})(n+r))\\
      &\qquad f_{\mu_1}((m+sk^{\mu_1})(n+r)) f_{\mu_1}(m(n+r))^{H} f_{\mu_1, \mu_2}(m(n+r))^{H} f_{\mu_1,\mu_2}(mn)f_{\mu_1}(mn).
\end{align*}
However, we find that $\tr(\widetilde{S'_2}(r,s)) = \tr(S'_2(r,s))$ for
{\small
\begin{align*}
  S'_2(r,s) = \sum_{m}\sum_{n} f_{\mu_1,\mu_2}(mn) f_{\mu_1,\mu_2}((m+sk^{\mu_1})n)^{H} f_{\mu_1,\mu_2}((m+sk^{\mu_1})(n+r)) f_{\mu_1, \mu_2}(m(n+r))^{H} 
\end{align*}
}%
and, therefore,
\begin{align*}
  \sum_{1\leq r < R} \norm{S_1'(r)}_F^2 \ll \frac{M^2N^2R}{S} + \frac{MN}{S} \tr(S_2)
\end{align*}
where
\begin{align*}
  S_2 = \sum_{1\leq r<R}\sum_{1\leq s < S} \rb{1-\frac{s}{S}} \e(\theta k^{\mu_1} rs) S'_2(r,s).
\end{align*}

We choose the order of the terms so that we can use the Cauchy-Schwarz inequality efficiently.
Whenever we use \cref{def:1} (at least implicitly), we find a different error term, e.g.,
instead of \cite[(60)]{mauduit_rivat_rs} we can use
\begin{align*}
  \card \mathcal{E}_{\mu_0,\mu_1,\mu_2}(r,s) \ll \max(\tau(k),\log k) (\mu+\nu)^{\omega(k)} k^{\mu + \nu - 2 \eta \rho'}.
\end{align*}
Thus we find
\begin{align*}
  S_2'(r,s) = S_3(r,s) + O(\max(\tau(k),\log k)(\mu + \nu)^{\omega(k)} k^{\mu + \nu-2\eta\rho'}),
\end{align*}
where the order of the terms in $S_3$ has to be changed:
{\small
\begin{align*}
  S_3(r,s) &= \sum_{m}\sum_{n} \sum_{0\leq u_0 < k^{\mu_2-\mu_0}} \sum_{0\leq u_1<k^{\mu_2-\mu_0}} \chi_{k^{\mu_0-\mu_2}}\rb{\frac{mn}{k^{\mu_2}}-\frac{u_0}{k^{\mu_2-\mu_0}}}
		      \chi_{k^{\mu_0-\mu_2}}\rb{\frac{mn+mr}{k^{\mu_2}} - \frac{u_1}{k^{\mu_2-\mu_0}}}\\
	      &\qquad g(u_0)g(u_0 + k^{\mu_1-\mu_0}sn)^{H}g(u_1+k^{\mu_1-\mu_0}sn+k^{\mu_1-\mu_0}sr)g(u_1)^{H}.
\end{align*}
}%

This impact of $\eta$ carries through the rest of the work and we will only comment on the specific form of some intermediate results, e.g.,~ $S_4$:
{\small
\begin{align*}
  S_4(r,s) &= k^{2(\mu_2-\mu_0)} \sum_{\abs{h_0}\leq H} \sum_{\abs{h_1}\leq H} a_{h_0}(k^{\mu_0-\mu_2},H) a_{h_1}(k^{\mu_0-\mu_2},H) 
    \sum_{0\leq h_2 < k^{\mu_2-\mu_0}} \sum_{0 \leq h_3 < k^{\mu_2-\mu_0}}\e\rb{\frac{h_3sr}{k^{\mu_2-\mu_1}}}\\
    &\qquad \widehat{g}(h_0-h_2) \widehat{g}(-h_2)^{H} \widehat{g}(h_3) \widehat{g}(h_3-h_1)^{H}\\
    &\qquad \sum_{m}\sum_{n} \e\rb{\frac{(h_0+h_1) mn + h_1 mr + (h_2+h_3)k^{\mu_1}sn}{k^{\mu_2}}}.
\end{align*}
}%
The definitions of $S_6$ and $S_7$ have to be adapted as well, e.g.~:
\begin{align*}
  S_7(h_1) = \sum_{0\leq h'<k^{\mu_2-\mu_0}} \norm{\widehat{g}(h'-h_1) \widehat{g}(h')^{H}}_F^2.
\end{align*}
The following arguments of \cite{mauduit_rivat_rs} carry over to this generalization.

We find the following lemma, which is the analogue to \cite[Lemma 10]{mauduit_rivat_rs}.
\begin{lemma}
  If
  \begin{align*}
    \mu \leq \rb{2+\frac{4}{3}c}\rho
  \end{align*}
  then we have -- uniformly for $\lambda \in \N$ with $\frac{1}{3}(\mu_2-\mu_0) \leq \lambda \leq \frac{4}{5} (\mu_2-\mu_0)$ --
  \begin{align*}
    \sum_{0\leq h<k^{\mu_2-\mu_0}} \sum_{0\leq k <k^{\mu_2-\mu_0-\lambda}}\norm{\widehat{g}(h+k)\widehat{g}(h)^{H}}_F^2 \ll k^{-\gamma_1(\lambda,\mu_1-\mu_0)}(\log k^{\mu_2-\mu_1})^2
  \end{align*}
  where
  \begin{align}
    \gamma_1(\lambda,\mu_1-\mu_0) = \frac{\gamma(\lambda) - \mu_1+\mu_0}{2} \eta.
  \end{align}
\end{lemma}
\begin{proof}
  The proof works as in \cite{mauduit_rivat_rs}, where one has to be careful to keep the order of terms.
  However, this does not change the proof substantially and by using the new estimate given by \cref{def:1} one finds the desired result.
\end{proof}

The rest of the proof does not change and it just remains to balance the error terms differently. One finds in total, uniformly for $\theta \in \R$
\begin{align*}
  \norm{S_{II}(\theta)}_F^4 &\ll k^{4\mu + 4\nu +\mu_1-\mu_0}(k^{-\gamma_1(\mu_2-\mu_0-2\rho,\mu_1-\mu_0)}+k^{-\rho} \log k^{\rho})\\
  &\qquad\qquad (\tau\rb{k^{\mu_2-\mu_1}}+k^{\mu_2-\mu_1-\nu \log k^{\mu_2-\mu_1}})\\
  &\quad + (\log k)^3(\mu+\nu)^3 k^{4\mu+4\nu+3(\mu_2-\mu_0)+2\rho}(k^{-\mu_2}+k^{-\nu})\\
  &\quad + \max(\log k^{\mu_0},\tau(k^{\mu_0})) k^{4\mu+4\nu-2\rho}\\
  &\quad + \max(\tau(k),\log k) (\mu+\nu)^{\omega(k)} k^{4\mu+4\nu-2\eta \rho'}.
\end{align*}
Note that only the first and last error terms have changed compared to \cite{mauduit_rivat_rs}.
As in \cite{mauduit_rivat_rs}, we assume
\begin{align*}
  \mu_2 = \mu + 2\rho\\
  \mu_1 = \mu - 2\rho\\
  \mu_0 = \mu_1 - 2 \rho'\\
  \mu \leq \nu \leq 3\mu.
\end{align*}
In total, we find
\begin{align*}
  \norm{S_{II}(\theta)}_F^4 &\ll \tau(k)(\mu_2-\mu_1)^{\omega(k)} k^{4\mu+4\nu + \frac{2+\eta}{2}(\mu_1-\mu_0) -\frac{\eta}{2} \gamma(2\rho)} \log k^{\rho}\\
  &\quad + (\log k)^3(\mu + \nu)^3 k^{4\mu + 4\nu + 3(\mu_1-\mu_0) + 14\rho-\mu}\\
  &\quad + \max(\log k^{\mu_0},\tau(k^{\mu_0})) k^{4\mu+4\nu-2\rho}\\
  &\quad + \max(\tau(k),\log k) (\mu+\nu)^{\omega(k)} k^{4\mu+4\nu-2\eta \rho'}.
\end{align*}

Taking
\begin{align*}
  \rho' = \floor{\eta \gamma(2\rho)/10},
\end{align*}
we have $\mu_1-\mu_0 = 2\rho' \leq \eta \gamma(2\rho)/5 \leq \eta \rho/5$, and thus
\begin{align*}
  \frac{2+\eta}{2}(\mu_1-\mu_0) - \frac{\eta}{2}\gamma(2\rho) \leq \frac{3}{10} \eta \gamma(2\rho) - \frac{\eta}{2} \gamma(2\rho) = -\eta\frac{\gamma(2\rho)}{5}.
\end{align*}
Furthermore, we choose
\begin{align*}
  \rho = \floor{\mu/15},
\end{align*}
which yields by the same arguments as in \cite{mauduit_rivat_rs} that -- finally --
\begin{align*}
  \norm{S_{II}(\theta)}_F^4 \ll \max(\tau(k)\log k, \log ^3 k)(\mu + \nu)^{1+\max(\omega(k),2)} k^{4\mu + 4\nu -\eta \gamma(2\floor{\mu/15})/5},
\end{align*}
which completes the proof of \cref{pr:2}.
\end{proof}

\begin{proof}{Proof of \cref{th:moebius} and \cref{th:prime}:}
  The proof is completely analogous to the corresponding proof in \cite{mauduit_rivat_rs}.
\end{proof}

 \bibliographystyle{abbrv}
 \bibliography{bibliography}
 
\end{document}